\documentclass[a4paper, 12pt]{amsart}

\makeatletter
\@namedef{subjclassname@2020}{\textup{2020} Mathematics Subject Classification}
\makeatother

\usepackage{manfnt} 

\usepackage[all,cmtip,color]{xy} 
\usepackage{float}

\usepackage{dsfont}

\usepackage[dvipsnames]{xcolor}
\usepackage
[a4paper, margin=1.3in]{geometry}

\usepackage[textsize=footnotesize]{todonotes}

\usepackage{comment}

\usepackage{amsmath,amsthm,amssymb,mathrsfs} 
\usepackage{aliascnt}
\usepackage{lmodern}
\usepackage[T1]{fontenc}
\usepackage{stmaryrd} 
\usepackage{url} 
\usepackage{hyperref}
\definecolor{dblue}{rgb}{0,0,0.70}
\hypersetup{
	unicode=true,
	colorlinks=true,
	citecolor=dblue,
	linkcolor=dblue,
	anchorcolor=dblue, 
	urlcolor   = dblue
}

\makeatletter
\renewcommand{\tocsection}[3]{%
  \indentlabel{\@ifnotempty{#2}{\bfseries\ignorespaces#1 #2\quad}}\bfseries#3}
\renewcommand{\tocsubsection}[3]{%
  \indentlabel{\@ifnotempty{#2}{\ignorespaces#1 #2\quad}}#3}
\renewcommand{\tocsubsubsection}[3]{%
  \indentlabel{\hspace{30pt}\@ifnotempty{#2}{\ignorespaces#1 #2\quad}}#3}

\newcommand\@dotsep{4.5}
\def\@tocline#1#2#3#4#5#6#7{\relax
  \ifnum #1>\c@tocdepth 
  \else
    \par \addpenalty\@secpenalty\addvspace{#2}%
    \begingroup \hyphenpenalty\@M
    \@ifempty{#4}{%
      \@tempdima\csname r@tocindent\number#1\endcsname\relax
    }{%
      \@tempdima#4\relax
    }%
    \parindent\z@ \leftskip#3\relax \advance\leftskip\@tempdima\relax
    \rightskip\@pnumwidth plus1em \parfillskip-\@pnumwidth
    #5\leavevmode\hskip-\@tempdima{#6}\nobreak
    \leaders\hbox{$\m@th\mkern \@dotsep mu\hbox{.}\mkern \@dotsep mu$}\hfill
    \nobreak
    \hbox to\@pnumwidth{\@tocpagenum{\ifnum#1=1\bfseries\fi#7}}\par
    \nobreak
    \endgroup
  \fi}
\AtBeginDocument{%
\expandafter\renewcommand\csname r@tocindent0\endcsname{0pt}
}
\def\l@subsection{\@tocline{2}{0pt}{2.5pc}{5pc}{}}
\makeatother

\makeatletter
\expandafter\g@addto@macro\csname th@plain\endcsname{%
		\thm@notefont{\bfseries}
	}%
\expandafter\g@addto@macro\csname th@remark\endcsname{%
		\thm@headfont{\bfseries}
	}%
\makeatother

\newtheorem{theorem}{Theorem}[section]	
\newtheorem*{theorem*}{Theorem}

\newaliascnt{lemma}{theorem}
\newtheorem{lemma}[lemma]{Lemma}
\newtheorem{claim}[theorem]{Claim}
\aliascntresetthe{lemma}
\newtheorem*{lemma*}{Lemma}

\newaliascnt{proposition}{theorem}
\newtheorem{proposition}[proposition]{Proposition}
\aliascntresetthe{proposition}

\newaliascnt{corollary}{theorem}
\newtheorem{corollary}[corollary]{Corollary}
\aliascntresetthe{corollary}

\newtheorem*{subclaim*}{Subclaim}

\theoremstyle{remark}

\newaliascnt{remark}{theorem}
\newtheorem{remark}[remark]{Remark}
\aliascntresetthe{remark}

\theoremstyle{theorem}

\newaliascnt{question}{theorem}
\newtheorem{question}[question]{Question}
\aliascntresetthe{question}
\newtheorem{notation}[theorem]{Notation}

\newtheorem*{question*}{Question}

\newaliascnt{definition}{theorem}
\newtheorem{definition}[definition]{Definition}
\aliascntresetthe{definition}

\newaliascnt{example}{theorem}
\newtheorem{example}[example]{Example}
\aliascntresetthe{example}

\newcommand{\axiomft}[1]{\mathsf{#1}} 
\newcommand{\ZFC}{\axiomft{ZFC}}

\newcommand{\CH}{\axiomft{CH}}

\newcommand{\Ord}{\mathrm{Ord}}

\newcommand{\PFA}{\axiomft{PFA}}

\newcommand{\MA}{\axiomft{MA}}

\newcommand{\cof}{\mathrm{cof}}
\newcommand{\Lev}{\mathrm{Lev}}

\DeclareMathOperator{\dom}{dom}

\DeclareMathOperator{\rank}{rank}

\DeclareMathOperator{\Add}{Add}
\DeclareMathOperator{\crit}{crit}

\newcommand{\forces}{\mathrel{\Vdash}}

\newcommand{\PP}{\mathbb{P}}

\newcommand{\QQ}{\mathbb{Q}}

\newcommand{\BB}{\mathbb B}

\newcommand{\ran}{\mathrm{ran}}

\newcommand{\Tr}{\mathrm{Tr}}

\newcommand{\FA}{\mathsf{FA}}
\newcommand{\BFA}{\mathsf{BFA}}
\newcommand{\NP}{\mathsf{N}}
\newcommand{\club}{\mathsf{club}\text{-}}
\newcommand{\stat}{\mathsf{stat}\text{-}}
\newcommand{\ub}{\mathsf{ub}\text{-}}

\newcommand{\fo}{\Sigma_0^{\mathrm{(sim)}}\text{-}}

\newcommand{\sforces}{\Vdash^+} 

\newcommand{\pow}{\mathcal{P}}

\newcommand{\BN}{\mathsf{BN}}

\usepackage{enumitem}

\newenvironment{enumerate-(a)}{\begin{enumerate}[label={\upshape (\alph*)}, leftmargin=2pc]}{\end{enumerate}}
\newenvironment{enumerate-(1)}{\begin{enumerate}[label={\upshape (\arabic*)}, leftmargin=2pc]}{\end{enumerate}}
\newenvironment{enumerate-(i)}{\begin{enumerate}[label={\upshape (\roman*)}, leftmargin=2pc]}{\end{enumerate}}
\newenvironment{itemizenew}{\begin{itemize}[leftmargin=2pc]}{\end{itemize}}

\author{Philipp Schlicht}
\address[Philipp Schlicht]{School of Mathematics, 
University of Bristol, 
Fry Building.
Woodland Road, 
Bristol, BS8~1UG, UK}
\email[Philipp Schlicht]{philipp.schlicht@bristol.ac.uk}

\author{Christopher Turner}
\address[Christopher Turner]{School of Mathematics, 
University of Bristol, 
Fry Building.
Woodland Road, 
Bristol, BS8~1UG, UK}
\email[Christopher Turner]{christopher.turner@bristol.ac.uk}

\date{\today}

\title{Forcing axioms via ground model interpretations} 

\begin{document}
\begin{abstract} 
We study principles of the form: if a name $\sigma$ is forced to have a certain property $\varphi$, then there is a ground model filter $g$ such that $\sigma^g$ satisfies $\varphi$. 
We prove a general correspondence connecting these name principles to forcing axioms. 
Special cases of the main theorem are:  
\begin{itemize} 
\item 
Any forcing axiom can be expressed as a name principle. For instance, $\PFA$ is equivalent to: 
\begin{itemizenew} 
\item[\mantriangleright] 
A principle for rank $1$ names (equivalently, nice names) for subsets of $\omega_1$. 
\item[\mantriangleright] 
A principle for rank $2$ names for sets of reals. 
\end{itemizenew} 
\item 
$\lambda$-bounded forcing axioms are equivalent to name principles. 
Bagaria's characterisation of $\BFA$ via generic absoluteness is a corollary. 
\end{itemize} 

We further systematically study name principles where $\varphi$ is a notion of largeness 
for subsets of $\omega_1$ (such as being unbounded, stationary or in the club filter) and corresponding forcing axioms. 
\end{abstract}

\maketitle

\setcounter{tocdepth}{3}
\tableofcontents

\section{Introduction} 

In this paper, we isolate and study \emph{name principles}. They express that for names $\sigma$ such that a certain property is forced for $\sigma$, there exists a filter $g$ in the ground model $V$ such that $\sigma^g$ already has this property in $V$. 
In general, we fix a class $\Sigma$ of names, for example nice names for sets of ordinals. 
Given a forcing $\PP$ and a formula $\varphi(x)$, one can then study the principle: 
\begin{quote}
\emph{``If $\sigma\in \Sigma$ and $\PP\Vdash \varphi(\sigma)$ holds, then there exists a filter $g\in V$ on $\PP$ such that $\varphi(\sigma^g)$ holds in $V$.''} 
\end{quote} 
Such principles are closely related to Bagaria's work on generic absoluteness and forcing axioms \cite{bagaria2000bounded}. 
Recall that the forcing axiom $\FA_{\PP,\kappa}$ associated to a forcing $\PP$ and an uncountable cardinal $\kappa$ states: 
\begin{quote} 
\emph{``For any sequence $\vec{D}=\langle D_\alpha \mid \alpha<\kappa\rangle$ of predense subsets of $\PP$, there is a filter $g\in V$ on $\PP$ such that $g\cap D_\alpha \neq\emptyset$ for all $\alpha<\kappa$.''} 
\end{quote} 
Often, proofs from forcing axioms can be formulated by first proving a name principle and then obtaining the desired result as an application. 

\begin{example} 
$\FA_{\PP,\omega_1}$ implies that for any stationary subset $S$ of $\omega_1$, $\PP$ does not force that $S$ is nonstationary. 
We sketch an argument via a name principle. 
We shall show in Section \ref{Section correspondence and applications} that 
$\FA_{\PP,\omega_1}$ implies the name principle for any nice name $\tau$ for a subset of $\omega_1$ and any $\Sigma_0$-formula $\varphi$.  
So towards a contradiction, suppose there is a name $\tau$ for a club with $\Vdash_\PP \tau\cap S=\emptyset$. 
Apply the name principle for the formula ``$\tau$ is a club in $\omega_1$ and $\tau\cap \check{S}
=\emptyset$''. 
Hence there is a filter $g\in V$ such that $\tau^g$ is a club and $\tau^g\cap S=\emptyset$. 
However, the existence of $\tau^g$ contradicts the assumption that $S$ is stationary. 
\end{example} 

Name principles for stationary sets have appeared implicitly in combination with forcing axioms. 

\begin{example} 
The forcing axiom $\PFA^+$ states: 
\begin{quote} 
For any proper forcing $\PP$, any sequence $\vec{D}=\langle D_\alpha \mid \alpha<\omega_1\rangle$ of predense subsets of $\PP$ and any nice name $\sigma$ for a stationary subset of $\omega_1$, there is a filter $g$ on $\PP$ such that 
\begin{itemize} 
\item 
$g\cap D_\alpha \neq\emptyset$ for all $\alpha<\omega_1$ and 
\item 
$\sigma^g$ is stationary. 
\end{itemize} 
\end{quote} 
Thus $\PFA^+$ is a combination of the forcing axiom $\PFA$ with a name principle for stationary sets. 
Note that the formula ``$\sigma$ is stationary'' is not $\Sigma_0$. 
\end{example} 

We aim for an analysis of name principles for their own sake. 
The main result of this paper is that name principles are more general than forcing axioms. 
In other words, all known forcing axioms can be reformulated as name principles. 
For instance, we have: 

\begin{theorem}
\label{special case of main theorem}  (see Theorem \ref{correspondence forcing axioms name principles}\footnote{This follows from Theorem \ref{correspondence forcing axioms name principles} \ref{correspondence forcing axioms name principles 2} for $X=\kappa$ and $\alpha=1$.})   
	Suppose that $\PP$ is a forcing and $\kappa$ is a cardinal. 
	Then the following statements are equivalent: 
	\begin{enumerate-(1)} 
	\item 
	$\FA_{\PP,\kappa}$ 
	\item 
	The name principle $\NP_{\PP,\kappa}$ for nice names $\sigma$ and the formula $\sigma=\check{\kappa}$. 
	\item 
	The	 simultaneous name principle $\fo\NP_{\PP,\kappa}$ for nice names $\sigma$ and all first-order formulas over the structure $(\kappa,\in,\sigma)$. 
	\end{enumerate-(1)} 
\end{theorem}
The main Theorems \ref{correspondence forcing axioms name principles} and \ref{correspondence bounded forcing axioms name principles} are more general and cover: 
(i) arbitrary names instead of nice names and 
(ii) bounded forcing axioms. 

Bagaria proved an equivalence between bounded forcing axioms and generic absoluteness principles \cite{bagaria1997characterization, bagaria2000bounded}. 
The following corollary of \ref{correspondence bounded forcing axioms name principles} has Bagaria's result as a special case. 
Here $\BFA_{\PP,\kappa}$ denotes the usual bounded forcing axiom, i.e. for $\kappa$ many predense sets of size at most $\kappa$, and the principle in \ref{Bagaria's characterisation 2c} denotes the name principle for names of the form 
$ \{ (\check{\alpha},p_\alpha)\mid \alpha\in \kappa \}$ and for all $\Sigma_0$-formulas simultaneously.

\begin{theorem}
\label{Theorem intro Bagaria} 
(see Theorems \ref{Bagaria's characterisation} and \ref{variant of Bagaria's characterisation for countable cofinality}) 
Suppose that $\kappa$ is an uncountable cardinal, $\PP$ is a complete Boolean algebra and $\dot{G}$ is a $\PP$-name for the generic filter. 
The following conditions are equivalent: 
\begin{enumerate-(1)} 
\item 
\label{Bagaria's characterisation 1c} 
$\BFA_{\PP,\kappa}$ 
\item 
\label{Bagaria's characterisation 2c} 
$\fo\BN^1_{\PP,\kappa}$ 
\item 
\label{Bagaria's characterisation 3c} 
$\Vdash_\PP V \prec_{\Sigma^1_1(\kappa)}V[\dot{G}]$
\end{enumerate-(1)} 
If $\cof(\kappa)>\omega$ or there is no inner model with a Woodin cardinal, then the next condition is equivalent to \ref{Bagaria's characterisation 1c}, \ref{Bagaria's characterisation 2c} and  \ref{Bagaria's characterisation 3c}: 
\begin{enumerate-(1)} 
\setcounter{enumi}{3}
\item 
\label{Bagaria's characterisation 5c} 
$\Vdash_\PP H_{\kappa^+}^V \prec_{\Sigma_1} H_{\kappa^+}^{V[\dot{G}]}$
\end{enumerate-(1)} 
If $\cof(\kappa)=\omega$ and $2^{<\kappa}=\kappa$, then the next condition is equivalent to \ref{Bagaria's characterisation 1c}, \ref{Bagaria's characterisation 2c} and \ref{Bagaria's characterisation 3c}: 
\begin{enumerate-(1)} 
\setcounter{enumi}{4}
\item 
\label{Bagaria's characterisation 4c} 
$1_\PP$ forces that no new bounded subset of $\kappa$ are added. 
\end{enumerate-(1)} 
\end{theorem}

The second topic of this paper is the study of name principles for \emph{specific} formulas $\varphi(x)$. 
In particular, we will consider these principles when $\varphi(x)$ denotes a notion of largeness for subsets of $\kappa$ such as being unbounded, stationary, or in the club filter. 
For each of these notions, we also study the corresponding forcing axiom. 
For instance, the \emph{unbounded forcing axiom} $\ub\FA_{\PP,\kappa}$ states: 
\begin{quote} 
\emph{``For any sequence $\vec{D}=\langle D_\alpha \mid \alpha<\kappa\rangle$ of predense subsets of $\PP$, there is a filter $g$ on $\PP$ such that 
$g\cap D_\alpha \neq\emptyset$ for unboundedly many $\alpha<\kappa$.''} 
\end{quote} 
All these principles are defined formally in Section \ref{Section definitions}. 
The next diagram displays some results about them. 
Solid arrows denote non-reversible implications, dotted arrows stand for implications whose converse remains open, and dashed lines indicate that no implication is provable. 
The numbers indicate where to find the proofs. 

\begin{figure}[H]
\[ \xymatrix@R=3.5em{ 
{\txt{$\NP_\kappa$}} \ar@{<->}[r]^{\ref{Corollary equivalence of FA, NP, clubFA, clubNP}} \ar@{<->}[d]_{\ref{Lemma_FA iff N}} & {\txt{$\club\NP_\kappa$}} \ar@{<->}[d]_{\ref{Corollary equivalence of FA, NP, clubFA, clubNP}} \ar@{--}[r]^{\ref{Lemma stat-NP for sigma-centred},\ \ref{Remark_FACohen_meagre} }_{\ref{Remark_strength of statNsigma-closed}} & \txt{$\stat\NP_\kappa$} \ar@{->}[r] \ar@{->}[d]^{\ref{Lemma statN to statFA}}  & \txt{$\ub\NP_\kappa$} \ar@{<->}[d]_{\ref{Lemma ubN to ubFA}}^{\ref{Lemma ubFA to ubN}} & \\ 
{\txt{$\FA_\kappa$}} \ar@{<->}[r]^{\ref{Lemma basic FA implications}}_{\ref{Lemma equivalence of clubFA and FA}} & \txt{$\club\FA_\kappa$} \ar@{->}[r]^{\ref{Lemma basic FA implications}} & \txt{$\stat\FA_\kappa$} \ar@{.>}[r]^{\ref{Lemma basic FA implications}} & \txt{$\ub\FA_\kappa$} & \\ 
}\]
\caption{Forcing axioms and name principles for regular $\kappa$} 
\label{diagram of implications} 
\end{figure}
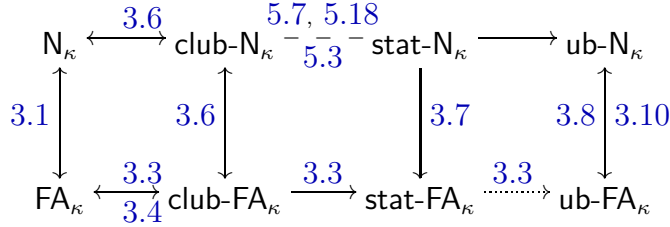


We also investigate whether similar implications hold for $\lambda$-bounded name principles and forcing axioms, where $\lambda$ is any cardinal. 
The results about the cases $\kappa\leq\lambda$, $\omega\leq\lambda<\kappa$ and $1\leq\lambda<\kappa$ are displayed in the next diagrams. 
Here a CBA is a complete Boolean algebra. 

\begin{figure}[H]
\[ \xymatrix@R=3.5em{ 
{\txt{$\BN_\kappa^\lambda$}} \ar@{<.}[r]^{\leftrightarrow \text{ for CBAs}} \ar@{<->}[d] & {\txt{$\club\BN_\kappa^\lambda$}} \ar@{.>}[d]_{\leftrightarrow \text{ for}}^{\text{CBAs}} 
& \txt{$\stat\BN_\kappa^\lambda$} \ar@{.>}[r] \ar@{.>}[d] & \txt{$\ub\BN_\kappa^\lambda$} \ar@{<->}[d] & \\ 
{\txt{$\BFA_\kappa^\lambda$}} \ar@{<->}[r] & \txt{$\club\BFA_\kappa^\lambda$} \ar@{->}[r] & \txt{$\stat\BFA_\kappa^\lambda$} \ar@{.>}[r] & \txt{$\ub\BFA_\kappa^\lambda$} & \\ 
}\]
\caption{$\lambda$-bounded forcing axioms and name principles for regular $\kappa$ and $\lambda\geq\kappa$ } 
\label{diagram of implications - bounded with lambda>kappa} 
\end{figure}
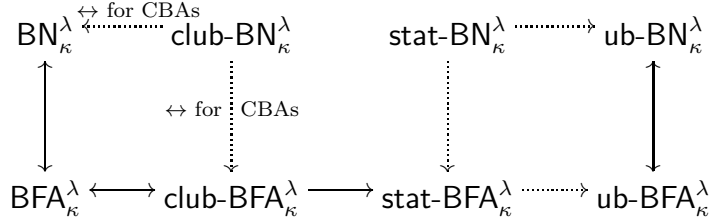 

It is open whether $\club\BN^\lambda_{\PP,\kappa}$ implies $\stat\BN^\lambda_{\PP,\kappa}$. 
Conversely, there are forcings $\PP$ where $\stat\BN^\lambda_{\PP,\kappa}$ holds for all $\lambda$, but $\club\BN^\lambda_{\PP,\kappa}$ fails for all $\lambda\geq\omega$ (see Section \ref{section ccc forcings}, Lemma \ref{Lemma stat-NP for sigma-centred} and Remark \ref{Remark_FACohen_meagre}). 

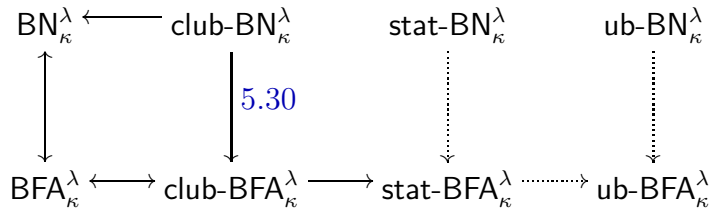
\begin{figure}[H]
\[ \xymatrix@R=3.5em{ 
{\txt{$\BN_\kappa^\lambda$}} \ar@{<->}[d] \ar@{<-}[r] & {\txt{$\club\BN_\kappa^\lambda$}} \ar@{->}[d]^{\ref{Lemma separating BFA from clubBN} } 
& \txt{$\stat\BN_\kappa^\lambda$} \ar@{.>}[d] & \txt{$\ub\BN_\kappa^\lambda$} \ar@{.>}[d] & \\ 
{\txt{$\BFA_\kappa^\lambda$}} \ar@{<->}[r] & \txt{$\club\BFA_\kappa^\lambda$} \ar@{->}[r] & \txt{$\stat\BFA_\kappa^\lambda$} \ar@{.>}[r] & \txt{$\ub\BFA_\kappa^\lambda$} & \\ 
}\]
\caption{$\lambda$-bounded forcing axioms and name principles for regular $\kappa$ and $\omega\leq\lambda<\kappa$ } 
\label{diagram of implications - bounded with lambda<kappa} 
\end{figure} 

Again, it is open whether $\club\BN^\lambda_{\PP,\kappa}$ implies $\stat\BN^\lambda_{\PP,\kappa}$, but the converse implication does not hold by the previous remarks.

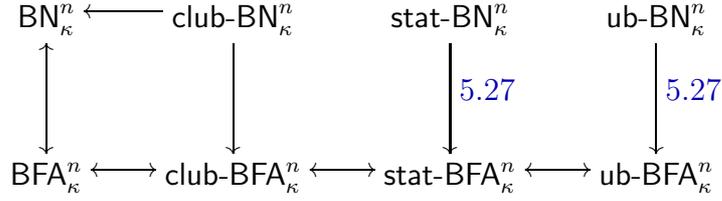
\begin{figure}[H]
\[ \xymatrix@R=3.5em{ 
{\txt{$\BN_\kappa^n$}} \ar@{<->}[d] \ar@{<-}[r] & {\txt{$\club\BN_\kappa^n$}} \ar@{->}[d] 
& \txt{$\stat\BN_\kappa^n$} \ar@{->}[d]^{\ref{Suslin trees}} & \txt{$\ub\BN_\kappa^n$} \ar@{->}[d]^{\ref{Suslin trees}} & \\ 
{\txt{$\BFA_\kappa^n$}} \ar@{<->}[r] & \txt{$\club\BFA_\kappa^n$} \ar@{<->}[r] & \txt{$\stat\BFA_\kappa^n$} \ar@{<->}[r] & \txt{$\ub\BFA_\kappa^n$} & \\ 
}\]
\caption{$n$-bounded forcing axioms and name principles for regular $\kappa$ and $1\leq n<\omega$} 
\label{diagram of implications - bounded with finite lambda} 
\end{figure} 

The principles in the bottom row and $\BN^n_\kappa$ are all provable.

\medskip 
The implications and separations in the previous diagrams are proved using specific forcings such as Cohen forcing, random forcing and Suslin trees. 
For instance, we have the following results:

\begin{proposition}(see Lemma \ref{Lemma_Random ubFA}) 
	Let $\PP$ denote random forcing. The following are equivalent:
	\begin{enumerate-(1)}
		\item 
		$\FA_{\PP,\omega_1}$
		\item 
		$\ub\FA_{\PP,\omega_1}$
		\item 
		$2^\omega$ is not the union of $\omega_1$ many null sets
	\end{enumerate-(1)}
\end{proposition}

\begin{proposition} (see Corollary \ref{Suslin trees})
	Suppose that a Suslin tree exists. 
	Then there exists a Suslin tree $T$ such that  $\stat\BN^1_{T,\omega_1}$ fails.
\end{proposition}

For some forcings, most of Figure \ref{diagram of implications} collapses. 
In particular, if $\ub\FA_{\PP,\kappa}$ implies $\FA_{\PP,\kappa}$, then all entries other than $\stat\NP_{\PP,\kappa}$ are equivalent. 
We investigate when this implication holds. For instance: 

\begin{proposition} (see Lemma \ref{ubFA implies FA for sigma-distributive forcings}) 
For any ${<}\kappa$-distributive forcing $\PP$, we have $\ub\FA_{\PP,\kappa} \implies\FA_{\PP,\kappa}$. 
\end{proposition} 

In a broader range of cases,  $\ub\FA_{\PP,\kappa} $ implies most of the entries in Figure \ref{diagram of implications - bounded with lambda>kappa}:  

\begin{proposition} (see Lemma \ref{ubFA implies BFA}) 
If $\kappa$ an uncountable cardinal and $\PP$ is a complete Boolean algebra that does not add bounded subsets of $\kappa$, then 
$$(\forall q\in \PP\ \ub\FA_{\PP_q,\kappa}) \Longrightarrow \BFA_{\PP,\kappa}^\kappa.$$ 
\end{proposition} 

The previous result is a corollary to the proof of Theorem \ref{Theorem intro Bagaria}. 

\medskip 
We collect some definitions in Section \ref{Section definitions}. 
In Section \ref{Section results for rank 1}, we prove the positive implications in Figure \ref{diagram of implications}. 
In Section \ref{Section correspondence and applications}, we prove a general correspondence between forcing axioms and name principles. 
Theorem \ref{special case of main theorem} is a special case. 
We further derive results about generic absoluteness and other consequences of the correspondence. 
In Section \ref{Section specific classes of forcings}, we study the principles in Figures \ref{diagram of implications}-\ref{diagram of implications - bounded with finite lambda} for specific classes of forcings such as $\sigma$-distributive and c.c.c. and for specific forcings such as Cohen and random forcing. 
We use these results to separate some of the principles in the figures.

\addtocontents{toc}{\protect\setcounter{tocdepth}{1}}
\subsection*{Acknowledgements}

The authors would like to thank Joel David Hamkins for a discussion in February 2020 and for permission to include his proof of Lemma \ref{lemma_failure of lambda-bounded and 1-bounded name principle} and Corollary \ref{lemma_failure of lambda-bounded and 1-bounded name principle for trees} from June 2021. 
They are further grateful to Philip Welch for several discussions and Joan Bagaria for an email exchange in July 2021. 

This project has received funding from the European Union’s Horizon 2020 research and innovation programme under the Marie Sk\l odowska-Curie grant agreement No 794020 of the first-listed author. 
He was partially supported by EPSRC grant number EP/V009001/1, FWF grant number I4039 and the  RIMS Research Center at Kyoto University, where the idea for this project originated in November 2019.

\addtocontents{toc}{\protect\setcounter{tocdepth}{3}}
\section{Some definitions}
\label{Section definitions} 

In this section, we introduce the axioms we will be finding equivalences between. We will also define a few concepts that we will want to use repeatedly. 

\begin{definition}\label{P^alpha(X)}
	Let $X$ be a set and $\alpha$ an ordinal. We recursively define $\mathcal{P}^\alpha(X)$ and $\mathcal{P}^{<\alpha}(X)$:
	
	$\mathcal{P}^0(X)=X$
	
	$\mathcal{P}^{<\alpha}(X)=\bigcup_{\beta<\alpha} \mathcal{P}^\beta(X)$
	
	$\mathcal{P}^\alpha(X) = \mathcal{P}(\mathcal{P}^{<\alpha}(X))$ for $\alpha>0$.
\end{definition}

The axioms we are working with come under two headings: \textit{forcing axioms} and \textit{name principles}. Within these headings there are a variety of different axioms we will be working with.

A \emph{forcing} is a partial order with a largest element $1$. 
Throughout this section, assume that $\PP$ is a forcing and $\mathcal{C}$ is a class of forcings. $G$ will be a generic filter (on $\PP$); $g$ will be a filter on $\PP$ which is contained in the ground model $V$ (and therefore certainly not generic, if $\PP$ is atomless).

\subsection{Forcing axioms}

\begin{notation} 
\label{notation_trace} 
In the following, $\vec{D}=\langle D_\gamma:\gamma<\kappa\rangle$ always denotes a sequence of dense (or predense) subsets of a forcing $\PP$. 
If $g$ is a subset of $\PP$, then its \emph{trace} with respect to $\vec{D}$ is defined as the set $$\Tr_{g,\vec{D}}=\{\alpha<\kappa\mid g\cap D_\alpha \neq \emptyset\}.$$ 
\end{notation} 

\begin{definition}\label{Defn_FA}
	Let $\kappa$ be a cardinal. The forcing axiom $\FA_{\PP,\kappa}$ says: 
	\begin{quote} 
``For any $\vec{D}$, there exists a filter $g\in V$ with $\Tr_{g,\vec{D}}=\kappa$.'' 
	\end{quote}  
	The forcing axiom $\FA_{\mathcal{C},\kappa}$ asserts that $\FA_{\PP,\kappa}$ holds for all $\PP\in \mathcal{C}$. 
\end{definition}

Of course, we could just as well have written ``predense'' instead of ``dense'' in the above definition.

We will suppress the $\PP$ or $\mathcal{C}$ in the above notation when it is clear which forcing we are referring to. If $\kappa=\omega_1$ we will suppress it too, just writing $\FA_\PP$ (or just $\FA$ if $\PP$ is clear as well).

We can weaken this axiom: instead of insisting that $g$ must meet every $D_\gamma$, we could insist only that it meets ``many`` of them in some sense. The following forcing axioms do exactly that, for various senses of ``many''.

\begin{definition}\label{Defn_SpecialFA}
Suppose that $\kappa$ is a cardinal and $\varphi(x)$ is a formula. 
The axiom $\varphi\text{-}\FA_{\PP,\kappa}$ states:  
\begin{quote}
``For any $\vec{D}$, there is a filter $g$ on $\PP$ such that $\varphi(\Tr_{g,\vec{D}})$ holds.''
\end{quote} 
In particular, we will consider the following formulas: 
	\begin{enumerate-(1)}
		\item 
		$\mathsf{club}(x)$ states that $x$ contains a club in $\kappa$. $\club\FA_{\PP,\kappa}$ is called the \textit{club forcing axiom}. 
		\item 
		$\mathsf{stat}(x)$ states that $x$ is stationary in $\kappa$. $\stat\FA_{\PP,\kappa}$ is called the \textit{stationary forcing axiom}. 
		\item 
		$\mathsf{ub}(x)$ states that $x$ contains a club in $\kappa$. $\ub\FA_{\PP,\kappa}$ is called the \textit{unbounded forcing axiom}. 
		\item 
		$\omega\text{-}\mathsf{ub}(x)$ states that $x$ contains $\omega$ as a subset and is also unbounded in $\kappa$. $\omega\text{-}\ub\FA_{\PP,\kappa}$ is called the \textit{$\omega$-unbounded forcing axiom}. 
	\end{enumerate-(1)}
We define $\club\FA_{\mathcal{C},\kappa}$, $\stat\FA_{\mathcal{C},\kappa}$, $\ub\FA_{\mathcal{C},\kappa}$ and $\omega\text{-}\ub\FA_{\mathcal{C},\kappa}$ in the same way as we defined $\FA_{\mathcal{C},\kappa}$ in Definition \ref{Defn_FA}.
\end{definition} 

$\omega\text{-}\ub\FA$ can also be expressed as a combined version of two forcing axioms: that given a $\kappa$ long sequence $\vec{D}$ and a separate $\omega$ long sequence $\vec{E}$ of (pre)dense sets, we can find a filter $g$ such that $\Tr_{g,\vec{D}}$ is unbounded and $\Tr_{g,\vec{E}}=\omega$.

Again, we will suppress $\PP$ or $\mathcal{C}$ where they are obvious, and will suppress $\kappa$ when $\kappa=\omega_1$.

We can also weaken the axiom by insisting that every dense set $D_\gamma$ be \textit{bounded} in cardinality, by some small cardinal.

\begin{definition}\label{Defn_BFA}
	Let $\kappa$ and $\lambda$ be cardinals. The bounded forcing axiom 
	$\BFA_{\PP,\kappa}^\lambda$ says 
	\begin{quote}``Whenever $\langle D_\gamma,\gamma<\kappa\rangle$ is a sequence of \textit{predense} subsets of $\PP$, \textit{and for all} $\gamma$ \textit{we have} $\lvert D_\gamma\rvert \leq \lambda$, then there is a filter $g\in V$ such that for all $\gamma<\kappa$, $g\cap D_\gamma\neq \emptyset$.'' 
	\end{quote} 
	
	We define $\BFA_{\mathcal{C}, \kappa}^\lambda$, $\club\BFA_{\PP,\kappa}^\lambda$ and so forth in the natural way, using definitions analogous to those in \ref{Defn_FA} and \ref{Defn_SpecialFA}.
\end{definition}

Again, we will suppress notation as described above. We will suppress the $\lambda$ if $\lambda=\kappa$.

Note that we are definitely looking at predense sets here, since actual dense sets are likely to be rather large and the axiom would be likely to be trivial if we had to use dense sets. These bounded forcing axioms are only really of interest when $\PP$ is a Boolean algebra, since they always contain (nontrivial) predense sets with as few as two elements so the axiom will not be vacuous.

There is one more forcing axiom we want to introduce, but it requires some additional notation so we will postpone it until later in this section.

\subsection{Name principles}

We ought to define name principles at this point, but we need to cover some other terminology first in order to express the definitions.

As one might expect, name principles are about different $\PP$ names, and it will be useful to have some measure of how complex a name is. The following three definitions are all different ways of doing this; we will be using all of them.

\begin{definition}\label{Defn_Ranks}
	Let $X$ be a set (in $V$). We recursively define a a name's rank as follows.
	
	$\sigma$ is an $\alpha$ rank $X$ name (or a rank $\alpha$ name for short) if either:
	\begin{itemize}
		\item $\alpha=0$ and $\sigma=\check{x}$ for some $x\in X$; or
		\item $\sigma$ is not rank $0$ and $\alpha=\sup\{\text{rank}(\tau): \exists p\in \PP (\tau,p)\in \sigma\}$
	\end{itemize}
\end{definition}

We also call a $1$ (or $0$) rank $X$ name a \textit{good} name.
Of course, we will also talk about rank $\leq \alpha$ names, meaning names which are either rank ${<}\alpha$ or rank $\alpha$.

This definition is a name analogue to saying that $\sigma \in \pow^\alpha(X)$, where $X$ is transitive. Most of the time, we will be interested in the case where $X$ is some cardinal, most often either $0$ or $\omega_1$. Note that every $\PP$ name is an $\alpha$ rank $X$ name for some $\alpha$.

\begin{definition}\label{Defn_kappasmall}
	Let $\sigma$ be a $\PP$ name and $\kappa$ be a cardinal. We say $\sigma$ is \emph{locally $\kappa$ small} if there are at most $\kappa$ many names $\tau$ such that for some $p\in \PP$, we have $(\tau,p)\in \sigma$.
	A name $\sigma$ is \textit{$\kappa$ small} if it is locally $\kappa$ small, and every name $\tau$ in the above definition is $\kappa$ small.
\end{definition}

If being rank $\alpha$ is analogous to being in $\pow^\alpha$ (or $\pow^\alpha(X)$) then the analogue of being $\kappa$ small would be being in $H_{\kappa^+}$. We could also easily define a version of this for $H_{\kappa^+}(X)$ if we wanted. However, we don't actually need to: in all the cases we're going to be interested in, $\bar{X}$ will have cardinality $\leq \kappa$ and the definition would be equivalent to the above one.

The following proposition says that we only really need to worry about $\kappa$ smallness when we go above rank $1$ names.

\begin{proposition}
	Let $X$ be transitive, and of size at most $\kappa$. Let $\sigma$ be a $0$ rank or $1$ rank $X$ name. Then $\sigma$ is $\kappa$ small.
\end{proposition}

On the other hand if $X$ has size greater than $\kappa$ then no interesting rank $1$ name will be $\kappa$ small.

The next definition does not have an easy analogue, but is a kind of complement to the previous one and is critical when we work with bounded forcing axioms.

\begin{definition}\label{Defn_lambdabounded}
	Let $\sigma$ be a $\PP$ name and $\lambda$ be a cardinal. We say $\sigma$ is \textit{locally $\lambda$ bounded} if it can be written as
	\begin{equation*}
		\sigma = \{ (\tau,p): \tau \in T, p\in S_\tau\}
	\end{equation*}
where $T$ is some set of names, and for $\tau \in T$ the set $S_\tau$ is a subset of $\PP$ of size at most $\lambda$. 
A name $\sigma$ is \textit{$\lambda$ bounded} if it is locally $\lambda$ bounded, and every name $\tau\in T$ in the above definition is $\lambda$ bounded.\end{definition}

A good name which is $1$ bounded is known as a very good name.
A check name $\check{x}$ has the form $\{(\check{y},1): y\in x\}$ and is therefore guaranteed to be $\lambda$ bounded for any $\lambda>0$.

We will be talking about interpreting names with respect to a filter. Unfortunately, the literature uses two different meanings of the word ``interpretation'', which only coincide if the filter is generic. For clarity:

\begin{definition}\label{Defn_Interpretation}
	Let $\sigma$ be a name, and $g$ a filter. (Here, $g$ may be inside $V$ or in some larger model.) When we refer to the \textit{interpretation} $\sigma^g$ of $\sigma$, we mean the recursive interpretation:
	\begin{equation*}
		\sigma^g := \{ \tau^g: \exists p\in g (\tau,p)\in \sigma\}
	\end{equation*}
	When we refer to the \textit{quasi-interpretation} $\sigma^{(g)}$, we mean the following set:
	\begin{equation*}
		\sigma^{(g)}:=\{x\in V: \exists p\in g p\forces \check{x}\in \sigma\}
	\end{equation*}
\end{definition}

\begin{proposition}
	$\sigma^g=\sigma^{(g)}$ if either
	\begin{enumerate-(1)}
		\item $g$ is generic; or
		\item $\sigma$ is a $1$ rank $X$ name (for any $X$) and is $1$ bounded.
	\end{enumerate-(1)}
\end{proposition}

\begin{proposition}
	Suppose $\PP$ is a complete Boolean algebra, and $\sigma$ is a $1$ rank $X$ name. Then we can find a name $\tau$ such that for every filter $g$, $\tau^g=\tau^{(g)}=\sigma^{(g)}$.
\end{proposition}
\begin{proof}
	For $x\in X$ let $p_x=\sup\{p\in \PP: (\check{x},p)\in \sigma\}$ (so $p_x\in \PP\cup \{0\}$). Let $\tau=\{(\check{x},p_x): x\in X, p_x\neq 0\}$.
\end{proof}

We can now define our name principles. Here, we take $\PP$ to be a forcing, $\mathcal{C}$ a class of forcings, and $X$ an arbitrary set.

\begin{definition}\label{Defn_nameprinciple}
	Let $\alpha$ be an ordinal, $\kappa$ a cardinal and $X$ a transitive set of size at most $\kappa$. The name principle $\NP_{\PP,X,\kappa}(\alpha)$ says the following: 
	\begin{quote} ``Whenever $\sigma$ is a $\kappa$ small ${\leq}\alpha$ rank $X$ name, and $A\in H_{\kappa^+}\cap \pow^\alpha(X)$ is a set such that $\PP\forces \sigma=\check{A}$, there is a filter $g\in V$ such that $\sigma^g=A$.''
	\end{quote} 
	
	$\NP_{\mathcal{C},X,\kappa}(\alpha)$ is the statement that $\NP_{\QQ,X,\kappa}(\alpha)$ holds for all $\QQ\in \mathcal{C}$. 
	$\NP_{\PP,\kappa}(\infty)$ (resp. $\NP_{\mathcal{C},\kappa}(\infty)$) is the statement that $\NP_{\PP,X,\kappa}(\alpha)$ (resp. $\NP_{\mathcal{C},X,\kappa}(\alpha)$) holds for all $\alpha\in \Ord$ and all $X\in H_{\kappa^+}$. (Equivalently, we could just require that it holds for $\alpha\leq\kappa^+$ and all $X\in H_{\kappa^+}$.)
\end{definition}

Some comments on this definition: 
It is easy to see that if $\sigma$ is a $\kappa$ small $X$ name, and $g\in V$, then $\sigma^g\in H_{\kappa^+}$. If $\sigma$ is rank ${\leq} \alpha$, then it is also easy to see that $\sigma^g\in \mathcal{P}^\alpha(X)$. So if we didn't require that $A\in H_{\kappa^+}\cap \mathcal{P}^\alpha(X)$, then the principle would fail trivially for most forcings. The only forcings on which it could hold would be those which don't force any names to be equal to such large $A$ anyway.

This argument also shows that the name principle fails trivially if, for some $\lambda<\kappa$, there is a $\lambda$ small $\sigma$ which is forced to be equal to some $A\not \in H_{\lambda^+}$. So we might think we should exclude such names from the principle as well. But in fact, we shall see in Section \ref{Section correspondence and applications} that it makes little difference: the proof of Theorem \ref{correspondence forcing axioms name principles} shows that if a name principle 
fails because of such a name, then it also fails for non-trivial reasons.

We can easily see that if $\sigma$ is a $\kappa$-small $1$ rank $X$ name, and is forced to be equal to $A$, then $A\subseteq X$ and $\lvert A \rvert \leq \kappa$. Hence, when we're dealing with $\NP(1)$, we don't need to worry about checking if the names we're working with are in $H_{\kappa^+}\cap \pow(X)$, as this is automatically true.
On the other hand, once we go above rank $1$, these names can exist, even for small values of $\alpha$ and $\kappa$. For example, \cite[Lemma 7.1]{holy2019sufficient} has an $\omega$ bounded rank $2$ name which is forced to be equal to $(2^\omega)^V$.

One might ask why we allowed $X$-names for all $X\in H_{\kappa^+}$ in the definition of $\NP_{\PP,\kappa}(\infty)$. 
This is because any such name can be understood as an $\emptyset$-name of some high rank, so these principles already follow from the conjunction of $\NP_{\PP,\emptyset,\kappa}(\alpha)$ for all $\alpha\in \Ord$. 

As with the forcing axioms, we will sometimes omit part of this notation. We will drop $\PP$ and $\mathcal{C}$ when they are clear from context. We will omit $\alpha$ when $\alpha=1$. While $X$ is formally just some arbitrary set, most of the time it can be thought of as a cardinal; we will omit it in the case that $X=\kappa$, and will then omit $\kappa$ as well if $\kappa=\omega_1$.

Most often, these omissions will come up when we're assuming $\alpha=1$ and taking $X$ to be some cardinal. In that situation, $\kappa$ smallness is essentially trivial: if $\kappa<X$ then our class of names is too restrictive to do anything interesting, and if $\kappa \geq X$ then every $1$ rank $X$ name will be $\kappa$ small, automatically. So when $\alpha=1$ and $X$ is a cardinal we can find out everything we need to know just by looking at the case $X=\kappa$.

We can also define variations analogous to $\club \FA$, $\stat \FA$, etc. However, this only really makes sense when we know $\sigma$ a subset of some cardinal. For this reason, we only define these variations for the case where $\alpha=1$ (also dropping the requirement of $\kappa$-smallness) and where $X$ is a cardinal.

\begin{definition}\label{Defn_SpecialN}
	Let $\kappa$ be a cardinal and $\varphi(x)$ a formula. 
	The axiom $\varphi\text{-}\NP_{\PP,\kappa}$ states: 	
\begin{quote}
		``For any $1$ rank $\kappa$ name $\sigma$, 
		if $\PP\forces \varphi(\sigma)$ then there is a filter $g$ on $\PP$ such that $\varphi(\sigma^g)$ holds in $V$.'' 
\end{quote} 
In particular, we shall consider the axioms for the formulas $\mathsf{club}(x)$, $\mathsf{stat}(x)$, $\mathsf{ub}(x)$ and $\omega\text{-}\mathsf{ub}(x)$ given in Definition \ref{Defn_SpecialFA}: 
	\begin{enumerate-(1)}
		\item 
		The \emph{club name principle} $\club \NP_{\PP,\kappa}$. 
		\item 
		The \emph{stationary name principle} $\stat \NP_{\PP,\kappa}$. 
		\item 
		The \emph{unbounded name principle} $\ub \NP_{\PP,\kappa}$. 
		\item 
		The \emph{$\omega$-unbounded name principle} $\omega\text{-} \ub \NP_{\PP,\kappa}$. 
	\end{enumerate-(1)}
\end{definition}

As usual, we also define similar axioms with $\mathcal{C}$ in place of $\PP$. Note that we could also express $\omega\text{-}\ub\NP$ as an axiom about two names, one of which is forced to be an unbounded subset of $\kappa$ while the other is forced to be equal to $\omega$.

\begin{remark} 
The axioms $\club\FA_{\PP,\kappa}$, $\stat\FA_{\PP,\kappa}$, $\ub\FA_{\PP,\kappa}$ and $\omega\text{-} \ub \FA_{\PP,\kappa}$ in Definition \ref{Defn_SpecialFA} can be understood as a more general form of name principles for two formulas $\varphi(x)$ and $\psi(x)$: 
\begin{quote}
		``For any $1$ rank $\kappa$ name $\sigma$, 
		if $\PP\forces \varphi(\sigma)$ then there is a filter $g$ on $\PP$ such that $\psi(\sigma^g)$ holds in $V$,'' 
\end{quote} 
For instance, $\stat\FA_{\PP,\kappa}$ is equivalent to the statement: 
\begin{quote} 
``If $\sigma$ is a rank $1$ name for $\omega_1$, then there is a filter $g\in V$ such that $\sigma^g$ is stationary.'' 
\end{quote} 
\end{remark} 

We can also generalise the ideas here: rather than simply working with a single statement like ``$\sigma$ is unbounded'' or ``$\sigma$ is some particular set in $V$'', we could ask to be able to find a filter to correctly interpret every reasonable statement.

In the following definition, we allow bounded quantifiers in our $\Sigma_0$ formulas.

\begin{definition}\label{Defn_foN}
	Let $\alpha$ be an ordinal and $\kappa$ a cardinal. The simultaneous name principle $\fo\NP_{\PP,X,\kappa}(\alpha)$ says the following: 
	\begin{quote} 
	``Whenever $\sigma_0,\ldots, \sigma_n$ are $\kappa$ small ${\leq}\alpha$ rank $X$ names, we can find a filter $g$ in $V$ such that $\varphi(\sigma_0^g,\dots,\sigma_n^g)$ holds for \textit{every} $\Sigma_0$ formula $\varphi$ such that $\PP \forces \varphi(\sigma_0,\ldots,\sigma_n)$.'' 
	\end{quote} 
	Moreover: 
	\begin{itemize} 
	\item 
	The \emph{simultaneous name principle} $\fo\NP_{\PP,\kappa}(\infty)$ is the same statement, except that the names are $X$ names for some $X\in H_{\kappa^+}$ and there is no restriction on their rank. 
	\item 
	$\fo\NP_{\mathcal{C},X,\kappa}(\alpha)$ is the statement that $\fo\NP_{\QQ,X,\kappa}(\alpha)$ holds for all $\QQ\in \mathcal{C}$. 
	\item 
	 $\fo\NP_{\mathcal{C},\kappa}(\infty)$ is defined similarly. 
	\item 
The bounded name principles $\fo\BN^\lambda_{\PP,X,\kappa}(\alpha)$ are defined similarly. 
	\end{itemize} 
\end{definition}



The $\Sigma_0$ requirement on $\varphi$ is necessary, because otherwise the axiom would say that any sentence which is forced to be true by $\PP$ is already true in $V$. This would make the axiom trivially false for almost all interesting forcings.
Again we will suppress $X$, $\kappa$ and $\alpha$ as described earlier.

All of these name principles also have bounded variants:

\begin{definition}\label{Defn_NamePrincipleBounded}
	Let $\alpha$ be an ordinal and $\kappa, \lambda$ cardinals. The bounded name principle $\BN_{\PP,X,\kappa}^{\lambda}(\alpha)$ says the following: 
	\begin{quote} 
	``Whenever $\sigma$ is a $\kappa$ small $\lambda$ bounded $\leq\alpha$ rank $X$ name, and $A$ is a set such that $\PP\forces \sigma=A$, we can find a filter $g\in V$ such that $\sigma^g=A$.'' 
	\end{quote} 
\end{definition}

We define similar bounded forms of all the other name principles we have introduced so far. Again, we will suppress $\lambda$ when $\lambda=\kappa$ and will suppress other notation as described above.

\subsection{Hybrid axioms}

There is one more group of axioms which are worth mentioning, because of their frequent use in the literature. They are a hybrid of forcing axiom and name principle. 
The axioms $\MA^+$ and $\PFA^+$ were introduced introduced by Baumgartner in \cite[Section 8]{baumgartner1984applications}.  

\begin{definition}
	The forcing axiom $\FA^+_{\PP,\kappa}$ says: 
	\begin{quote} 
	Suppose $\vec{D}=\langle D_\gamma: \gamma<\kappa\rangle$ is a sequence of dense subsets of $\PP$ and let $\sigma$ be a $1$ rank $\kappa$ name such that $\PP\forces ``\sigma$ is stationary``. Then there is a filter $g$ such that
	\begin{enumerate-(1)}
		\item For all $\gamma$, $D_\gamma \cap g \neq \emptyset$; and
		\item $\sigma^g$ is stationary.
	\end{enumerate-(1)}
	\end{quote} 
	The forcing axiom $\FA^{++}_{\PP,\kappa}$ says: 
		\begin{quote} 
		Let $D_\gamma: \gamma<\kappa$ be dense subsets of $\PP$ and let $\sigma_\gamma: \gamma<\kappa$ be $1$ rank $\kappa$ names such that $\PP\forces ``\sigma_\gamma$ is stationary`` for every $\gamma$. Then we can find a filter $g$ such that
	\begin{enumerate-(1)}
		\item For all $\gamma$, $D_\gamma \cap g \neq \emptyset$; and
		\item For all $\gamma$, $\sigma_\gamma^g$ is stationary.
	\end{enumerate-(1)}
		\end{quote} 
\end{definition}

As usual, we will also use versions of the above with $\mathcal{C}$ in place of $\PP$, and bounded versions.

We have actually gone against convention slightly here: the literature generally uses the quasi-interpretation $\sigma^{(g)}$ when defining $\FA^+$ and $\FA^{++}$ style axioms. However, our version is in fact equivalent, as the following theorem shows:

\begin{theorem} 
	Let $\FA^{(+)}$ and $\FA^{(++)}$ be defined in the same way as $\FA^+$ and $\FA^{++}$ above, but with $\sigma^{(g)}$ and $\sigma_\gamma^{(g)}$ in place of $\sigma^g$ and $\sigma_\gamma^g$ respectively.
Then 
	$\FA_{\PP,\kappa}^+\iff \FA_{\PP,\kappa}^{(+)}$ and $\FA_{\PP,\kappa}^{++}\iff \FA_{\PP,\kappa}^{(++)}.$ 
\end{theorem}
\begin{proof}
	We will prove the $\FA^+$ case; the $\FA^{++}$ version is similar. The $\Leftarrow$ direction is trivial.
	
	$\Rightarrow$: Let $D_\gamma: \gamma<\kappa$ be a collection of $\kappa$ many dense subsets of $\mathbb{P}$. Let $\sigma$ be a rank $1$ name with $\mathbb{P}\Vdash ``\sigma \text{ is stationary}"$.
	
	For $\gamma\in \kappa$, let
	\begin{equation*}
		E_\gamma:=\{p\in \PP\colon p\forces \check{\gamma}\not \in \sigma \text{ or } \exists q\geq p \, (\check{\gamma},q) \in \sigma\}
	\end{equation*}
	We can see that $E_\gamma$ is dense: given $p\in \PP$, either we can find some $q\parallel p$ with $\langle \check{\gamma},q\rangle \in \sigma$ and we're done, or $p\forces \check{\gamma}\not \in \sigma$ since all the elements of $\sigma$ are check names.
	\begin{claim} 
		If $g$ is any filter which meets all the $E_\gamma$, then $\sigma^g=\sigma^{(g)}$
	\end{claim}
	\begin{proof}
		$\subseteq$: Let $\gamma \in \sigma^g$. Then there is a $q\in g$ with $(\check{\gamma},q) \in \sigma$. Clearly $q\forces \check{\gamma}\in \sigma$, so $\gamma \in \sigma^{(g)}$.
		
		$\supseteq$: Let $\gamma\in \sigma^{(g)}$. Then we can find $r\in g$ with $r\forces \check{\gamma}\in \sigma$. Certainly, then, there is no $p\in g$ with $p\forces \check{\gamma}\not \in \sigma$. Since nonetheless $g$ meets $E_\gamma$, there must be some $q\in g$ with $(\check{\gamma},q) \in \sigma$. Hence $\gamma \in \sigma^g$. 
	\end{proof}
	Now we simply use our forcing axiom to take a filter $g$ which meets all the $D_\gamma$, all the $E_\gamma$, and which is such that $\sigma^{(g)}$ is stationary.
\end{proof}

In defining the $E_\gamma$ in the above proof, we used a technique which we will be invoking many times. It will save us a lot of time if we give it a name now.

\begin{definition}
	Let $\tau$ and $\sigma$ be names, and $p\in \PP$. We say $p$ strongly forces $\tau \in \sigma$, and write $p\sforces \tau \in \sigma$, if there exists $q\geq p$ with $(\tau,q)\in \sigma$.
\end{definition}

The value of this definition is shown in the following two propositions.

\begin{proposition}\label{Prop_sforcingAndForcing} Let $\sigma$ and $\tau$ be names, and $p\in \PP$.
	\begin{enumerate-(1)}
		\item If $p\forces \tau \in \sigma$ then there exist densely many $r\leq p$ such that for some name $\tilde{\tau}$, $r\forces \tilde{\tau}=\tau$ and $r\sforces \tilde{\tau}\in \sigma$.
		\item If $p\sforces \tau \in \sigma$ then $p\forces \tau \in \sigma$.
	\end{enumerate-(1)}
\end{proposition}
	
\begin{proposition}\label{Prop_sforcingAndInterpretation}
	Let $\sigma$ and $\tau$ be names, let $p\in \PP$ and let $g$ be any filter containing $p$.
	\begin{enumerate-(1)}
	\item If $p\sforces \tau \in \sigma$ then $\tau^g\in \sigma^g$.
	\item If for all $\tilde{\tau}$ with $(\tilde{\tau},q)\in \sigma$ (for some $q\in \PP$) we either know $\tau^g\neq \tilde{\tau}^g$ or have $p\forces \tilde{\tau}\not \in \sigma$ then $\tau^g \not \in \sigma^g$.
	\end{enumerate-(1)}
\end{proposition}

\section{Results for rank $1$} 
\label{Section results for rank 1} 

We will start by looking at the positive results we can prove in general about forcing axioms and rank 1 name principles. 
We again take $\PP$ to be an arbitrary forcing. We also take $\kappa$ to be an uncountable cardinal, although we're mostly interested in the case where $\kappa=\omega_1$.
Since $\PP$ is arbitrary, we could just as easily replace it with a class $\mathcal{C}$ of forcings in all our results.

\subsection{Basic implications} 

All the positive results expressed in Figure \ref{diagram of implications} are proved in this section. The negative results will be proved later, when we look at the specific forcings that provide counterexamples.
We will not need that $\kappa$ is regular. In the case of $\cof(\kappa)=\omega$, a club is 

\begin{lemma}\label{Lemma_FA iff N}
$\FA_{\PP,\kappa}\iff \NP_{\PP,\kappa}$
\end{lemma}
\begin{proof}
	$\Rightarrow$: Assume $\FA_\kappa$. (That is, $\FA_{\PP,\kappa}$, recall that we said we'd suppress the $\PP$ whenever it was clear.) Let $\sigma$ be a rank $1$ name for a subset of $\kappa$, and suppose that $1\forces \sigma=A$ for some $A\subseteq \kappa$. For $\gamma \in A$, let
	\begin{equation*}
		D_\gamma=\{p\in\PP: p\sforces \check{\gamma} \in \sigma\}
	\end{equation*}
	It is clear that $D_\gamma$ is dense by Proposition \ref{Prop_sforcingAndForcing}.
	
	For $\gamma \in \kappa \setminus A$, let $D_\gamma=\PP$.
	
	Using $\FA_\kappa$, take a filter $g$ that meets every $D_\gamma$. We claim that $\sigma^g=A$.
	
	For $\gamma\in A$, we know that some $p\in g$ strongly forces $\check{\gamma} \in \sigma$. By \ref{Prop_sforcingAndInterpretation} then, $\gamma \in\sigma^g$. Conversely, if $\gamma \not \in A$ then $1\forces \check{\gamma}\not \in \sigma$ and by the same proposition $\gamma \not \in \sigma$.
	
	$\Leftarrow$: Assume $\NP_\kappa$. Let $D_\gamma, \gamma<\kappa$ be a collection of dense subsets of $\PP$.
	
	Let
	\begin{equation*}
		\sigma=\{(\check{\gamma},p): \gamma<\kappa, p\in D_\gamma\}
	\end{equation*}
	
	It is easy to see that $1\forces \sigma=\kappa$. Take a filter $g$ such that $\sigma^g=\kappa$, and then for all $\gamma<\kappa$ $D_\gamma \cap g \neq \emptyset$.
\end{proof}

\begin{lemma} 
\label{Lemma FA bracket interpretation} 
$\FA_{\PP,\kappa}$ holds if and only if for every rank $1$ name $\sigma$ for a subset of $\kappa$, there is some $g$ with $\sigma^{(g)}=\sigma^g$. 
\end{lemma} 
\begin{proof} 
First suppose that $\FA_{\PP,\kappa}$ holds and $\sigma$ is a rank $1$ $\PP$-name for a subset of $\kappa$. 
Note that $\sigma^g\subseteq \sigma^{(g)}$ holds for all filters $g$ on $\PP$. 
For each $\alpha<\omega_1$, 
$$ D_\alpha= \{ p\in \PP \mid p\forces \check{\alpha}\notin \sigma \vee p \sforces \alpha\in \sigma \}   $$ 
is dense. 
By $\FA_{\PP,\kappa}$, there is a filter $g$ with $g\cap D_\alpha$ for all $\alpha<\omega_1$. 
To see that $\sigma^{(g)}\subseteq \sigma^g$ holds, suppose that $\alpha\in \sigma^{(g)}$. 
Thus there is some $p\in g$ which forces $\check{\alpha}\in \sigma$. 
Take any $q\in g\cap D_\alpha$. 
Since $p\parallel q$, we have $p \sforces \alpha\in \sigma$ by the definition of $D_\alpha$ and thus $\alpha\in \sigma^g$. 

On the other hand, $\NP_{\PP,\kappa}$ and thus $\FA_{\PP,\kappa}$ (by Lemma \ref{Lemma_FA iff N}) follows trivially from this principle, since for any rank $1$ name $\sigma$ with $\forces \sigma=\check{A}$, we have $\sigma^{(g)}=A$ for any filter $g$. 
\end{proof}

\begin{lemma} \ 
\label{Lemma basic FA implications}  
	\begin{enumerate-(1)}
		\item 
		\label{Lemma basic FA implications 1}  
		$\FA_{\PP,\kappa}\implies \club\FA_{\PP,\kappa}\implies \ub \FA_{\PP,\kappa}$ 
		\item 
		\label{Lemma basic FA implications 2}  
		$\FA_{\PP,\kappa}\implies \stat \FA_{\PP,\kappa}\implies \ub\FA_{\PP,\kappa}$ 
		\item 
		\label{Lemma basic FA implications 3}  
		$\FA_{\PP,\kappa}\implies\omega\text{-}\ub\FA_{\PP,\kappa}\implies \ub\FA_{\PP,\kappa}$
		\item 
		\label{Lemma basic FA implications 4}  
		If $\cof(\kappa)>\omega$, then $\club\FA_{\PP,\kappa}\implies \stat \FA_{\PP,\kappa}$
	\end{enumerate-(1)}
\end{lemma}
\begin{proof}
	Follows immediately from the definitions of the axioms.
\end{proof}

\begin{lemma} 
\label{Lemma equivalence of clubFA and FA}
	$\club\FA_{\PP,\kappa}\Longleftrightarrow\FA_{\PP,\cof(\kappa)}$. 
\end{lemma} 
\begin{proof} 
For $\cof(\kappa)=\omega$, the statements are both provably true. So assume $\cof(\kappa)>\omega$. 

$\Longleftarrow$: 
Let $\pi\colon \cof(\kappa)\rightarrow \kappa$ be a continuous cofinal function. 
Let $\vec{D}=\langle D_\alpha\mid \alpha<\kappa\rangle$ be a sequence of dense open subsets of $\PP$. 
Let $\vec{E}=\langle E_\beta\mid \beta<\lambda\rangle$, where $E_\alpha=D_{\pi(\alpha)}$ for $\alpha<\cof(\kappa)$. 
By $\FA_{\PP,\cof(\kappa)}$, there is a filter $g$ with $g\cap E_\alpha$ for $\alpha<\cof(\kappa)$. 
Thus for all $\beta=\pi(\alpha)\in \ran(\pi)$, $g\cap D_\alpha=g\cap E_\beta\neq\emptyset$. This suffices since $\ran(\pi)$ is club in $\kappa$. 

$\Longrightarrow$: 
We first claim that $\club\FA_{\PP,\kappa}$ implies $\club\FA_{\PP,\cof(\kappa)}$. 
To see this, let $\pi\colon \cof(\kappa)\rightarrow \kappa$ be a continuous cofinal function. 
Let $\vec{D}=\langle D_\alpha\mid \alpha<\cof(\kappa)\rangle$ be a sequence of dense open subsets of $\PP$. 
Let $E_{\pi(\alpha)}=D_\alpha$ and $E_\gamma=\PP$ for all $\gamma\notin \ran(\pi)$. 
Since $C\cap \ran(\pi)$ is club in $\kappa$ and $\pi$ is continuous, $\pi^{-1}(C)$ is club in $\cof(\kappa)$ and $g\cap D_\alpha =g \cap E_{\pi(\alpha)}\neq \emptyset$ for all $\alpha\in \pi^{-1}(C)$ as required. 


It now suffices to prove $\club\FA_{\PP,\lambda}\Longrightarrow\FA_{\PP,\lambda}$ for regular $\lambda$. 
Given a sequence $\vec{D}=\langle D_\alpha\mid \alpha<\lambda\rangle$ of dense open subsets, partition $\lambda$ into disjoint stationary sets $S_\alpha$ for $\alpha<\kappa$. 
	Let $\vec{E}=\langle E_\beta\mid \beta<\lambda\rangle$, where $E_\beta=D_\alpha$ for $\beta\in S_\alpha$. 
	By $\club\FA_\lambda$, there is a filter $g$ and a club $C$ in $\lambda$ with $g\cap E_\beta$ for $\beta\in C$. 
	Since $C$ is club, $S_\alpha\cap C\neq \emptyset$ for all $\alpha<\lambda$. 
	Thus $g\cap D_\alpha=g\cap E_\beta\neq\emptyset$. 
\end{proof}

\begin{lemma} \ 
\label{Lemma FA, club-N and club-FA} 
\begin{enumerate-(1)}
	\item 
	\label{Lemma FA, club-N and club-FA 1}
	$\FA_\kappa\implies \club\NP_\kappa$
	\item 
    \label{Lemma FA, club-N and club-FA 2}
	$\club\NP_\kappa \implies \club\FA_\kappa$
\end{enumerate-(1)}
\end{lemma} 
\begin{proof} 
	\ref{Lemma FA, club-N and club-FA 1}:
	Let $\sigma$ be a rank $1$ name such that $1\forces ``\sigma$ contains a club in $\kappa$''. Then we can find a rank $1$ name $\tau$ such that $1\forces \tau\subseteq \sigma$ and $1\forces ``\tau$ is a club in $\kappa$''.
	For $\gamma<\kappa$, let $D_\gamma$ denote the set of $p\in\PP$ such that either 
	\begin{enumerate-(a)} 
		\item 
		$p\sforces \check{\gamma}\in\tau$, or 
		\item 
		for all sufficiently large  $\alpha<\gamma$,  $p \forces \check{\alpha} \notin \tau$. 
	\end{enumerate-(a)}
	We claim $D_\gamma$ is dense. Let $p\in \PP$. If $p\forces \check{\gamma} \in \tau$ then by Proposition \ref{Prop_sforcingAndForcing} we can find $q\leq p$ strongly forcing this, and then $q\in D_\gamma$. Otherwise, take $q\leq p$ with $q\forces \check{\gamma}\not \in \tau$. Then $q\forces ``\tau \cap \gamma$ is bounded in $\gamma$''. Take $r\leq q$ deciding that bound, and then $r$ satisfies condition b above.
 
	For any filter $g$ with $g\cap D_\gamma \neq \emptyset$, $\tau^g$ is closed at $\gamma$ by Proposition \ref{Prop_sforcingAndInterpretation}.
	
	Let $E_\gamma$ denote the set of $p\in \PP$ such that for some $\delta\geq \gamma$, $p\sforces \check{\delta}\in \tau$. Again, this is dense since $\tau$ is forced to be unbounded. For any filter $g$ with $g\cap E_\gamma \neq \emptyset$ for all $\gamma<\kappa$, $\tau^g$ is unbounded. 
	
	Let $F_\gamma$ denote the dense set of $p\in\PP$ such that $p\sforces \check{\gamma}\in \sigma$ or $p\forces\check{\gamma}\notin\tau$. 
	Once again, $F_\gamma$ is dense: given $p\in \PP$ take $q\leq p$ deciding whether $\gamma \in \tau$. If it decides $\gamma \not \in \tau$ then we're done; otherwise $q\forces \check{\gamma} \in \sigma$ and we can find $r\leq q$ with $r\sforces\check{\gamma}\in \sigma$
	
	For any filter $g$ with $g\cap F_\gamma\neq\emptyset$, $\gamma\in\tau^g \Rightarrow \gamma\in\sigma^g$.
	
	Putting things together, if we find a filter $g$ which meets every $D_\gamma$, $E_\gamma$ and $F_\gamma$ then $\tau^g$ will be both a club and a subset of $\sigma^g$.
	
	\ref{Lemma FA, club-N and club-FA 2}:
	This works much like the proof that $\NP\Rightarrow \FA$ above. Let $D_\gamma:\gamma<\kappa$ be a collection of dense sets. Let
	
	\begin{equation*}
		\sigma=\{(\check{\gamma},p):\gamma<\kappa,p\in D_\gamma\}
	\end{equation*}

	Clearly $1\forces \sigma=\check{\kappa}$, and hence that $\sigma$ contains a club. Take a filter $g$ where $\sigma^g$ contains a club. Then $\sigma^g=\{\gamma<\kappa: D_\gamma\cap g\neq \emptyset\}$ so $g$ meets a club of $D_\gamma$.
\end{proof} 

Putting together the previous results, we complete the top left corner of Figure \ref{diagram of implications}. 

\begin{corollary}
\label{Corollary equivalence of FA, NP, clubFA, clubNP} 
	The following are all equivalent for all uncountable regular cardinals $\kappa$: $\FA_\kappa$, $\NP_\kappa$, $\club\FA_\kappa$, $\club\NP_\kappa$.
\end{corollary}

The second half of the previous lemma also applies for the other special name principles.

\begin{lemma}\label{Lemma statN to statFA}
	$\stat\NP_\kappa \implies \stat\FA_\kappa$
\end{lemma}
\begin{proof}
As for the $\mathsf{club}$ case, except that we just insist on $\sigma^g$ being stationary.
\end{proof}

\begin{lemma}
\label{Lemma ubN to ubFA} 
	$\ub\NP_\kappa \implies \ub\FA_\kappa$
\end{lemma}
\begin{proof}
As for the $\mathrm{club}$ case, except that we insist on $\sigma^g$ being unbounded.
\end{proof}

\begin{lemma}
	$\omega\text{-}\ub\NP_\kappa \implies \omega\text{-}\ub\FA_\kappa$
\end{lemma}
\begin{proof}Define $\sigma$ as in the club case. Define
	\begin{equation*}
		\tau=\{(n,p):n<\omega,p\in E_n\}
	\end{equation*}
where we want to meet all of the dense sets $\langle E_n\mid n<\omega\rangle$ as well as unboundedly many of the dense sets $D_\gamma$. Take $g$ such that $\tau^g=\omega$ and $\sigma^g$ is unbounded.
\end{proof}

We can also get converses for these in the case of $\mathsf{ub}$ and $\omega\text{-}\mathsf{ub}$.

\begin{lemma} \ 
\label{Lemma ubFA to ubN} 
	\begin{enumerate-(1)}
		\item 
		\label{Lemma ubFA to ubN 1} 
		$\ub\FA_\kappa\implies \ub\NP_\kappa$
		\item 
		\label{Lemma ubFA to ubN 2} 
		$\omega\text{-}\ub\FA_\kappa \implies \omega\text{-}\ub\NP_\kappa$
	\end{enumerate-(1)}
\begin{proof}
\ref{Lemma ubFA to ubN 1}: 
Assume $\ub\FA_\kappa$. Let $\sigma$ be a rank $1$ name for an unbounded subset of $\kappa$. For $\gamma<\kappa$ let $D_\gamma$ be the set of all $p\in\PP$ such that for some $\delta>\gamma$, $p\sforces \check{\delta}\in \sigma$. Let $g$ be a filter meeting unboundedly many $D_\gamma$; then $\sigma^g$ is unbounded.
		
\ref{Lemma ubFA to ubN 2}: 
		Let $\sigma$ be a rank $1$ name for an unbounded subset of $\kappa$ and $\tau$ be a good name for $\omega$. Define $D_\gamma$ as above, and for $n<\omega$ let $E_n$ be the set of all $p\in\PP$ which strongly force $n\in \tau$. Find $g$ meeting unboundedly many $D_\gamma$ and every $E_n$; then $\sigma^g$ is unbounded and $\tau^g=\omega$.
\end{proof}
\end{lemma}

This proves every implication in the left two columns of Figure \ref{diagram of implications}.

\subsection{Extremely bounded name principles} 

Now, we address the right most column of Figure \ref{diagram of implications - bounded with finite lambda}. 
These axioms are more interesting if $\PP$ is a complete Boolean algebra, since they can be trivial otherwise. 


\begin{lemma}
	$\BN^1_\kappa$ is provable in $\ZFC$.
\end{lemma}
\begin{proof}
	Let $\sigma$ be a $1$-bounded rank $1$ name such that $1\forces \sigma=\check{A}$ for some set $A$. Then for $\gamma\in\kappa\setminus A$, there is no $p\in \PP$ such that $(\check{\gamma},p)\in \sigma$. For $\gamma \in A$ there is a unique $p\in \PP$ such that $(\check{\gamma},p)\in \sigma$; and $p$ is contained in every generic filter. Assuming $\PP$ is atomless, it follows that $p=1$ and hence that, if we let $g$ be any filter at all, $\sigma^g=A$. 
	It is also possible to adjust this proof to work for forcings with atoms; this is left as an exercise for the reader.
\end{proof}

All of these results also hold if we work with bounded name principles and forcing axioms, provided that the bound is at least $\kappa$.

For bounds below $\kappa$, we can almost get an equivalence between the different bounds for the stationary and unbounded name principles. 
A forcing is called \emph{well-met} if any two compatible conditions $p,q$ have a greatest lower bound $p\wedge q$. 

The next result and proof is due to Hamkins for trees (see Corollary \ref{lemma_failure of lambda-bounded and 1-bounded name principle for trees}). 
We noticed that his proof shows a more general fact. 

\begin{lemma}[with Hamkins] 
\label{lemma_failure of lambda-bounded and 1-bounded name principle} 
Suppose $\lambda<\kappa$ and $\PP$ is well-met. 
	\begin{enumerate-(1)}
		\item If $\stat \BN_{\PP,\kappa}^\lambda$ fails, then there are densely many conditions $p\in \PP$ such that $\stat \BN_{\PP_p,\kappa}^1$ fails, where  $\PP_p:=\{q\in \PP: q\leq p\}$.
		\item The same result holds with $\mathsf{ub}$ in place of $\mathsf{stat}$.
	\end{enumerate-(1)}
\end{lemma}
\begin{proof}
	We prove the $\mathsf{stat}$ case; the $\mathsf{ub}$ case is identical. The key fact the proof uses is that if we partition a stationary/unbounded subset of $\kappa$ into $\lambda<\kappa$ many parts, then one of those parts must be stationary/unbounded.
	
	Let $\sigma$ be a $\lambda$-bounded (rank $1$) name for a stationary set, such that there is no $g\in V$ with $\sigma^g$ stationary. Then, without loss of generality, we can enumerate the elements of $\sigma$:
	
	\begin{equation*}
		\sigma=\{ (\check{\gamma},p_{\gamma,\delta}) : \gamma<\kappa, \delta<\lambda\} 
	\end{equation*}

	For $\delta<\lambda$, we define:

	\begin{equation*}
		\sigma_\delta=\{(\check{\gamma},p_{\gamma,\delta}) : \gamma<\kappa\}
	\end{equation*}

	Clearly, $\sigma_\delta$ is $1$-bounded.
	
	For any generic filter $G$, $\bigcup \sigma_\delta^G=\sigma^G$ is stationary in $V[G]$. Hence, $\PP$ forces ``There is some $\delta<\lambda$ such that $\sigma_\delta$ is stationary.'' Now, let $p\in \PP$ be one of the densely many conditions which decides which $\delta$ this is.
Then 	
	\begin{equation*}
		\sigma_{\delta,p}=\{(\check{\gamma},p_{\gamma,\delta}\wedge p) : \gamma<\kappa\}
	\end{equation*}
is a $1$-bounded $\PP_p$-name and $\PP_p\forces \sigma_{\delta,p}$ is stationary. 
If $\stat \BN_{\PP_p,\kappa}^1$ would hold, there would exist a filter $g$ such that $\sigma_{\delta,p}^g$ is stationary. 
Then $g$ generates a filter $h$ in $\PP$ such that $\sigma_{\delta,p}^h\supseteq \sigma_{\delta,p}^g$  is stationary. 
\end{proof}

\begin{corollary}[Hamkins] 
\label{lemma_failure of lambda-bounded and 1-bounded name principle for trees} 
Suppose that $T$ is a tree, $\PP_T$ is $T$ with reversed order and $\lambda<\kappa$. 
	\begin{enumerate-(1)}
		\item If $\stat \BN_{\PP_T,\kappa}^\lambda$ fails, then there are densely many conditions $p\in \PP$ such that $\stat \BN_{(\PP_T)_p,\kappa}^1$ fails, where  $(\PP_T)_p:=\{q\in \PP_T: q\leq p\}$.
		\item The same result holds with $\mathsf{ub}$ in place of $\mathsf{stat}$.
	\end{enumerate-(1)}
\end{corollary} 

\begin{corollary} 
\label{corollary_equivalence of lambda-bounded and 1-bounded name principles} 
Suppose $\lambda<\kappa$ and $\PP$ is a well-met forcing such that for every $p\in \PP$, $\PP_p$ embeds densely into $\PP$. 
Then 
$$\stat \BN_{\PP,\kappa}^\lambda \Longleftrightarrow \stat \BN_{\PP,\kappa}^1$$ 
$$\ub \BN_{\PP,\kappa}^\lambda \Longleftrightarrow \ub \BN_{\PP,\kappa}^1$$
\end{corollary} 
\begin{proof} 
We show that a failure of $\stat \BN_{\PP,\kappa}^\lambda$ implies the failure of $\stat \BN_{\PP,\kappa}^1$. 
The converse direction is clear and the proof for the unbounded name principles is analogous. 

By Lemma \ref{lemma_failure of lambda-bounded and 1-bounded name principle}, there is some $p\in \PP$ such that $ \stat \BN_{\PP_p,\kappa}^1$ fails. 
Let $i\colon \PP_p \rightarrow \PP$ be a dense embedding and $\QQ:=i(\PP_p)$. 
Since $ \stat \BN_{\QQ,\kappa}^1$ fails, let $\sigma$ be a $1$-bounded $\QQ$-name witnessing this failure. 
We claim that there is no filter $g$ on $\PP$ such that $\sigma^g$ is stationary. 
Assume otherwise. 
Using that $\QQ$ is well-met, let $h$ denote the set of all $q\geq p_0\wedge_\QQ \dots \wedge_\QQ p_n$ for some $p_0,\dots,p_n \in g\cap \QQ$. 
It is easy to check that $h$ is a well-defined filter on $\QQ$ and contains $g\cap \QQ$. 
Then $\sigma^h\supseteq \sigma^g$ is stationary. 
But this contradicts the choice of $\sigma$. 
\end{proof}

\subsection{Extremely bounded forcing axioms} 

We next study forcing axioms for very small predense sets. 
The next lemmas show that $\BFA^\omega_{\PP,\omega_1}$ has some of the same consequences as $\BFA$. 

\begin{lemma} 
If $\PP$ is a complete Boolean algebra such that $\BFA^\omega_{\PP,\omega_1}$ holds, then $1_\PP$ does not force that $\omega_1$ is collapsed. 
\end{lemma} 
\begin{proof} 
Suppose $\forces \dot{f}\colon \omega_1\rightarrow \omega$ is injective. 
Let $A_\alpha=\{ \llbracket \dot{f}(\alpha)=n \rrbracket\neq 0 \mid n\in\omega\}$. 
Since each $A_\alpha$ is a maximal antichain, there is a filter $g$ with $g\cap A_\alpha\neq\emptyset $ for all $\alpha<\omega_1$. 
Define $f'\colon \omega_1\rightarrow \omega$ by letting $f'(\alpha)=n$ if $\llbracket \dot{f}(\alpha)=n \rrbracket \in g$ for all $\alpha<\omega_1$. 
Since $g$ is a filter, $f'\colon \omega_1\rightarrow \omega$  is well-defined and injective. 
\end{proof} 

\begin{lemma} 
If $\PP$ is a complete Boolean algebra such that $\BFA^\omega_{\PP,\omega_1}$ holds and $\PP$ adds a real, then $\CH$ fails. 
\end{lemma} 
\begin{proof} 
Suppose $\CH$ holds and let $\langle x_\alpha\mid \alpha<\omega_1\rangle$ be an enumeration of all reals. Let $\sigma$ be a name for the real added by $\PP$.
For $\alpha<\omega_1$, let
\begin{equation*}
D_\alpha = \{ \llbracket t^\smallfrown \langle n\rangle \subseteq \sigma\rrbracket \colon t\in 2^{<\omega}, n\in 2, t\subseteq x_\alpha, t^\smallfrown \langle n\rangle \not \subseteq x_\alpha\}
\end{equation*}
For $n<\omega$, let
\begin{equation*}
E_n=\{\llbracket \sigma(n)=m\rrbracket \mid m\in 2\}
\end{equation*} 
Then the $D_\alpha$ and $E_n$ are all predense and countable. 
Take a filter $g$ which meets every $D_\alpha$ and $E_n$. The $E_n$ ensure that $g$ defines a real $x$ (by $x(n)=m$ where $\llbracket \sigma(n)=m\rrbracket \in g$). But if $x=x_\alpha$ then $g\cap D_\alpha=\emptyset$.
\end{proof} 

There exist forcings $\PP$ such that the implication $\BFA^\omega_{\PP,\omega_1} \Rightarrow \BFA^{\omega_1}_{\PP,\omega_1}$ fails. 
To see this, suppose that $\QQ$ is a forcing such that $\BFA^{\omega_1}_{\QQ,\omega_1}$ fails. 
Let $\PP$ be a lottery sum of $\omega_1$ many copies of $\QQ$. 
Since $\BFA^{\omega_1}_{\QQ,\omega_1}$ fails, $\BFA^{\omega_1}_{\PP,\omega_1}$ fails as well. 
On the other hand, $\BFA^\omega_{\PP,\omega_1}$ holds trivially since any countable predense subset of $\PP$ contains $0_\PP$.  

\begin{question} 
Does the implication $\BFA^\omega_{\PP,\omega_1} \Rightarrow \BFA^{\omega_1}_{\PP,\omega_1}$ hold for all complete Boolean algebras $\PP$? 
\end{question}   

By the previous lemmas, any forcing which is a counterexample cannot force that $\omega_1$ is collapsed, and if it adds reals then $\CH$ holds.

\subsection{Basic results on $\ub\FA$} 

In this section, we collect some observations about weak forcing axioms. 
We aim to prove some consequences of these axioms.  
We first consider $\ub\FA$ and $\stat\FA$. 
How strong is $\ub\FA$? 
The next lemmas show that is has some of the same consequences as $\FA$. 

\begin{lemma} 
If $\ub\FA_{\PP,\omega_1}$ holds, then $\PP$ does not force that $\omega_1$ is collapsed. 
\end{lemma} 
\begin{proof} 
Towards a contradiction, suppose $\PP$ forces that $\omega_1$ is collapsed. 
Let $\dot{f}$ be a $\PP$-name for an injective function $ \omega_1\rightarrow \omega$. 
For $\alpha<\omega_1$, let $D_\alpha=\{ p \in \PP \mid \exists n\in\omega\ p\Vdash \dot{f}(\alpha)=n \}$. 
By $\ub\FA_{\PP,\omega_1}$, there is a filter $g$ and an unbounded subset $A$ of $\omega_1$ such that $g\cap D_\alpha\neq\emptyset$ for all $\alpha\in A$. 
Define $f\colon A\rightarrow \omega$ by letting $f(\alpha)=n$ if there is some $p\in g\cap D_\alpha$ with $p\Vdash \dot{f}(\alpha)=n$. 
Since $g$ is a filter, $f$ is injective. 
\end{proof} 

\begin{lemma} 
\label{Lemma - ubFA and no new reals imply stationary set preserving} 
If $\ub\FA_{\PP,\omega_1}$ holds and $\PP$ does not add reals, then for each stationary subset $S$ of $\omega_1$, $\PP$ does not force that $S$ is nonstationary. 
\end{lemma} 
\begin{proof} 
Suppose that $\dot{C}$ is a name for a club such that $\Vdash_\PP S\cap \dot{C}=\emptyset$. 
Let $\dot{f}$ be a name for the characteristic function of $\dot{C}$. 
For each $\alpha<\omega_1$, 
$$D_\alpha=\{ p\in \PP \mid \exists t \in 2^\alpha\ t \subseteq \dot{f}\}$$ 
is dense in $\PP$, since $\PP$ does not add reals. 
By $\ub\FA_{\PP,\omega_1}$, there is a filter $g$ and an unbounded subset $A$ of $\omega_1$ such that $g\cap D_\alpha\neq\emptyset$ for all $\alpha\in A$. 
Since $g$ is a filter, $C:=\{\alpha<\omega_1 \mid \exists p\in g\ p\Vdash \alpha \in \dot{C} \}$ is a club in $\omega_1$ with $S\cap C\neq\emptyset$. 
\end{proof}

The previous lemma also follows from Theorem \ref{Bagaria's characterisation} and Lemma \ref{ubFA implies BFA} below via an absoluteness argument, assuming $\PP$ is a homogeneous complete Boolean algebra. 
It is open whether the lemma holds for forcings $\PP$ which add reals. 

What is the relationship between $\ub\FA_{\PP,\omega_1}$ and other forcing axioms? 
We find two opposite situations. 
For any $\sigma$-centred forcing, $\ub\FA_{\PP,\omega_1}$ and $\stat\FA_{\PP,\omega_1}$ are provable in $\ZFC$ by Lemma \ref{Lemma stat-NP for sigma-centred} below.  
For many other forcings though, $\ub\FA_{\PP,\omega_1}$ implies nontrivial axioms such as $\FA_{\PP,\omega_1}$ or $\BFA^{\omega_1}_{\PP,\omega_1}$. 
For instance, the implication $\ub\FA_{\PP,\omega_1}$ $\Rightarrow$ $\FA_{\PP,\omega_1}$ holds for all $\sigma$-distributive forcing by Lemma \ref{ubFA implies FA for sigma-distributive forcings} below. 
We will further see in Lemma \ref{ubFA implies BFA} below that for any complete Boolean algebra $\PP$ which does not add reals, $(\forall q\in \PP\ \ub\FA_{\PP_q,\omega_1})$ implies $\BFA^{\omega_1}_{\PP,\omega_1}$. 
Moreover, the implication $\ub\FA_{\PP,\omega_1}$ $\Rightarrow$ $\FA_{\PP,\omega_1}$ also holds for some forcings that add reals, for instance for random forcing by Lemma \ref{Lemma_Random ubFA}. 

We do not have any examples of forcings where $\ub\FA_{\PP,\omega_1}$ and $\stat\FA_{\PP,\omega_1}$ sit between these two extremes: strictly weaker than $\FA_{\PP,\omega_1}$, but not provable in ZFC.
 
In particular, we have not been able to separate the two axioms: 

\begin{question} 
\label{Question ubFA versus statFA} 
Can forcings $\PP$ exist such that $\ub\FA_{\PP,\kappa}$ holds, but $\stat\FA_{\PP,\kappa}$ fails? 
\end{question} 

For instance, we would like to know if these axioms hold for the following forcings: 

\begin{question} 
\label{Question Baumgartner's forcing} 
Do Baumgartner's forcing to add a club in $\omega_1$ with finite conditions \cite[Section 3]{baumgartner1984applications} and Abraham's and Shelah's forcing for destroying stationary sets  with finite conditions \cite[Section 2]{abraham1983forcing} satisfy $\ub\FA_{\PP,\omega_1}$ and $\stat\FA_{\PP,\omega_1}$? 
\end{question}

\subsection{Characterisations of $\FA^+$ and $\FA^{++}$}

The proof of the equivalence of $\FA$ and $\NP$ still goes through fine if we change the axioms slightly, demanding some extra property to be true of the filter $g$ we're looking for. This gives us a nice way to express $\FA^+$ and $\FA^{++}$.

\begin{lemma}
\label{Lemma_characterisation_of_FA+} 
	$\FA_{\mathcal{C},\kappa}^+$ is equivalent to the following statement: 
	\begin{quote} 
	For all $\mathbb{P}\in \mathcal{C}$, for all rank $1$ names $\sigma$ and $\tau$ for subsets of $\kappa$ such that $\PP$ forces ``$\sigma=\check{A}$'' for some $A$ and $``\tau$ is stationary'', there is some filter $g$ with $\sigma^g=A$ and $\tau^g$ stationary. 
	\end{quote} 
	Similarly, $\FA_{\mathcal{C},\kappa}^{++}$ is equivalent to being able to correctly interpret $\kappa$ many stationary rank $1$ names and a single rank $1$ name for a specific set $A$.
\end{lemma}
\begin{proof}
	Analogous to the proof of \ref{Lemma_FA iff N} in the previous section.
\end{proof}

In the case of $\FA^{++}$ this result can be sharpened further, getting rid of the name for $A$:

\begin{lemma}
	$\FA_{\mathcal{C},\kappa}^{++}$ is equivalent to the statement: 
	\begin{quote} 
	For all collections of $\kappa$ many rank $1$ names $\langle \sigma_\gamma\mid \gamma<\kappa\rangle$ with $\mathbb{P} \forces ``\sigma_\gamma \text{ is stationary for all }\gamma"$, there is a filter $g\in V$ such that for all $\gamma$, $\sigma_\gamma^g$ is stationary. 
	 \end{quote} 
\end{lemma}
\begin{proof}
$\Rightarrow$: By the previous lemma.

$\Leftarrow$: Let $\sigma$ be a rank $1$ name, such that $\PP\forces \sigma=\check{A}$ for some $A\subseteq \kappa$. We claim there is a collection $\langle \tau_\gamma\mid \gamma<\kappa\rangle$ of rank $1$ names, which are forced to be stationary in $\kappa$, such that any filter $g$ which interprets every $\tau_\gamma$ as stationary will interpret $\sigma$ as $A$. Once we have proved this claim, the lemma follows immediately from the second part of Lemma \ref{Lemma_characterisation_of_FA+}. 
For $\gamma \in A$, let $\tau_\gamma=\{( \check{\alpha},p) : \alpha \in \kappa, p\sforces \check{\gamma} \in \sigma\}$. For $\gamma \not \in A$, let $\tau_\gamma=\check{\kappa}$. 
We will see that $\PP\forces ``\tau_\gamma =\kappa"$ for $\gamma\in A$. 
Note that $\PP\forces \sigma=\check{A}$ by assumption. 
So for $\gamma\in A$, every generic filter will contain some $p$ with $p\sforces \check{\gamma} \in \sigma$. Hence $\PP \forces \tau_\gamma=\check{\kappa}$. 
There is a filter $g$ such that  $\tau_\gamma^g$ is stationary for all $\gamma<\kappa$ by assumption. 
If $\gamma \in A$, then in particular $\tau_\gamma^g\neq \emptyset$. Hence $\gamma \in \sigma^g$. 
If a filter interprets all the $\tau_\gamma$ as stationary sets, then $\sigma^g\supseteq A$.  If $\gamma \in \sigma^g\setminus A$, then there is some $p\in \PP$ with $\langle \check{\gamma},p\rangle \in \sigma$, which is impossible as $\PP\forces \check{\gamma}\not \in \sigma$.
\end{proof}

\section{A correspondence for arbitrary ranks}
\label{Section correspondence and applications} 

We now move on to discuss higher ranked name principles, including those of the ranked or unranked simultaneous variety. It turns out that even at high ranks, a surprising variety of these are equivalent to one another and to a suitable forcing axiom. These are summarised in the following theorems.

\subsection{The correspondence}
\label{Section_correspondence} 

\begin{theorem}
\label{correspondence forcing axioms name principles} 
	Let $\PP$ be a forcing and let $\kappa$ be a cardinal. 
	The following implications hold, given the assumptions noted at the arrows: 

	\begin{enumerate-(1)}
		\item 
		\label{correspondence forcing axioms name principles 1} 
\[ \xymatrix@R=1em{ 
& & \FA_\kappa  \ar@/_-0.3cm/@{->}[rdd] & & \\ 
& & & & \\ 
& \ \ \ \ \ \ \NP_{\PP,\kappa}(\infty)  \ar@{<-}[rr] \ar@/_-0.3cm/@{->}[ruu]  & &  \fo\NP_{\PP,\kappa}(\infty) & 
}\] 		
		\item 
		\label{correspondence forcing axioms name principles 2} 
For any ordinal $\alpha>0$, and any transitive set $X$ of size at most $\kappa$: \footnote{Recall that $\NP_{\PP,X,\kappa}(\alpha)$ is only defined if $X$ has size at most $\kappa$.}  
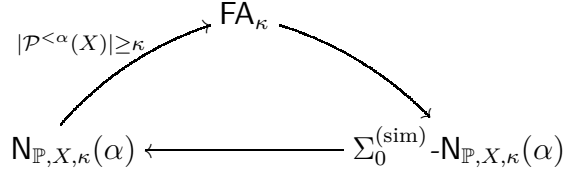
\[ \xymatrix@R=1em{ 
& & \FA_\kappa  \ar@/_-0.3cm/@{->}[rdd]
 & & \\ 
& & & & \\ 
&  \ \ \ \ \ \  \NP_{\PP,X,\kappa}(\alpha) \ar@{<-}[rr] \ar@/_-0.3cm/@{->}[ruu]^{\lvert \mathcal{P}^{<\alpha}(X)\rvert \geq\kappa} & &  \fo\NP_{\PP,X,\kappa}(\alpha) & 
}\] 
	\end{enumerate-(1)}		
\end{theorem}

As usual, we can generally think of $X$ as being a cardinal.

There is also a bounded version of this theorem.

\begin{theorem}
\label{correspondence bounded forcing axioms name principles} 
	Let $\PP$ be a complete Boolean algebra, and let $\kappa,\lambda$ be cardinals.
	The following implications hold, given the assumptions noted at the arrows: 
	\begin{enumerate-(1)}
		\item 
\[ \xymatrix@R=1em{ 
& & \BFA^\lambda_\kappa  \ar@/_-0.3cm/@{->}[rdd]^{\kappa\leq \lambda} & & \\ 
& & & & \\ 
&  \ \ \ \ \ \   \BN_{\PP,\kappa}^{\lambda}(\infty) \ar@{<-}[rr] \ar@/_-0.3cm/@{->}[ruu] & &  \fo\BN_{\PP,\kappa}^{\lambda}(\infty) & 
}\] 
		\item 
		For any ordinal $\alpha>0$, and transitive set $X$ of size at most $\kappa$: 
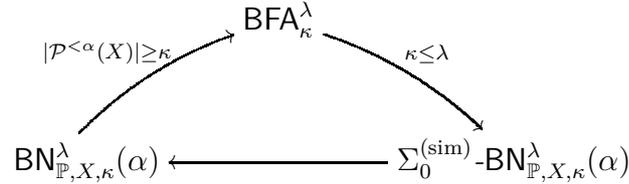
\[ \xymatrix@R=1em{ 
& & \BFA^\lambda_\kappa  \ar@/_-0.3cm/@{->}[rdd]^{\kappa\leq\lambda} 
& & \\ 
& & & & \\ 
&  \ \ \ \ \ \  \BN_{\PP,X,\kappa}^\lambda(\alpha)  \ar@{<-}[rr] \ar@/_-0.3cm/@{->}[ruu]^{\lvert \mathcal{P}^{<\alpha}(X)\rvert \geq\kappa} & & \fo\BN_{\PP,X,\kappa}^\lambda(\alpha)   & 
}\] 
	\end{enumerate-(1)}		
\end{theorem}

\begin{remark} 
For the $\infty$ case it suffices to look only at $\emptyset$ names, as we discussed after Definition \ref{Defn_nameprinciple}. 
Moreover, for the implication $\NP_{\PP,\kappa}(\infty) \Rightarrow\FA_{\PP,\kappa}$ (and the corresponding ones in the other diagrams), we need only rank $1$ $\kappa$-names for $\kappa$. 
These can be understood as rank $\kappa$ $\emptyset$-names for $\kappa$.  
For $\NP_{\PP,X,\kappa}(\alpha) \Rightarrow\FA_{\PP,\kappa}$, rank $1$ $Y$-names for a fixed set $Y$ of size $\kappa$ suffice. 
These can be understood as rank ${\leq}\alpha$ $X$-names. 
These remarks are also true for the bounded versions. 
Note that for $\NP_{\PP,\kappa}(1) \Rightarrow\FA_{\PP,\kappa}$, rank $1$ $\kappa$-names for $\kappa$ suffice by Lemma \ref{Lemma_FA iff N}. 
\end{remark} 

We give some simple instances of Theorem \ref{correspondence forcing axioms name principles} \ref{correspondence forcing axioms name principles 2} and postpone the proofs to Section \ref{Section proofs}. 
The variant for bounded forcing axioms has similar consequences. 
The next result follows by letting $\kappa=X$ and $\alpha=1$. 
\begin{corollary} 
For any forcing $\PP$, $\FA_{\PP,\kappa}$ $\Longleftrightarrow$ $\fo\NP_{\PP,\kappa}$ $\Longleftrightarrow$ $\NP_{\PP,\kappa}$. 
\end{corollary} 

To illustrate this, we note how some concrete forcing axioms can be characterized by name principles. 
For example, we can characterize $\PFA$ as follows: 
$$\PFA \Longleftrightarrow \fo\NP_{\mathrm{proper},\omega_1} \Longleftrightarrow \NP_{\mathrm{proper},\omega_1}.$$ 
In other words, rank $1$ names for $\omega_1$ can be interpreted correctly. 

For higher ranks, it is useful to choose $\alpha$, $\kappa$ and $X$ such that $|\mathcal{P}^{<\alpha}(X)\rvert \geq\kappa$ holds to get an equivalence  in Theorem \ref{correspondence forcing axioms name principles} \ref{correspondence forcing axioms name principles 2}. 
This condition holds for $\kappa\geq 2^\omega$, $X=\omega$ and $\alpha=2$. 

\begin{corollary} 
For any cardinal $\kappa\geq 2^\omega$ and any forcing $\PP$, we have $\FA_{\PP,\kappa}$ $\Longleftrightarrow$ $\fo\NP_{\PP,\omega,\kappa}(2)$ $\Longleftrightarrow$ $\NP_{\PP,\omega,\kappa}(2)$. 
\end{corollary} 

For example, we can characterize $\PFA$ as follows: 
$$\PFA \Longleftrightarrow \fo\NP_{\mathrm{proper},\omega,\omega_1}(2) \Longleftrightarrow \NP_{\mathrm{proper},\omega,\omega_1}(2).$$ 
In other words, rank $2$ names for sets of reals can be interpreted correctly. 
We leave open how to characterise higher rank (e.g. rank $2$) principles for names for reals.

\subsection{The proofs}
\label{Section proofs} 

\begin{proof}[Proof of Theorem \ref{correspondence forcing axioms name principles}] 
We prove both parts of the theorem simultaneously, by fixing $X$ and $\alpha$ and proving all the implications in the following diagram: 

\smallskip 

	\[ \xymatrix@R=1em{ 
 & \fo\NP_{\PP,\kappa}(\infty) \ar@{->}[r] \ar@{->}[dd]
&  \NP_{\PP,\kappa}(\infty) \ar@{->}[rd] \ar@{->}[dd]
& \\ 
 \FA_{\PP,\kappa} \ar@{->}[ru] \ar@{->}[rd] & & & \FA_{\PP,\kappa} \\ 
\labelmargin{10pt}
 &  \fo\NP_{\PP,X,\kappa}(\alpha)  \ar@{->}[r]  &  \NP_{\PP,X,\kappa}(\alpha)  \ar@{->}[ru]_{\ \  \lvert \mathcal{P}^{<\alpha}(X)\rvert\geq\kappa} & 
}\] 

\bigskip 

Of these, the first $\FA_{\PP,\kappa}\Rightarrow \fo\NP_{\PP,\kappa}(\infty)$  is the hardest to prove, and the main work on the theorem. We'll leave it to the end, and prove the other implications first. 
Note that $\FA_{\PP,\kappa}\Rightarrow \fo\NP_{\PP,X,\kappa}(\alpha)$ follows from the rest of the diagram. 

	\begin{proof}[Proof of $\fo\NP_{\PP,\kappa}(\infty)\Rightarrow \fo\NP_{\PP,\kappa,X}(\alpha)$]
The latter is a special case of the former. 
	\end{proof}
	\begin{proof}[Proof of $\NP_{\PP,\kappa}(\infty)\Rightarrow \NP_{\PP,X,\kappa}(\alpha)$] 
	Again, this is a special case. 
	\end{proof}
	\begin{proof}[Proof of $\fo\NP_{\PP,X,\kappa}(\alpha)\Rightarrow \NP_{\PP,X,\kappa}(\alpha)$]
		Given a $\kappa$-small name $\sigma$ of rank $\alpha$ or less, and a set $A$ as called for by $\NP_{\PP,\kappa}(\alpha)$, we know $A\in \mathcal{P}^\alpha(X)\cap H_{\kappa^+}$. Hence $\check{A}$ is a $\kappa$ small $\alpha$ rank $X$ name, so ``$\sigma=\check{A}$'' is one of the formulas discussed by the simultaneous name principle.
	\end{proof}
	\begin{proof}[Proof of $\fo\NP_{\PP,\kappa}(\infty)\Rightarrow \NP_{\PP,\kappa}(\infty)$]
		Similar to the previous proof: if $\sigma$ is any $\kappa$-small name, and $A\in H_{\kappa^+}$ is such that $\PP\forces \sigma=\check{A}$, then since $\check{A}$ is $\kappa$-small we know from $\fo\NP(\infty)$ that we can find a filter $g$ such that $\sigma^g=\check{A}^g=A$.
	\end{proof}
	\begin{proof}[Proof of $\NP_{\PP,X,\kappa}(\alpha)\Rightarrow \FA_{\PP,\kappa}$]
		We assume $\lvert P^{<\alpha}(X)\rvert \geq \kappa$. The idea is similar to the proof of $\NP_{\PP,\kappa}\Rightarrow \FA_{\PP,\kappa}$ from Lemma \ref{Lemma_FA iff N}, but first we must prove a technical claim.
		
		\begin{claim}
			$\mathcal{P}^{<\alpha}(X)$ contains at least $\kappa$ many elements whose check names are $\kappa$-small ${<}\alpha$-rank $X$-names.
		\end{claim}
		\begin{proof}[Proof (Claim)]
			Let $\alpha'\leq \alpha$ be minimal such that $\lvert \mathcal{P}^{<\alpha'}(X)\rvert \geq \kappa$.
			
			Let $A\in \mathcal{P}^{<\alpha'}(X)$. Then $A\in \mathcal{P}^\epsilon(X)$ for some $\epsilon<\alpha'$. We show by induction on $\epsilon$ that $\check{A}$ is in fact a $\kappa$-small $\epsilon$-rank $X$-name. From this and the assumption on the size of $\kappa$, it of course follows that there are at least $\kappa$ many elements of $\mathcal{P}^{<\alpha'}(X)\subseteq \mathcal{P}^{<\alpha}(X)$ whose check names are $\kappa$-small $<\alpha$-rank $X$-names.
			
			The case $\epsilon=0$ is trivial. Suppose $\epsilon>0$.	By inductive hypothesis, we know that all the names which are contained in $\check{A}$ are $\kappa$-small $<\epsilon$-rank $X$-names. It remains to check that there are at most $\kappa$ many of them; that is, that $\lvert A \rvert \leq \kappa$. But this is obvious, since $A\subseteq \mathcal{P}^{<\epsilon}(X)$ and $\lvert \mathcal{P}^{<\epsilon}(X)\rvert < \kappa$ by our choice of $\alpha'$.
		\end{proof}
		
		Given the claim, we can now take a set of $\kappa$ many distinct sets $A:=\{A_\gamma: \gamma<\kappa\}\subseteq \mathcal{P}^{<\alpha}(X)$, such that for all $\gamma$, the name $\check{A}_\gamma$ is a $\kappa$ small ${<}\alpha$ rank $X$-name.
		
		Let $\langle D_\gamma\rangle_{\gamma<\kappa}$ be a sequence of dense sets in $\PP$. We define a name $\sigma$:
		\begin{equation*}
			\sigma=\{\langle \check{A}_\gamma,p\rangle : \gamma<\kappa, p\in D_\gamma\}
		\end{equation*}
		Then $\sigma$ is a $\kappa$-small ${\leq}\alpha$-rank $X$-name, and $\PP\forces \sigma=\check{A}$. Hence, if we assume $\NP_{\PP,X,\kappa}(\alpha)$ we can choose a filter $g$ such that $\sigma^g=A$. It is easy to see that $g$ must meet every $D_\gamma$.
	\end{proof}
	\begin{proof}[Proof of $\NP_{\PP,\kappa}(\infty)\Rightarrow \FA_{\PP,\kappa}$]
		Essentially the same as the previous proof, but since we're no longer required to make sure $\sigma$ has rank $\alpha$ we can omit the technical claim and just take $A_\gamma:=\gamma$ for all $\gamma<\kappa$.
	\end{proof}
\begin{proof}[Proof of $\FA_{\PP,\kappa}\Rightarrow \fo\NP_{\PP,\kappa}(\infty)$]
	This is the main work of the theorem. By a delicate series of inductions, we will prove the following lemma:

	\begin{lemma}
	\label{collection of dense sets to witness first order statement} 
		Let $\varphi(\vec{\sigma})$ be a $\Sigma_0$ formula where $\vec{\sigma}$ is a tuple of $\kappa$-small names. Then there is a collection $\mathcal{D}_{\varphi(\vec{\sigma})}$ of at most $\kappa$ many dense sets, which has the following property: if $g$ is any filter meeting every set in $\mathcal{D}_{\varphi(\vec{\sigma})}$ and $g$ contains some $p$ such that $p\forces\varphi(\vec{\sigma})$, then in fact $\varphi(\vec{\sigma}^g)$ holds in $V$.
	\end{lemma}
	The result we're trying to show follows easily from this lemma: Fix a tuple $\vec{\sigma}=\langle \sigma_0,\ldots,\sigma_n\rangle$ of $\kappa$ small names, and let $\mathcal{D}:=\bigcup\{\mathcal{D}_{\varphi(\vec{\sigma})} : \varphi(v_0,\ldots,v_n) \text{ is }\Sigma_0\}$. $\mathcal{D}$ is a collection of at most $\kappa$ many dense sets. Using $\FA_{\PP,\kappa}$, take a filter $g$ meeting every dense set in $\mathcal{D}$. If $\varphi(v_0,\ldots,v_n)$ is a $\Sigma_0$ formula and $1\forces \varphi(\vec{\sigma})$ then since $1\in g$ we know that $\varphi(\vec{\sigma}^g)$ holds.
	
	We will work our way up to proving the lemma, by first proving it in simpler cases.
	We opt for a direct proof of the name principle $\NP_{\PP,\kappa}(\infty)$ in the next Claim \ref{claim sigma equal emptyset}.  
	This and Claim \ref{claim sigma neq emptyset} could be replaced by shorter arguments for $\kappa$-small $\emptyset$-names, since it suffices to deal with $\fo\NP_{\PP,\emptyset,\kappa}(\infty)$ as discussed after Definition \ref{Defn_nameprinciple}. 
	
	\begin{claim}\label{sigma=A}
	\label{claim sigma equal emptyset} 
		The lemma holds when $\varphi$ is of the form $\sigma=\check{A}$ for some set 
		$A\in H_{\kappa^+}$ and ($\kappa$-small) name $\sigma$.
	\end{claim}
	Note that since $A\in H_{\kappa^+}$, we know that $\check{A}$ is a $\kappa$ small name. So the statement in the claim does make sense.
	\begin{proof}
		We use induction on the rank of $\sigma$. If $\sigma$ is rank $0$ then it is a check name, and so the lemma is trivial: we can just take $\mathcal{D}_{\sigma=\check{A}}=\emptyset$.	So say $\sigma$ is rank $\alpha>0$ and the lemma is proved for all names of rank ${<}\alpha$. Since $\sigma$ is $\kappa$-small, we can write $\sigma=\{(\sigma_\gamma,p): \gamma<\kappa, p\in S_\gamma\}$ for some $\kappa$-small names $\sigma_\gamma$ and sets $S_\gamma\subseteq \PP$.
		
		First, let $B\in A$. We shall define a set $D_B$, whose ``job'' is to ensure $B$ ends up in $\sigma^g$.
		\begin{equation*}
		D_B=\big\{p\in \mathbb{P} : \big(p\forces \sigma \neq \check{A}\big)\vee\big(\exists \gamma<\kappa \: (p\forces \sigma_\gamma=\check{B})\wedge (p\sforces \sigma_\gamma \in \sigma)\big) \big\}
		\end{equation*}
		$D_B$ is dense: if we take $p\in \PP$ then either we can find $r\leq p$ with $r\forces \sigma \neq \check{A}$, or else $p\forces \sigma = \check{A}$. In the first case, we're done. In the second, given any (truly) generic filter $G$ containing $p$, there will be some $\gamma<\kappa$ and $q\in G$ such that\footnote{Note the somewhat delicate nature of this statement: we cannot first take an arbitrary $\gamma$ such that $\sigma_\gamma^G=B$ then try to find $q$ such that $q\sforces \sigma_\gamma\in \sigma$.} $\sigma_\gamma^G=B$ and $(\sigma_\gamma,q)\in \sigma$, so $q\sforces \sigma_\gamma \in \sigma$. Take $r\in G$ such that $r\forces \sigma_\gamma=\check{B}$, and take $s$ below $p,q$ and $r$ by compatibility; then $s\in D_B$.
		
		Now let $\gamma<\kappa$. In a similar way, we define a set $E_\gamma$, which is designed to ensure that $\sigma_\gamma$ ends up in $A$ if it's going to be in $\sigma$.
		\begin{equation*}
		E_\gamma=\big\{p\in \PP: (p\forces \sigma \neq \check{A}) \vee (p\forces \sigma_\gamma \not \in \sigma) \vee \big(\exists B\in A, p\forces \sigma_\gamma=\check{B}\big)\big\}
		\end{equation*}
	
		Again, $E_\gamma$ is dense: 
		Let $p\in \PP$. We can assume that $p\forces \sigma=\check{A}$ and $p\forces \sigma_\gamma\in \sigma$; otherwise we're done immediately. But now we can strengthen $p$ to some $r\leq p$ which forces $\sigma_\gamma\in \check{B}$ for some $B\in A$ and again we're done.
	    
		We define
		\begin{equation*}
		\mathcal{D}_{\sigma=\check{A}}:=\{D_B\colon B\in A\}\cup \{E_\gamma \colon \gamma<\kappa\} \cup \bigcup_{\gamma<\kappa}\bigcup_{B\in A} \mathcal{D}_{\sigma_\gamma=\check{B}}
		\end{equation*}
		
		Every $\sigma_\gamma$ is a $\kappa$-small name of rank less than $\alpha$, and every $B\in H_{\kappa^+}$, so this is well defined by inductive hypothesis.	By assumption, $\lvert A \rvert \leq \kappa$. Hence $\mathcal{D}_{\sigma=\check{A}}$ contains at most $\kappa$ many dense sets.	Fix a filter $g$ which meets every element of $\mathcal{D}_{\sigma=\check{A}}$, and which contains some $p$ forcing $\sigma=\check{A}$. We must verify that $\sigma^g=A$.
		
		First, let $B\in A$. Find $q\in g\cap D_B$, and without loss of generality say $q\leq p$. Then clearly $q \forces \sigma =\check{A}$, so (by definition of $D_B$) we can find $\gamma$ such that $q\forces \sigma_\gamma=\check{B}$ and $q\sforces \sigma_\gamma\in \sigma$. The latter means that $\sigma_\gamma^g\in \sigma^g$. Since $g$ also meets every element of $\mathcal{D}_{\sigma_\gamma=\check{B}}$, the fact that $q\in g$ forces $\sigma_\gamma=\check{B}$ implies that $\sigma_\gamma^g=\check{B}^g=B$. Hence $B\in \sigma^g$.
		
		Now let $B\in \sigma^g$. Then we can find $\gamma<\kappa$ such that $B=\sigma_\gamma^g$ and such that for some $q\in g$ we have $q\sforces \sigma_\gamma\in \sigma$. Without loss of generality, say $q\leq p$. Then $q\forces \sigma=\check{A}$. Let $r\in g\cap E_\gamma$, and again without loss of generality say $r\leq q$. Then for some $B'\in A$, $r\forces \sigma_\gamma=\check{B}'$. Since $g$ meets every element of $\mathcal{D}_{\sigma_\gamma=\check{B}'}$, this tells us that $\sigma_\gamma^g=B'$. But then $B=\sigma_\gamma^g=B'\in A$.
		
		Hence $\sigma^g=A$ as required.
	\end{proof}
	Next, we go up one step in complexity, by allowing both sides of the equality to be nontrivial.
	\begin{claim}
		The lemma holds when $\varphi$ has the form $\sigma=\tau$ for two ($\kappa$-small) names $\sigma$ and $\tau$.
	\end{claim}
	\begin{proof}
		We use induction on the ranks of $\sigma$ and $\tau$. Without loss of generality, let us assume the rank of $\sigma$ is $\alpha$, and the rank of $\tau$ is $\leq\alpha$. If $\rank(\tau)=0$ then $\tau$ is a check name. Since $\tau$ is $\kappa$-small, it can only be a check name for some $A\in H_{\kappa^+}$, so we are already done by the previous claim. So suppose $\rank(\sigma)=\alpha \geq \rank(\tau)>0$, and the result is proven for all $\tau',\sigma'$ where $\rank(\sigma')<\rank(\sigma)$ and $\rank(\tau')<\rank(\tau)$.
		
		Let us write $\sigma=\{(\sigma_\gamma,p):\gamma<\kappa,p\in S_\gamma\}$ and $\tau=\{(\tau_\delta,q): \delta<\kappa,q\in T_\delta\}$.
		
		For $\gamma \in \kappa$, we define a set $D_\gamma$, whose job is to ensure that if $\sigma_\gamma$ ends up being put in $\sigma$ by $g$, then it will also be equal to some element of $\tau$.
		
		\begin{align*}
		D_\gamma=\Big\{p\in \PP: &(p\forces \sigma \neq \tau) \vee (p\forces \sigma_\gamma\not\in \sigma) \\
		&\vee\exists \delta<\kappa\Big((p\forces\sigma_\gamma=\tau_\delta) \wedge (p\sforces\tau_\delta\in\tau)\Big)\Big\}
		\end{align*}
		We claim $D_\gamma$ is dense: Let $p\in \PP$. If $p\not \forces \sigma_\gamma\in \sigma$ or $p\not \forces \sigma=\tau$ then take some $q\leq p$ forcing the converse of one of these statements, and we are done. If $p\forces \sigma_\gamma\in \sigma\wedge \sigma=\tau$ then take a generic filter $G$ containing $p$. We know $\sigma_\gamma^G\in \tau^G$, so $\sigma_\gamma^G=\tau_\delta^G$ for some $\tau_\delta$ which is strongly forced to be in $\tau$ by some $q\in G$. Then take $r\in G$ below $p$ and $q$, and we know $r\forces \sigma_\gamma=\tau_\delta$ and $r\sforces\tau_\delta\in \tau$. Hence $r\in D_\gamma$.
		
		Symmetrically, for $\delta<\kappa$ let 
		\begin{align*}
		E_\delta=\Big\{p\in \PP: &(p\forces \sigma \neq \tau) \vee (p\forces \tau_\delta\not\in \tau)\\
		&\vee\exists \gamma<\kappa \Big((p\forces\sigma_\gamma=\tau_\delta) \wedge (p\sforces\sigma_\gamma\in\sigma)\Big)\Big\}
		\end{align*}
		Again, $E_\delta$ is dense.
		
		We now let
		\begin{equation*}
		\mathcal{D}_{\sigma=\tau}:=\{D_\gamma: \gamma<\kappa\}\cup \{E_\delta:\delta<\kappa\} \cup \bigcup_{\gamma,\delta<\kappa} \mathcal{D}_{\sigma_\gamma=\tau_\delta}
		\end{equation*}
		Note that for all $\sigma,\delta<\kappa$, we know $\rank(\sigma_\gamma)<\rank(\sigma)$ and $\rank(\tau_\delta)<\rank(\tau)$, so $\mathcal{D}_{\sigma_\gamma=\tau_\delta}$ is already defined. Clearly, $\mathcal{D}_{\sigma=\tau}$ contains at most $\kappa$ many dense sets. Let $g$ be a filter meeting every element of it, and let $p\in g$ force $\sigma=\tau$.
		
		Suppose $B\in \sigma^g$. Then for some $q\in g$ and $\gamma<\kappa$, $B=\sigma_\gamma^g$ and $q\sforces \sigma_\gamma\in \sigma$ (and hence $q\forces \sigma_\gamma \in \sigma$). We can also find some $r\in g\cap D_\gamma$. Without loss of generality, say $r$ is below both $p$ and $q$. Certainly $r$ cannot force $\sigma \neq \tau$, nor that $\sigma_\gamma \not \in \sigma$. Hence, for some $\delta<\kappa$, we know $r\forces \sigma_\gamma=\tau_\delta$ and $r\sforces\tau_\delta\in \tau$. But then $\tau_\delta^g\in \tau^g$, and since $g$ meets every element of $\mathcal{D}_{\sigma_\gamma=\tau_\delta}$, we also know that $B=\sigma_\gamma^g=\tau_\delta^g$. Hence $B\in \tau$.
		
		Hence $\sigma^g\subseteq \tau^g$, and by a symmetrical argument $\tau^g\subseteq \sigma^g$. 
	\end{proof}

	\begin{claim}The lemma holds when $\varphi$ has the form $\tau\in\sigma$.
	\end{claim}
	\begin{proof}
		Write $\sigma=\{(\sigma_\gamma,p): \gamma<\kappa,p\in S_\gamma\}$ as usual. Let
		\begin{equation*}
		D=\Big\{p\in \PP: (p\forces \tau\not\in \sigma) \vee \exists \gamma<\kappa \Big( (p\forces \tau=\sigma_\gamma)\wedge(p\sforces \sigma_\gamma\in \sigma)\Big)\Big\} 
		\end{equation*}
		As usual, $D$ is dense. Let
		\begin{equation*}
		\mathcal{D}_{\tau\in\sigma}:=\{D\}\cup \bigcup_{\gamma<\kappa}\mathcal{D}_{\tau=\sigma_\gamma}
		\end{equation*}
		Let $g$ meet every element of $\mathcal{D}_{\tau\in\sigma}$ and contain some $p$ forcing $\tau\in\sigma$. Let $q\in g\cap D$, and assume $q\leq p$. Then for some $\gamma$, $q\forces \tau=\sigma_\gamma$ and $q\sforces \sigma_\gamma\in \sigma$, so $\sigma_\gamma^g\in \sigma^g$. Since $g$ meets every element of $\mathcal{D}_{\tau=\sigma_\gamma}$ we know $\tau^g=\sigma_\gamma^g\in \sigma^g$.
	\end{proof}

	We next need to prove similar claims about the negations of all these formulas.
	\begin{claim} 
	\label{claim sigma neq emptyset} 
		The lemma holds when $\varphi$ is of the form $\sigma \neq \check{A}$ for $A\in H_\kappa$.
	\end{claim}
	\begin{proof}
		As before, this is trivial is $\sigma$ is rank $0$. Otherwise, let us write {$\sigma=\{(\sigma_\gamma,p):\gamma<\kappa,p\in S_\gamma\}$} and let
		\begin{align*}
		D=\Big\{p\in \PP: &(p\forces \sigma=\check{A})\vee \Big(\exists \gamma<\kappa (p\sforces \sigma_\gamma\in \sigma) \wedge (p\forces \sigma_\gamma\not \in \check{A})\Big)\\
		& \vee (\exists B\in A: p\forces \check{B}\not \in \sigma)\Big\}
		\end{align*}
		As usual, $D$ is dense.
		
		We then let
		\begin{equation*}
		\mathcal{D}_{\sigma\neq \check{A}}:=\{D\}\cup \bigcup_{\gamma<\kappa}\bigcup_{B\in A} \mathcal{D}_{\sigma_\gamma\neq \check{B}}
		\end{equation*}
		
		By induction, this is well defined, and since $A$ is in $H_{\kappa^+}$ it has cardinality at most $\kappa$. Let $g$ be a filter meeting all of $\mathcal{D}_{\sigma\neq \check{A}}$ with $p\in g$ forcing $\sigma\neq \check{A}$. Take $q\in g \cap D$ below $p$. There are two cases to consider.
		\begin{enumerate-(1)}
			\item For some $\gamma$, $q\sforces \sigma_\gamma\in\sigma$ and $q\forces \sigma_\gamma\not\in\check{A}$. Then certainly $\sigma_\gamma^g\in\sigma^g$. Let $B\in A$. Then $q\forces \sigma_\gamma\neq \check{B}$. Since $g$ meets all of $\mathcal{D}_{\sigma_\gamma \neq B}$, we know $\sigma_\gamma^g\neq B$. Hence $\sigma_\gamma^g\in \sigma^g\setminus A$ so $\sigma^g\neq A$.
			\item For some $B\in A$, $q\forces \check{B}\not \in\sigma$. Let $B'\in \sigma^g$. Then for some $\gamma<\kappa$ and $r\leq q$ in $g$, $\sigma_\gamma^g=B'$ and $r\sforces \sigma_\gamma\in \sigma$. Hence $r\forces \sigma_\gamma\in\sigma$. But also $r\forces \check{B}\not\in\sigma$ since $r\leq q$. Therefore $r\forces \sigma_\gamma\neq \check{B}$, and so $B'=\sigma_\gamma^g\neq B$ since $g$ meets $\mathcal{D}_{\sigma_\gamma \neq \check{B}}$. Hence $B\in A\setminus \sigma^g$, so again $\sigma^g\neq A$.
		\end{enumerate-(1)}
	\end{proof}
	\begin{claim}
		The lemma holds when $\varphi$ is of the form $\sigma \neq \tau$.
	\end{claim}
	\begin{proof}
		The dense sets we need to use are very similar to the ones in the previous lemma. We assume $\rank(\sigma)\geq \rank(\tau)$ and note that if $\rank(\tau)=0$ we're looking at the previous case. So let us assume $\rank(\sigma)\geq \rank(\tau)>0$ and that we have proved the statement for all $\sigma'$ and $\tau'$ with lower ranks than $\sigma$ and $\tau$ respectively. As usual, write $\sigma=\{(\sigma_\gamma,p):\gamma<\kappa,p\in S_\gamma\}$ and $\tau=\{(\tau_\delta,q): \delta<\kappa,q\in T_\gamma\}$.
		
		Let
		\begin{align*}
		D=\Big\{p\in\PP: &(p\forces \sigma=\tau) \vee \Big(\exists \gamma<\kappa (p \sforces \sigma_\gamma \in \sigma) \wedge (p\forces \sigma_\gamma \not \in \tau)\Big)\\
		&\vee \Big(\exists \delta<\kappa (p\sforces\tau_\delta\in \tau) \wedge (p\forces \tau_\delta \not \in \sigma)\Big)\Big\}
		\end{align*}
		Once again $D$ is dense. We define
		\begin{equation*}
		\mathcal{D}_{\sigma \neq \tau} := \{D\}\cup \bigcup_{\gamma,\delta<\kappa} \mathcal{D}_{\sigma_\gamma \neq \tau_\delta}
		\end{equation*}
		Letting $g$ be our usual filter meeting all of $\mathcal{D}_{\sigma \neq \tau}$ and containing some $p$ forcing $\sigma \neq \tau$, we can find $q\in g\cap D$ below $p$. Without loss of generality, there exists $\gamma<\kappa$ such that $q\sforces \sigma_\gamma\in \sigma$ and $q\forces \sigma_\gamma \not \in \tau$. As always, the first statement implies $\sigma_\gamma^g\in \sigma^g$. If $\sigma_\gamma^g\in \tau^g$ then for some $\delta<\kappa$ and $r\in g$ (which we can take to be below $q$), $\sigma_\gamma^g=\tau_\delta^g$ and $r\sforces \tau_\delta \in \tau$. But then we know $r\forces \sigma_\gamma \neq \tau_\delta$. Since $g$ meets all of $\mathcal{D}_{\sigma_\gamma\neq \tau_\delta}$ this implies $\sigma_\gamma^g\neq \tau_\gamma^g$. Contradiction. Hence $\sigma_\gamma^g \in \sigma^g\setminus\tau^g$, so $\sigma^g\neq \tau^g$.
	\end{proof}	
	\begin{claim}
		The lemma holds when $\varphi$ has the form $\tau\not\in\sigma$.
	\end{claim}
	\begin{proof}
		Write $\sigma=\{(\sigma_\gamma,p):\gamma<\kappa,p\in S_\gamma\}$ as usual. Let
		\begin{equation*}
		\mathcal{D}_{\tau\not\in\sigma} := \bigcup_{\gamma<\kappa}\mathcal{D}_{\tau\neq\sigma_\gamma}
		\end{equation*}
		Suppose $g$ meets all of $\mathcal{D}_{\tau\not\in\sigma}$ and contains some $p$ forcing $\tau \not \in \sigma$. Let $B\in\sigma^g$. For some $\gamma<\kappa$ and some $q\in g$ below $p$, $B=\sigma_\gamma^g$ and $q\sforces \sigma_\gamma\in\sigma$. Then $q\forces \tau \neq \sigma_\gamma$, so $ \tau^g\neq\sigma_\gamma^g=B$. Hence $\tau^g\not \in \sigma^g$.
	\end{proof}	
	We can now finally prove the full lemma.
	\begin{claim}The lemma holds in all cases.
	\end{claim}
	\begin{proof}
		We use induction on the length of the formula $\varphi$. By rearranging $\varphi$, we can assume that all the $\neg$'s in $\varphi$ are in front of atomic formulas. Throughout this proof, we will suppress the irrelevant variables $\vec{\sigma}$ of formulas $\psi(\vec{\sigma})$, and will write $\psi^g$ to denote $\psi(\vec{\sigma}^g)$.
		
		The base case, where $\varphi$ is either atomic or the negation of an atomic formula, was covered in the previous lemmas.
		
		$\varphi=\psi \wedge \chi$: We let $\mathcal{D}_{\varphi}:=\mathcal{D}_\psi \cup \mathcal{D}_\chi$. If $p\in g$ forces $\varphi$ then it also forces $\psi$ and $\chi$, so if also $g$ meets all of $\mathcal{D}_\varphi$ then $\psi^g$ and $\chi^g$ hold.
		
		$\varphi=\psi\vee \chi$: We let $D=\{p\in\PP: (p\forces \neg\varphi) \vee (p\forces \psi) \vee (p\forces \chi)\}$, and let $\mathcal{D}_\varphi:=\{D\}\cup \mathcal{D}_\psi\cup \mathcal{D}_\chi$. If $g$ meets all of $\mathcal{D}_\varphi$ and contains some $p$ which forces $\varphi$ then take $q\leq p$ in $g \cap D$. Then $q\forces \psi$ or $q\forces \chi$, and by definition of $\mathcal{D}_\psi$ and $\mathcal{D}_\chi$ this implies $\psi^g$ or $\chi^g$ respectively.
		
		
		
		$\varphi=\forall x \in \sigma \ \psi(x)$: Write $\sigma=\{(\sigma_\gamma,p):\gamma<\kappa,p\in S_\gamma\}$, and let $\mathcal{D}_\varphi:=\bigcup_{\gamma<\kappa} \mathcal{D}_{\psi(\sigma_\gamma)}$. Suppose, as usual, that $g$ meets all of $\mathcal{D}_\varphi$ and contains some $p$ forcing $\varphi$. Let $B\in \sigma^g$. Then we have some $\gamma<\kappa$ and $q\in g$ such that $\sigma_\gamma^g=B$ and $q\sforces \sigma_\gamma\in \sigma$. Taking (without loss of generality) $q\leq p$, we then have that $q\forces \psi(\sigma_\gamma)$. Hence $\psi^g(\sigma_\gamma^g)$ holds. But we know $\sigma_\gamma^g=B$. Hence $\psi^g(B)$ holds for all $B\in \sigma^g$, so $\varphi^g$ holds.
				
		$\varphi=\exists x\in \sigma \ \psi(x)$: Again we write $\sigma=\{(\sigma_\gamma,p):\gamma<\kappa,p\in S_\gamma\}$. Let $D$ be the dense set $\{p\in\PP: (p\forces \neg \varphi) \vee \exists \gamma<\kappa \ ( p\sforces \sigma_\gamma\in\sigma \wedge p\forces \psi(\sigma_\gamma) )\}$, and let $\mathcal{D}_\varphi := \{D\}\cup \bigcup_{\gamma<\kappa} \mathcal{D}_{\psi(\sigma_\gamma)}$. If $g$ meets all of $\mathcal{D}_\varphi$ and contains $p$ forcing $\varphi$ then we can take some element $q$ of $g\cap D$ below $p$. Then for some $\gamma<\kappa$, we know $q\forces \psi(\sigma_\gamma)$ and $q\sforces \sigma_\gamma\in \sigma$. Then $\psi^g(\sigma_\gamma^g)$ holds, and $\sigma_\gamma^g\in \sigma^g$.
	\end{proof}
This completes the proof of Lemma \ref{collection of dense sets to witness first order statement}. 
Hence $\FA_{\PP,\kappa}$ implies $\fo\NP_{\PP,\kappa}(\infty)$, as discussed earlier.
\end{proof}

This completes the proof of Theorem \ref{correspondence forcing axioms name principles}. 
\end{proof}

In fact, this proof works even if we allow formulas to have conjunctions and disjunctions of $\kappa$ many formulas (and accordingly let formulas have $\kappa$ many variables).

The proof of Theorem \ref{correspondence bounded forcing axioms name principles} is essentially the same:

\begin{proof}[Proof of Theorem \ref{correspondence bounded forcing axioms name principles}]
	
	We prove all the implications in the following diagram.

\smallskip 

	\[ \xymatrix@R=1em{ 
  & \fo\BN^\lambda_{\PP,\kappa}(\infty) \ar@{->}[r] \ar@{->}[dd]
&  \BN^\lambda_{\PP,\kappa}(\infty) \ar@{->}[rd] \ar@{->}[dd]
& \\ 
 \BFA^\lambda_{\PP,\kappa} \ar@{->}[ru]^{\kappa\leq\lambda \ \ \  } \ar@{->}[rd]_{\kappa\leq\lambda \ \ \  }  & & & \BFA^\lambda_{\PP,\kappa} \\ 
\labelmargin{10pt}
 &  \fo\BN^\lambda_{\PP,X,\kappa}(\alpha)  \ar@{->}[r]  &  \BN^\lambda_{\PP,X,\kappa}(\alpha)  \ar@{->}[ru]_{\ \  \lvert \mathcal{P}^{<\alpha}(X)\rvert\geq\kappa} & 
}\] 

\bigskip 

Note that $\BFA^\lambda_{\PP,\kappa}\Rightarrow \fo\BN^\lambda_{\PP,X,\kappa}(\alpha)$ for $\kappa\leq\lambda$ follows from the rest of the diagram. 

	\begin{proof}[Proof of $\fo\BN_{\PP,\kappa}^\lambda(\infty)\Rightarrow \fo \BN_{\PP,\kappa}^\lambda(\alpha)$ and $\BN_{\PP,\kappa}^\lambda(\infty) \Rightarrow \BN_{\PP,\kappa}^\lambda(\alpha)$] 
	\mbox{}\\*
	The latter are special cases of the former. 
	\end{proof}


	\begin{proof}[Proof of $\fo\BN_{\PP,X,\kappa}^\lambda(\alpha)\Rightarrow\BN_{\PP,X,\kappa}^\lambda(\alpha)$ and $\fo\BN_{\PP,\kappa}^\lambda(\infty)\Rightarrow \BN_{\PP,\kappa}^\lambda(\infty)$]
	\mbox{}\\*
		As before, similar to the proofs in Theorems \ref{correspondence forcing axioms name principles}.
	\end{proof}
	
	\begin{proof}[Proof of $\BN_{\PP,X,\kappa}^\lambda(\alpha)\Rightarrow \BFA_{\PP,\kappa}^\lambda$ and $\BN_{\PP,\kappa}^\lambda(\infty)\Rightarrow \BFA_{\PP,\kappa}^\lambda$]
		Letting $\langle D_\gamma \mid \gamma<\kappa\rangle$ be a sequence of predense sets of cardinality at most $\lambda$, we define a name $\sigma$ exactly as in the corresponding proof from Theorem \ref{correspondence forcing axioms name principles}. Since the $D_\gamma$ have cardinality at most $\lambda$, and all the names that appear in $\sigma$ are $1$ bounded check names, $\sigma$ is $\lambda$-bounded.
		
		As in the earlier proof, a filter $g$ such that $\sigma^g=A$ will meet all of the $D_\gamma$.
	\end{proof}
	
	\begin{proof}[Proof of $\BFA_{\PP,\kappa}^\lambda \Rightarrow \fo\BN_{\PP,\kappa}^\lambda$]
	
	Assume $\lambda \geq \kappa$. We prove the following lemma (very similar to Lemma \ref{collection of dense sets to witness first order statement}). 
	
	\begin{lemma}
		Let $\varphi(\vec{\sigma})$ be a $\Sigma_0$ formula where $\vec{\sigma}$ is a tuple of $\kappa$-small $\lambda$-bounded names. Then there is a collection $\mathcal{D}_{\varphi(\vec{\sigma})}$ of at most $\kappa$ many predense sets each of cardinality at most $\lambda$, which has the following property: if $g$ is any filter meeting every set in $\mathcal{D}_{\varphi(\vec{\sigma})}$ and $g$ contains some $p$ such that $p\forces\varphi(\vec{\sigma})$, then in fact $\varphi(\vec{\sigma^g})$ holds in $V$.
	\end{lemma}
	
	We use the same proof as in Theorem \ref{correspondence forcing axioms name principles}, adjusting the dense sets we work with. Whenever a dense set appears, we will replace it with a predense set of size at most $\lambda$ which fulfills all the same functions. To obtain these sets, we use a few techniques.
	
	First, whenever the original proof calls for an arbitrary condition which forces some desirable property, we replace it with the supremum of all such conditions (exploiting the fact that we are in a complete Boolean algebra).
	
	For example, in place of 
	
	\begin{equation*}
	E_\gamma = \Big\{p\in \mathbb{P}: (p\forces \sigma \neq \check{A}) \vee (p\forces \sigma_\gamma \not \in A) \vee \Big(\exists B\in A, p\forces \sigma_\gamma=\check{B}\Big)\Big\}
	\end{equation*} in Claim \ref{sigma=A}, we would take the set
	\begin{equation*}
	E_\gamma^* := \{q_0,q_1\}\cup \{q_B: B\in A\}
	\end{equation*}
	where
	\begin{align*}
	q_0&=\sup \{p\in \PP: p\forces \sigma \neq \check{A}\} \\
	q_1&=\sup \{p\in \PP: p\forces \sigma_\gamma\not \in \check{A}\}
	\end{align*}
	and for $B\in A$,
	\begin{equation*}
	q_B=\sup \{p\in \PP: p\forces \sigma_\gamma=\check{B}\}.
	\end{equation*}
	$E_\gamma^*$ has cardinality at most $\lambda$, since $\lvert A \rvert \leq \kappa \leq \lambda$.
	
	When the original set calls for a condition which strongly forces $\tau\in \sigma$ for some $\tau$ and $\sigma$, simply taking suprema won't work. Instead, we ask for a condition $q$ such that $(\tau,q)\in \sigma$. Since all the names $\sigma$ we deal with in the proof are $\lambda$-bounded, there will be at most $\lambda$ many such conditions.
	
	For example, in the same claim,
	
	\begin{equation*}
	D_B:=\Big\{p\in \mathbb{P} : (p\forces \sigma \neq \check{A})\vee\Big(\exists \gamma<\kappa \: (p\forces \sigma_\gamma=\check{B})\wedge (p\sforces \sigma_\gamma \in \sigma)\Big) \Big\}
	\end{equation*}
	will be replaced by
	\begin{equation*}
	D_B^*:=\{r\} \cup \{r_{\gamma,q}: \gamma<\kappa, q\in \PP, (\sigma_\gamma,q)\in \sigma, r_{\gamma,q}\neq 0\}
	\end{equation*}
	where
	\begin{equation*}
	r=\sup\{p\in\PP: p\forces \sigma \neq \check{A}\}
	\end{equation*}
	and for $\gamma<\kappa$, $q\in \PP$,
	\begin{equation*}
	r_{\gamma,q}=\sup \{p\leq q: p\forces \sigma_\gamma=\check{B}\}.
	\end{equation*}

Checking that we can indeed apply these techniques to turn all the dense sets in the proof into predense sets of cardinality at most $\lambda$ is left as an exercise for the particularly thorough reader. 
\end{proof} 
This completes the proof of Theorem \ref{correspondence bounded forcing axioms name principles}. 
\end{proof}

\subsection{Generic absoluteness}
\label{subsection - generic absoluteness} 

In this section, we derive generic absoluteness principles from the above correspondence. 

Fix a cardinal $\kappa$. 
We start by defining the class of $\Sigma^1_1(\kappa)$-formulas. 
To this end, work with a two-sorted logic with two types of variables, interpreted as ranging over ordinals below $\kappa$ and over subsets of $\kappa$, respectively. 
The language contains a binary relation symbol $\in$ and a binary function symbol $p$ for a pairing function $\kappa\times\kappa\rightarrow \kappa$. 
Thus, atomic formulas are of the form $\alpha=\beta$, $x=y$, $\alpha\in x$ and $p(\alpha,\beta)=\gamma$, where $\alpha,\beta,\gamma$ denote ordinals and $x,y$ denote subsets of $\kappa$. 

\begin{definition}
A $\Sigma^1_1(\kappa)$ formula is of the form 
$$\exists x_0,\dots,x_m\ \varphi(x_0,\dots,x_m,y_0,\dots,y_n),$$ 
where the $x_i$ are variables for subsets of $\kappa$, the $y_i$ are either type of variables, and $\varphi$ is a formula which only quantifies over variables for ordinals.
\end{definition}

As a corollary to the results in Section \ref{Section_correspondence}, we obtain Bagaria's characterisation of bounded forcing axioms \cite[Theorem 5]{bagaria2000bounded} as the equivalence \ref{Bagaria's characterisation 1} $\Leftrightarrow$ \ref{Bagaria's characterisation 4} of the next theorem. 
It also shows that the principles $\fo\BN^\lambda_{\PP,\kappa}$ for $\lambda<\kappa$ are all equivalent to $\BFA^\kappa_{\PP,\kappa}$.


\begin{theorem} 
\label{Bagaria's characterisation} 
Suppose that $\kappa$ is a cardinal with $\cof(\kappa)>\omega$, $\PP$ is a complete Boolean algebra and $\dot{G}$ is a $\PP$-name for the generic filter. 
Then the following conditions are equivalent:\footnote{The equivalence \ref{Bagaria's characterisation 1} $\Leftrightarrow$ \ref{Bagaria's characterisation 4} is equivalent to Bagaria's version, since his definition of $\BFA$ refers to Boolean completions.}
\begin{enumerate-(1)} 
\item 
\label{Bagaria's characterisation 1} 
$\BFA^\kappa_{\PP,\kappa}$ 
\item 
\label{Bagaria's characterisation 2} 
$\fo\BN^1_{\PP,\kappa}(1)$ \footnote{The version  $\Sigma_0-\BN^1_{\PP,\kappa}(1)$ for single $\Sigma_0$-formulas is also equivalent by the proof below.} 
\item 
\label{Bagaria's characterisation 3} 
$\forces_\PP V \prec_{\Sigma^1_1(\kappa)} V[\dot{G}]$
\item 
\label{Bagaria's characterisation 4} 
$\forces_\PP H_{\kappa^+}^V \prec_{\Sigma_1} H_{\kappa^+}^{V[\dot{G}]}$
\end{enumerate-(1)} 
\end{theorem} 

\begin{proof} 
The implication \ref{Bagaria's characterisation 1} $\Rightarrow$ \ref{Bagaria's characterisation 2} holds since $\BFA^\kappa_{\PP,\kappa} \Leftrightarrow \fo\BN^\kappa_{\PP,\kappa}(1)$ by Theorem \ref{correspondence bounded forcing axioms name principles} and $ \fo\BN^\kappa_{\PP,\kappa}(1) $ implies  $\fo\BN^1_{\PP,\kappa}(1)$. 

\ref{Bagaria's characterisation 2} $\Rightarrow$ \ref{Bagaria's characterisation 3}: 
To simplify the notation, we will only work with $\Sigma^1_1(\kappa)$-formulas of the form $\exists x\ \varphi(x,y)$, where $x$ and $y$ range over subsets of $\kappa$. 
Suppose that $y$ is a subset of $\kappa$ and $p\forces \exists x\ \varphi(x,\check{y})$. 
Let $\sigma$ be a $1$-bounded rank $1$ $\PP$-name with $p \Vdash_\PP \varphi(\sigma,\check{y})$. 
Note that $\check{y}$ is a $1$-bounded rank $1$ name, too. 
By $\fo\BN_{\PP,\kappa}^{\kappa}(1)$, there exists a filter $g\in V$ on $\PP$ such that $V \models \varphi(\sigma^g,y)$. 
Hence $V\models \exists x\ \varphi(x,y)$. 

The implication 
\ref{Bagaria's characterisation 3} $\Rightarrow$ \ref{Bagaria's characterisation 1} works just like in the proof of \cite[Theorem 5]{bagaria2000bounded}. 
In short, the existence of the required filter is equivalent to a $\Sigma^1_1(\kappa)$-statement. 

For \ref{Bagaria's characterisation 3} $\Rightarrow$ \ref{Bagaria's characterisation 4}, suppose that $\psi=\exists x\ \varphi(x,y)$ is a $\Sigma_1$-formula with a parameter $y\in H_{\kappa^+}$. 
Then 
$$ H_{\kappa^+} \models \psi \Longleftrightarrow H_{\kappa^+} \models ``\exists M \text{ transitive s.t. } y\in M \wedge M\models \psi". $$ 
We express the latter by a $\Sigma^1_1(\kappa)$-formula $\theta$ with a parameter $A\subseteq \kappa$ which codes $y$ in the sense that $f(0)=y$ for the transitive collapse $f$ of $(\kappa,p^{-1}[A])$. 

$\theta$ states the existence of a subset $B$ of $\kappa$ such that $\in_M:=p^{-1}[B]$ has the following properties: 
\begin{itemize} 
\item 
$\in_M$ is wellfounded and extensional 
\item 
For all $\alpha<\beta<\kappa$, $2 \cdot \alpha \in_M 2\cdot \beta$ and  for all $\alpha,\beta<\kappa$, $2\cdot \alpha+1 \not\in_M 2\cdot \beta$. 
\item 
There is some $\hat{\kappa}< \kappa$ with $\{ \alpha<\kappa \mid \alpha \in_M \hat{\kappa}\} = \{ 2\cdot \alpha \mid \alpha<\kappa \}$
\item 
There exists some $\hat{A}<\kappa$ such that for all $\beta<\kappa$, $\beta \in_M \hat{A} \Leftrightarrow \exists \alpha\in A\ 2\cdot \alpha=\beta$  
\item 
There exists some $\hat{y}<\kappa$ such that in $(\kappa,\in_M)$, $\hat{A}$ codes $\hat{y}$ 
\item 
$\varphi(\hat{y})$ holds in $(\kappa,\in_M)$ 
\end{itemize} 
The transitive collapse $f$ of $(\kappa,\in_M)$ to a transitive set $M$ will satisfy $f(2\cdot \alpha)=\alpha$ for all $\alpha<\kappa$, $f(\hat{\kappa})=\kappa$, $f(\hat{A})=A$, $f(\hat{y})=y$ and $M\models \psi(y)$. 

All the above conditions apart from wellfoundedness of $\in_M$ are first order over $(\kappa,\in,p,A,\in_M)$. 
It remains to express wellfoundedness of $\in_M$ in a $\Sigma^1_1(\kappa)$ way.\footnote{$\cof(\kappa)>\omega$ is in fact necessary to ensure that the set of codes on $\kappa$ for elements of $H_{\kappa^+}$ is $\Sigma^1_1(\kappa)$-definable with parameters in $\pow(\kappa)$. 
If $\cof(\kappa)=\omega$ and $\kappa$ is a strong limit, then this set is $\Pi^1_1(\kappa)$-complete and hence not $\Sigma^1_1(\kappa)$ by a result of Dimonte and Motto Ros \cite{dimontemottoros}.} 
To see that we can do this, suppose that $R$ is a binary relation on $\kappa$. 
Since $\cof(\kappa)>\omega$, $R$ is wellfounded if and only if for all $\gamma<\kappa$, $R{\upharpoonright}\gamma$ is wellfounded. 
Since $\gamma<\kappa$, $R{\upharpoonright}\gamma$ is wellfounded if and only if there exists a 
map $f\colon \gamma\rightarrow \kappa$ such that for all $\alpha,\beta<\gamma$, $(\alpha,\beta)\in R \Rightarrow f(\alpha)<f(\beta)$. 
The existence of such a map $f$ is a $\Sigma^1_1(\kappa)$ statement. 

Finally, \ref{Bagaria's characterisation 4} $\Rightarrow$ \ref{Bagaria's characterisation 3} holds since every $\Sigma^1_1(\kappa)$-formula is equivalent to a $\Sigma_1$-formula over $H_{\kappa^+}$ with parameter $\kappa$. 
\end{proof} 


\begin{remark} 
Note that for rank $1$, $\fo\BN^\lambda_{\PP,\kappa}(1)$ implies the simultaneous $\lambda$-bounded rank $1$ name principle for all $\Sigma^1_1(\kappa)$-formulas (see Definition \ref{Defn_foN}) by picking $1$-bounded names for witnesses. 
\end{remark} 

\begin{remark} 
The previous results cannot be extended to higher complexity. 
To see this, recall that a $\Pi^1_1(\kappa)$-formula is the negation of a $\Sigma^1_1(\kappa)$-formula. 
We claim that there exists a $\Pi^1_1(\omega_1)$-formula $\varphi(x)$ such that the $1$-bounded rank $1$ $\Pi^1_1(\omega_1)$-name principle for the class of c.c.c. forcings fails. 
Otherwise $\mathsf{MA}_{\omega_1}$ would hold by \ref{Bagaria's characterisation 2} $\Rightarrow$ \ref{Bagaria's characterisation 1} of Theorem \ref{Bagaria's characterisation}. 
So in particular, there are no Suslin trees. 
Since adding a Cohen real adds a Suslin tree, let $\sigma$ be a $1$-bounded rank $1$ $\PP$-name for it, where $\PP$ denotes the Boolean completion of Cohen forcing, and apply the name principle to the statement ``$\sigma$ is a Suslin tree''. But then we would have a Suslin tree in $V$. 
\end{remark} 

\begin{remark} 
Fuchs and Minden show in \cite[Theorem 4.21]{fuchs2018subcomplete} assuming $\CH$ that the bounded subcomplete forcing axiom $\mathsf{BSCFA}$ can be characterised by the preservation of $(\omega_1,{\leq}\omega_1)$-Aronszajn trees. 
The latter can be understood as the $1$-bounded name principle for statements of the form  ``$\sigma$ is an $\omega_1$-branch in $T$'', where $T$ is an $(\omega_1,{\leq}\omega_1)$-Aronszajn tree. 
(See \cite{fuchs2018subcomplete,jensen2014subcomplete} for more about subcomplete forcing.)


\end{remark} 

We now consider forcing axioms at cardinals $\kappa$ of countable cofinality. 
To our knowledge, these have not been studied before. 
$\BFA^\kappa_{c.c.c.,\kappa}=\MA_\kappa$ is an example of a consistent forcing axiom of this form. 
We fix some notation. 
If $\kappa$ is an uncountable cardinals with $\cof(\kappa)=\mu$, we 
fix a continuous strictly increasing sequence $\langle \kappa_i \mid i\in\mu \rangle $ of ordinals with $\kappa_0=0$ and $\sup_{i\in\mu} \kappa_i=\kappa$. 
We assume that each $\kappa_i$ is closed un the pairing function $p$.\footnote{If $\kappa_i$ is multiplicatively closed, i.e. $\forall \alpha<\kappa \alpha\cdot \alpha<\kappa_i$, then this holds for G\"odel's pairing function.} 
For each $x\in 2^\kappa$, we define a function $f_x\colon \mu\rightarrow 2^{<\kappa}$ by letting $f_x(i)=x{\upharpoonright}\kappa_i$.

\begin{lemma} 
\label{tree projecting to a set} 
Suppose that $\kappa$ is an uncountable cardinal with $\cof(\kappa)=\mu$. 
Suppose that $\varphi(x,y)$ is a formula with quantifiers ranging over $\kappa$ and $y\in 2^\kappa$ is fixed. 
Then there is a subtree $T\in V$ of $((2^{<\kappa})^{<\mu})^2$ such that in all generic extensions $V[G]$ of $V$ \footnote{This includes the case $V[G]=V$.} which do not add new bounded subsets of $\kappa$, 
$$\varphi(x,y) \Longleftrightarrow \exists g\in (2^{<\kappa})^\mu\ (f_x,g)\in [T]$$ 
holds for all $x\in (2^\kappa)^{V[G]}$. 
Moreover, for any branch $(\vec{s},\vec{t})\in [T]$ in $V[G]$ with $\vec{s}=\langle s_i\mid i\in\mu\rangle$, $\bigcup_{i\in\mu} s_i=f_x$ for some $x\in  (2^\kappa)^{V[G]}$. 
\end{lemma} 
\begin{proof} 
We construct the $i$-th levels $\Lev_i(T)$ by induction on $i\in\mu$. 
Let $\Lev_0(T)=\{(\emptyset,\emptyset)\}$. 
If $j\in\mu$ is a limit, let $(\vec{s},\vec{t}) \in \Lev_j(T)$ if $(\vec{s}{\upharpoonright}i,\vec{t}{\upharpoonright}i) \in \Lev_i(T)$ for all $i<j$. 

For the successor step, suppose that $\Lev_j(T)$ has been constructed. 
Write $\vec{s}=\langle s_i \mid i\leq j\rangle$ and  $\vec{t}=\langle t_i \mid i\leq j\rangle$.  
Let $(\vec{s},\vec{t})\in \Lev_{j+1}(T)$ if the following conditions hold: 
\begin{enumerate-(1)} 
\item 
$(\vec{s}{\upharpoonright}j,\vec{t}{\upharpoonright}j)\in \Lev_j(T)$. 
\item 
$s_j\in 2^{\kappa_j}$ and $\forall i < j\ s_j{\upharpoonright}\kappa_i=s_i$. 
\item 
$t_j\in 2^{\kappa_j}$ 
codes the following two objects. 
\begin{enumerate-(i)} 
\item 
A truth table $p_j$ which assigns to each formula $\psi(\xi_0,\dots,\xi_k)$ and parameters $\alpha_0,\dots,\alpha_k<\kappa_j$ a truth value $0$ or $1$. 
\item 
A function $q_j$ which assigns a value in $\omega$ to each existential formula $\exists \eta\ \psi(\xi_0,\dots,\xi_k,\eta)$ and associated parameters $\alpha_0,\dots,\alpha_k<\kappa_j$. 
\end{enumerate-(i)}
They satisfy $p_i \subseteq p_j$, $q_i\subseteq q_j=q_i$ for all $i<j$ and the following conditions: 
\begin{enumerate-(a)}
\item 
$p_j(\varphi)=1$. 

\item 
$p_j$ satisfies the equality axioms: 
$$ p_j((\psi(\vec{\xi})),\vec{\alpha}) =1 \wedge \vec{\alpha}=\vec{\beta} \Longleftrightarrow  p_j((\psi(\vec{\xi})),\vec{\beta}) =1$$  

\item 
$p_j$ is correct about atomic formulas $\psi(\vec{\xi})$ which do not mention $\dot{x}$ and $\dot{y}$: 
$$ p_j((\psi(\vec{\xi})),\vec{\alpha}) =1 \Longleftrightarrow  \psi(\vec{\alpha})$$  

\item 
The truth in $p_j$ of all atomic formulas of the form $\xi\in \dot{x}$, $\xi\in \dot{y}$ is fixed according to $s_j$ and $y$, respectively: 
$$ p_j ((\xi \in \dot{x}),\alpha) =1 \Longleftrightarrow  \alpha\in s_j$$  
$$ p_j ((\xi \in \dot{y}),\alpha)=1 \Longleftrightarrow  \alpha\in y$$  

\item 
$p_j$ respects propositional connectives: 
$$ p_j(\psi\wedge \theta,\vec{\alpha})=1 \Longleftrightarrow p_j(\psi,\vec{\alpha})=1 \wedge p_j( \theta,\vec{\alpha})=1$$  
$$ p_j(\neg \psi ,\vec{\alpha}) =1 \Longleftrightarrow p_j (\psi,\vec{\alpha}) =0 $$  

 \item 
 $p_j$ respectes witnesses of existential formulas $\exists \eta\ \psi(\vec{\xi},\eta),\vec{\alpha})$ which it has identified: 
$$\exists \beta <\kappa_j\ p_j(\psi(\vec{\xi},\eta),\vec{\alpha},\beta)=1 \Longrightarrow p_j(\exists \eta\ \psi(\vec{\xi},\eta),\vec{\alpha})=1.$$ 

\item 
$q_j$ promises the existence of existential witnesses: 
for any existential formula $\exists \eta\ \psi(\vec{\xi},\eta)$ and any tuple $\vec{\alpha}$ of parameters, if ${p_j(\exists \eta\ \psi(\vec{\xi},\eta),\vec{\alpha})=1}$ and $q_j(\exists \eta\ \psi(\vec{\xi},\eta),\vec{\alpha})\leq n$, 
then there exists some $\beta<\kappa_j$ such that $p_j(\psi(\vec{\xi},\eta), \vec{\alpha},\beta)=1$. 
\end{enumerate-(a)} 
\end{enumerate-(1)}
Let $V[G]$ be a generic extension of $V$ with no new bounded subsets of $\kappa$. 
Work in $V[G]$. 

$\Rightarrow$: 
Suppose that $\varphi(x,y)$ holds. 
We define $s_j=x{\upharpoonright}\kappa_j$ for each $j\in\mu$ and $p_j(\psi(\vec{\xi}),\vec{\alpha})=1$ if $(\kappa,\in,p,x,y)\models \psi (\vec{\alpha})$. 
We further define $q_j(\exists \eta\ \psi(\vec{\xi},\eta),\vec{\alpha})=0$ if $p_j(\exists \eta\ \psi(\vec{\xi},\eta),\vec{\alpha})=0$. 
Otherwise, $q_j(\exists \eta\ \psi(\vec{\xi},\eta),\vec{\alpha})$ is defined as the least $l\in\mu$ such that for some $\beta<\kappa_l$, $(\kappa,\in,p,x,y)\models \psi (\vec{\alpha},\beta)$. 
Let $t_j$ code $p_j$ and $q_j$ (via the pairing function $p$). 
Note that $s_j$, $p_j$ and $q_j$ are in $V$, since $V[G]$ has no new bounded subsets of $\kappa$. 
Hence $\langle (s_j,t_j)\mid j\in\mu\rangle$ is a branch through $T$. 

$\Leftarrow$: 
Suppose that $\langle (s_j,t_j)\mid j\in\mu\rangle$ is a branch through $T$. 
Let $x=\bigcup_{j\in\mu} s_j$. 
By induction on complexity of formulas, $p_j$ and $q_j$ are correct about $x$ and $y$. 
Therefore $(\kappa,\in,p,x,y)\models \varphi (x,y)$. 
\end{proof} 

\begin{theorem} 
\label{variant of Bagaria's characterisation for countable cofinality} 
Suppose that $\kappa$ is an uncountable cardinal with $\cof(\kappa)=\omega$, $\PP$ is a complete Boolean algebra and $\dot{G}$ is a $\PP$-name for the generic filter. 
Then the following conditions are equivalent: 
\begin{enumerate-(1)} 
\item 
\label{Bagaria's characterisation 1b} 
$\BFA^\kappa_{\PP,\kappa}$ 
\item 
\label{Bagaria's characterisation 2b} 
$\fo\BN^1_{\PP,\kappa}$ 
\item 
\label{Bagaria's characterisation 3b} 
$\Vdash_\PP V \prec_{\Sigma^1_1(\kappa)}V[\dot{G}]$
\end{enumerate-(1)} 
If moreover $2^{<\kappa}=\kappa$ holds,\footnote{The assumption $2^{<\kappa}=\kappa$ is not needed for \ref{Bagaria's characterisation 4b} $\Rightarrow$ \ref{Bagaria's characterisation 3b}.} then the next condition is equivalent to \ref{Bagaria's characterisation 1b}, \ref{Bagaria's characterisation 2b} and \ref{Bagaria's characterisation 3b}: 
\begin{enumerate-(1)} 
\setcounter{enumi}{3}
\item 
\label{Bagaria's characterisation 4b} 
$1_\PP$ forces that no new bounded subset of $\kappa$ are added. 
\end{enumerate-(1)} 
If there exists no inner model with a Woodin cardinal,\footnote{The assumption that there is no inner model with a Woodin cardinal is not used for \ref{Bagaria's characterisation 5b} $\Rightarrow$ \ref{Bagaria's characterisation 3b}.} then the next condition is equivalent to \ref{Bagaria's characterisation 1b}, \ref{Bagaria's characterisation 2b} and  \ref{Bagaria's characterisation 3b}: 
\begin{enumerate-(1)} 
\setcounter{enumi}{4}
\item 
\label{Bagaria's characterisation 5b} 
$\Vdash_\PP H_{\kappa^+}^V \prec_{\Sigma_1} H_{\kappa^+}^{V[\dot{G}]}$
\end{enumerate-(1)} 
\end{theorem} 

\begin{proof}
The proofs of \ref{Bagaria's characterisation 1} $\Leftrightarrow$ \ref{Bagaria's characterisation 2} $\Leftrightarrow$ \ref{Bagaria's characterisation 3} $\Leftarrow$ \ref{Bagaria's characterisation 5b} in Theorem \ref{Bagaria's characterisation} work for all uncountable cardinals $\kappa$. 

\ref{Bagaria's characterisation 3b} $\Rightarrow$ \ref{Bagaria's characterisation 4b}: 
We assume $2^{<\kappa}=\kappa$. 
Towards a contradiction, suppose that $V[G]$ is a generic extension that adds a new subset of $\gamma<\kappa$. 
Note that $2^\gamma\leq \kappa$. 
Let $\vec{y}=\langle y_i \mid i<2^\gamma\rangle$ list all subsets of $\gamma$. 
We define $x \subseteq \gamma\cdot 2^\gamma \subseteq \kappa$ by letting 
$\gamma\cdot i + j \in x \Leftrightarrow j\in y_i$. 
The next formula expresses ``there is a new subset of $\gamma<\kappa$'' as a $\Sigma^1_1(\kappa)$-statement in parameters coding the $+$ and $\cdot$ operations: 
$$\exists z\ [z\subseteq \gamma \wedge \neg\exists i\ \forall j<\gamma\ ( j\in z \Leftrightarrow \gamma\cdot i + j \in x)].$$ 
This contradicts $\Sigma^1_1(\kappa)$-absoluteness. 
 
\ref{Bagaria's characterisation 4b} $\Rightarrow$ \ref{Bagaria's characterisation 3b}: 
Suppose that $\exists x\ \psi(x,y)$ is a $\Sigma^1_1(\kappa)$-formula and $y\in (2^\kappa)^V$. 
Let $T$ be a subtree of $((2^{<\kappa})^{<\omega})^2$ as in Lemma \ref{tree projecting to a set}. 
Let $G$ be $\PP$-generic over $V$ with $V[G]\vDash \exists x\ \psi(x,y)$. 
$V[G]$ does not have new bounded subsets of $\kappa$ by assumption. 
Then $[T]$ has a branch in $V[G]$ by the property of $T$ in Lemma \ref{tree projecting to a set}. 
Since wellfoundedness is absolute, $T$ has a branch $\langle s_n, t_n\mid n\in\omega \rangle$ in $V$. 
Then $\bigcup_{n\in\omega} s_n=f_x$ for some $x\in 2^\kappa$ by the properties of $T$. 
Since 
$$\psi(x,y) \Longleftrightarrow \exists g\ (f_x,g)\in [T],$$ 
 we have $V\models  \psi(x,y)$. 
 
\ref{Bagaria's characterisation 3b} $\Rightarrow$ \ref{Bagaria's characterisation 5b}: 
Note that the implication holds vacuously if $\kappa$ is collapsed in some $\PP$-generic extension of $V$. 
In this case, both \ref{Bagaria's characterisation 3b} and \ref{Bagaria's characterisation 5b} fail, since the statement ``$\kappa$ is not a cardinal'' is $\Sigma^1_1(\kappa)$.  

We next show: 
if $q\in \PP$ forces that $\kappa^+$ is preserved, then $q \Vdash H_{\kappa^+}^V \prec_{\Sigma_1} H_{\kappa^+}^{V[\dot{G}]}$ holds. 
To see this, let $G$ be $\PP$-generic over $V$ with $q\in G$. 
Suppose $\psi=\exists x\ \varphi(x,y)$ is a $\Sigma_1$-formula with a parameter $y\in H_{\kappa^+}$. 
We follow the proof of \ref{Bagaria's characterisation 3} $\Rightarrow$ \ref{Bagaria's characterisation 4} in Corollary \ref{Bagaria's characterisation} to construct a $\Sigma^1_1(\kappa)$-formula $\theta$ that is equivalent to $\psi$.  
However, we replace the first condition by: 
\begin{itemize} 
\item 
$\in_M$ is extensional and wellfounded of rank $\gamma$ 
\end{itemize} 
for a fixed $\gamma<(\kappa^+)^V=(\kappa^+)^{V[G]}$. 
If $\psi$ is true, then for sufficiently large $\gamma$, $\theta$ will be true. 
Now we only need to modify the last step of the above proof. 
Let $C$ be a subset of $\kappa$ such that $(\kappa,p^{-1}[C])\cong (\gamma,<)$. 
Suppose $R$ is a binary relation on $\kappa$. 
The condition ``$R$ is wellfounded of rank ${\leq}\gamma$'' is $\Sigma^1_1(\kappa)$ in $C$, since it is equivalent to the existence of a function $f\colon \kappa\rightarrow \gamma$ such that for all $\alpha,\beta<\kappa$, $(\alpha,\beta)\in R \Rightarrow f(\alpha)<f(\beta)$. 

Towards a contradiction, suppose that there is no inner model with a Woodin cardinal and in some $\PP$-generic extension $V[G]$ of $V$, $H_{\kappa^+}^V \prec_{\Sigma_1} H_{\kappa^+}^{V[G]}$ fails. 
By the previous remarks, $\kappa$ is preserved and $\kappa^+$ is collapsed in $V[G]$. 
Since there is no inner model with a Woodin cardinal, the Jensen-Steel core model $K$ from \cite{jensen2013k} is generically absolute and satisfies $(\lambda^+)^K=\lambda^+$ for all singular cardinals $\lambda$ by \cite[Theorem 1.1]{jensen2013k}. 
Therefore any generic extension $V[G]$ of $V$ which does not collapse $\lambda$ satisfies $(\lambda^+)^V=(\lambda^+)^{V[G]}$. 
For $\lambda=\kappa$, this contradicts our assumption. 
\end{proof} 

Can one remove the assumption that there is no inner model with a Woodin cardinal? 
A forcing $\PP$ that witnesses the failure of \ref{Bagaria's characterisation 3b} $\Rightarrow$ \ref{Bagaria's characterisation 5b} must preserve $\kappa$ and collapse $\kappa^+$ by the above proof. 
The existence of a forcing $\PP$ with these two properties is consistent relative to the existence of a $\lambda^+$-supercompact cardinal $\lambda$ by a result of Adolf, Apter and Koepke \cite[Theorem 7]{adolf2018singularizing}. 
Their forcing does not add new bounded subsets of $\kappa$ as in \ref{Bagaria's characterisation 4b} and thus also satisfies \ref{Bagaria's characterisation 1b}-\ref{Bagaria's characterisation 3b}. 
However, we do not know if it satisfies \ref{Bagaria's characterisation 5b}. 

\begin{question} 
Is it consistent that there exist an uncountable cardinal $\kappa$ with $\cof(\kappa)=\omega$ and a forcing $\PP$ with the properties: 
\begin{enumerate-(a)} 
\item 
$\PP$ does not add new bounded subsets of $\kappa$ and 
\item 
$\Vdash_\PP H_{\kappa^+}^V \prec_{\Sigma_1} H_{\kappa^+}^{V[\dot{G}]}$ fails? 
\end{enumerate-(a)} 
(Thus $\PP$ necessarily collapses $\kappa^+$.) 
\end{question}

\subsection{Boolean ultrapowers} 

In this section, we translate the above correspondence to Boolean ultrapowers and use this to characterise forcing axioms via elementary embeddings. 

The Boolean ultrapower construction generalises ultrapowers with respect to ultrafilters on the power set of a set to ultrafilters on arbitrary Boolean algebras. 
We recall the basic definitions from Hamkins' and Seabold's work on Boolean ultrapowers \cite[Section 3]{hamkins2012well}. 
Suppose that $\PP$ is a forcing and $\BB$ its Boolean completion. 
Fix an 
ultrafilter $U$ on $\BB$, which may or may not be in the ground model. 
We define two relations $=_U$ and $\in_U$ on $V^\BB$: 
$$\sigma =_U \tau :\Leftrightarrow \llbracket \sigma=\tau \rrbracket \in U$$ 
$$\sigma \in_U \tau :\Leftrightarrow \llbracket \sigma\in\tau \rrbracket \in U$$ 
Let $[\sigma]_U$ denote the equivalence class of $\sigma\in V^\BB$ with respect to $=_U$. 
Let $V^{\BB}/U=\{ [\sigma]_U \mid \sigma\in V^\BB \}$ denote the quotient with respect to $=_U$. 
$\in_U$ is well-defined on equivalence classes and $(V^\BB/U, \in_U)$ is a model of $\ZFC$ \cite[Theorem 3]{hamkins2012well}. 
It is easy to see from these definitions that for any $\PP$-generic filter $G$ over $V$, $V^\BB/G$ is isomorphic to the generic extension $V[G]$. 
Moreover, we can determine the truth of sentences in $V^\BB/U$ via \L os' theorem \cite[Theorem 10]{hamkins2012well}: $V^\BB/U \models \varphi ([\sigma_0]_U,\dots[\sigma_n]_U) \Longleftrightarrow \llbracket \varphi(\sigma_0,\dots,\sigma_n) \rrbracket \in U$. 
In other words, the forcing theorem holds. 

The \emph{Boolean ultrapower} is the subclass 
$$ \check{V}_U= \{ [\sigma]_U \mid \llbracket\sigma\in \check{V}\rrbracket\in U  \} $$ 
of $V^{\BB}/U$. It is isomorphic to $V$ if and only if $U$ is generic over $V$. 
The \emph{Boolean ultrapower embedding} is the elementary embedding  
$$j_U\colon V \rightarrow \check{V}_U, \text{ \ } j_U(x) = [\check{x}]_U.$$ 
We are interested in the case that $U$ is an ultrafilter in the ground model. 
In particular, $U$ is not $\PP$-generic over $V$. 
$j_U$ has the following properties: 
\begin{itemize} 
\item 
If $U$ is generic, then $j_U$ is an isomorphism. 
\item 
If $U$ is not generic, then $\check{V}_U$ is ill-founded and $\crit(j_U)$ equals the least size of a maximal antichain in $\BB$ not met by $U$ \cite[Theorem 17]{hamkins2012well}. 
For example, if $\PP$ is c.c.c. then $\crit(j_U)=\omega$. 
\end{itemize} 

For any $x\in V^\BB/U$, let $x^{\in_U}=\{y \in V^\BB/U\mid y \in_U x\}$ denote the set of all $\in_U$-elements of $x$. 
If $\kappa$ is a cardinal and $\sigma$ is a name for a subset of $\kappa$, then $[\sigma]_U^{\in_U}\cap j[\kappa]=j[\sigma^{(U)}]$, since 
$V^\BB/U \models j_U(\alpha)= [\check{\alpha}]_U \in [\sigma]_U \Leftrightarrow \llbracket \check{\alpha}\in \sigma\rrbracket \in U \Leftrightarrow  \alpha\in \sigma^{(U)}$ for all $\alpha<\kappa$. 

\begin{theorem} 
\label{characterisation of forcing axioms by Boolean ultrapowers} 
The following statements are equivalent: 

\begin{enumerate-(1)} 
\item 
\label{characterisation of forcing axioms by Boolean ultrapowers 1} 
$\FA_{\PP,\kappa}$ 

\item 
\label{characterisation of forcing axioms by Boolean ultrapowers 2} 
For any transitive set $M\in H_{\kappa^+}$ and for every $\kappa$-small $M$-name $\sigma$, there is an ultrafilter $U\in V$ on $\PP$ such that 
$$ j_U{\upharpoonright} M \colon M \rightarrow j_U(M)^{\in_U}$$ 
is an elementary embedding from  
$(M,\in,\sigma^U)$ to $(j_U(M)^{\in_U},\in_U,[\sigma]_U)$. 

\item 
\label{characterisation of forcing axioms by Boolean ultrapowers 3} 
For any transitive set $M\in H_{\kappa^+}$ and for any $\kappa$-small $M$-name $\sigma$, there is an ultrafilter $U$ on $\PP$ such that 
$$(M,\in,\sigma^U)\equiv (j_U(M)^{\in_U},\in_U,[\sigma]_U),$$ 
i.e. these structures are elementarily equivalent.  
\end{enumerate-(1)} 
\end{theorem} 
\begin{proof} 
\ref{characterisation of forcing axioms by Boolean ultrapowers 1} $\Rightarrow$ \ref{characterisation of forcing axioms by Boolean ultrapowers 2}: 
Recall from Lemma \ref{collection of dense sets to witness first order statement} that for any finite sequence $\vec{\sigma}= \sigma_0,\dots,\sigma_k$ of $\kappa$-small names and and every $\Sigma_0$-formula $\varphi(x_0,\dots,x_k)$, there is a collection $\mathcal{D}_{\varphi(\vec{\sigma})}$ of ${\leq}\kappa$ many dense subsets of $\PP$ with the following property: 
if $g$ is any filter meeting every set in $\mathcal{D}_{\varphi(\vec{\sigma})}$ and $g$ contains some $p$ such that $p\forces\varphi(\vec{\sigma})$, then in fact $\varphi(\vec{\sigma^g})$ holds in $V$. 
Let $\mathcal{D}$ be the union of all collections $\mathcal{D}_{\varphi(\vec{\sigma})}$, where $k\in \omega$, $\varphi(x_0,\dots,x_k)$ is a $\Sigma_0$-formula and each $\sigma_i$ is $\sigma$, $\check{M}$ or $\check{x}$ for some $x\in M$. 
By $\FA_{\PP,\kappa}$, there is a filter $g$ which meets all sets in $\mathcal{D}$. 
We extend $g$ to an ultrafilter $U$.

Suppose that $\psi(x_0,\dots,x_k)$ is a formula such that $(j_U(M)^{\in_U},\in_U,[\sigma]_U) \models \psi(j_U(y_0),\dots,j_U(y_k))$. 
We obtain $\varphi(x_0,\dots,x_{k+2})$ by replacing the unbounded quantifiers in $\psi$ by quantifiers bounded by $x_{k+1}$, and any occurence of $[\sigma]_U$ by $x_{k+2}$. 
Then 
$$(V^\BB/U,\in_U) \models \varphi(j_U(y_0),\dots,j_U(y_k),j_U(M),[\sigma]_U).$$ 
Recall that $j_U(y)=[\check{y}]_U$ for all $u\in M$. 
Therefore by \L os' theorem, we have $\llbracket \varphi(\check{y}_0,\dots,\check{y}_k,\check{M},\sigma) \rrbracket\in U$. 
So there is some $p\in U$ with $p\forces \varphi(\check{y}_0,\dots,\check{y}_k,\check{M},\sigma)$. 
Since $U$ meets all dense sets in $\mathcal{D}_{\varphi(\check{y}_0,\dots,\check{y}_k,\check{M},\sigma)}$, 
$$(V,\in)\models \varphi(y_0,\dots,y_k,M,\sigma^U).$$ 
Hence $(M,\in,\sigma^U)\models \psi(y_0,\dots,y_k)$. 

\ref{characterisation of forcing axioms by Boolean ultrapowers 2} $\Rightarrow$ \ref{characterisation of forcing axioms by Boolean ultrapowers 3}: 
This is clear. 

\ref{characterisation of forcing axioms by Boolean ultrapowers 3} $\Rightarrow$ \ref{characterisation of forcing axioms by Boolean ultrapowers 1}: 
Let  $M=\kappa$ and suppose that $\sigma$ is a rank $1$ $M$-name such that $\PP\Vdash \sigma=\check{\kappa}$. 
Then $\sigma^{(g)}=\kappa$ for any filter $g$ on $\PP$. 
It suffices to find a filter $g$ with $\sigma^g=\kappa$ by Lemma \ref{Lemma FA bracket interpretation}. 
Let $U$ be an ultrafilter as in \ref{characterisation of forcing axioms by Boolean ultrapowers 3}. 
Since $M=\kappa$ and $j_U(M)=j_U(\kappa)=[\check{\kappa}]_U=[\sigma]_U$, we have $(j_U(M)^{\in_U},\in_U,[\sigma]_U)\models \forall x\ x\in_U [\sigma]_U$. 
Thus $(\kappa,\in,\sigma^{U})\models \forall x\ x\in_U \sigma^{U}$ by elementary equivalence.  Thus $\sigma^{U}=\kappa$. 
\end{proof} 

A version of Theorem \ref{characterisation of forcing axioms by Boolean ultrapowers} for $\BFA^\lambda_{\PP,\kappa}$ and $\lambda$-bounded names also holds for any cardinal $\lambda\geq\kappa$. 
The proof is essentially the same.

\subsection{An application to $\ub\FA$} 

\begin{lemma} 
\label{ubFA implies BFA} 
If $\PP$ is a complete Boolean algebra that does not add reals, then 
$$(\forall q\in \PP\ \ub\FA_{\PP_q,\omega_1}) \Longrightarrow \BFA_{\PP,\omega_1}^{\omega_1}.$$ 
More generally, if $\kappa$ an uncountable cardinal and $\PP$ is a complete Boolean algebra that does not add bounded subsets of $\kappa$, then 
$$(\forall q\in \PP\ \ub\FA_{\PP_q,\kappa}) \Longrightarrow \BFA_{\PP,\kappa}^\kappa.$$ 
\end{lemma} 
\begin{proof} 
If $\cof(\kappa)=\omega$, then adding no new bounded subsets of $\kappa$ already implies $\BFA^\kappa_{\PP,\kappa}$ by the proof of \ref{Bagaria's characterisation 4b} $\Rightarrow$ \ref{Bagaria's characterisation 3b} in Theorem \ref{variant of Bagaria's characterisation for countable cofinality}. 
Now suppose that $\cof(\kappa)>\omega$. 
Towards a contradiction, suppose that $\BFA^\kappa_{\PP,\kappa}$ fails. 
Then $\Sigma^1_1(\kappa)$-absoluteness fails for some $\Sigma^1_1(\kappa)$-formula $\exists x\ \psi(x,y)$ and some $y\in (2^\kappa)^V$ by Theorem \ref{Bagaria's characterisation}. 
Take a subtree $T$ of $(2^{<\kappa}\times \kappa^{<\kappa})^{<\cof(\kappa)}$ for $\psi$ as in Lemma \ref{tree projecting to a set}. 
Then $[T]\neq \emptyset$ in $V[G]$ in some $\PP$-generic extension $V[G]$, but $[T]=\emptyset$ in $V$. 
Let $\sigma$ denote a rank $1$ $T$-name and let $q\in \PP$ such that $q \Vdash_\PP \sigma \in [T]$. 
Let 
$$ \tau = \{ (\alpha,p)  \mid  p\leq q\wedge \exists s\in \Lev_\alpha(T)\  p \sforces_\PP \check{s}\in \sigma  \} $$ 
Then $\Vdash_{\PP_q} \tau=\kappa$. 
For any filter $g\in V$ on $\PP_q$ we have $\tau^g = \dom(\sigma^g)$. 
But $\dom(\sigma^g)\in \kappa$, since  $[T]=\emptyset$. 
Therefore $\ub\NP_{\PP_q,\kappa}$ fails and hence $\ub\FA_{\PP_q,\kappa}$ fails by Lemma \ref{Lemma ubFA to ubN}.  
\end{proof} 

We have seen in Lemma \ref{ubFA implies FA for sigma-distributive forcings} that for any ${<}\kappa$-distributive forcing $\PP$, $\ub\FA_{\PP,\kappa}$ implies $\FA_{\PP,\kappa}$. 
In combination with the previous lemma, this begs the question: 

\begin{question} 
\label{question ubFA BFA} 
If $\lambda>\kappa$ is a cardinal and $\PP$ is a complete Boolean algebra that does not add new elements of ${}^{<\kappa}\lambda$, then does the implication 
$$(\forall q\in \PP\ \ub\FA_{\PP_q,\omega_1}) \Longrightarrow \BFA^\lambda_{\PP,\omega_1}$$ 
hold? 
\end{question} 

%

\section{Specific classes of forcings} 
\label{Section specific classes of forcings}  

\subsection{Classes of forcings} 

We now move on to look, over the next few sections, at what further results we can prove if we assume $\PP$ is some specific kinds of forcing. We shall mostly return to the rank $1$ cases for this and discuss the $\mathsf{club}$, $\mathsf{stat}$, $\mathsf{ub}$ and $\omega\text{-}\mathsf{ub}$ axioms in Figure \ref{diagram of implications}.

\subsubsection{$\sigma$-distributive forcings} 
We begin with a relatively simple case, where $\PP$ is ${<}\kappa$-distributive. In this case, several of our axioms turn out to be equivalent to one another. 
The implications for the class of ${<}\kappa$-distributive forcings are summarised in the next diagram. 

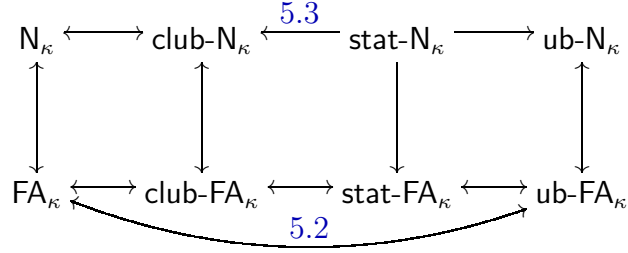
\begin{figure}[H]
\[ \xymatrix@R=3.5em{ 
{\txt{$\NP_\kappa$}} \ar@{<->}[r] \ar@{<->}[d] & {\txt{$\club\NP_\kappa$}} \ar@{<->}[d] \ar@{<-}[r]^{\ref{Remark_strength of statNsigma-closed}}  & \txt{$\stat\NP_\kappa$} \ar@{->}[r] \ar@{->}[d] & \txt{$\ub\NP_\kappa$} \ar@{<->}[d]& \\ 
{\txt{$\FA_\kappa$}} \ar@{<->}[r]& \txt{$\club\FA_\kappa$} \ar@{<->}[r] & \txt{$\stat\FA_\kappa$} \ar@{<->}[r] & \txt{$\ub\FA_\kappa$}\ar@/^0.8cm/@{<->}[lll]_{\ref{distributive stat}}  \\ 
}\]
\caption{Forcing axioms and name principles for any ${<}\kappa$-distributive forcing for regular $\kappa$. 
Lemma \ref{Remark_strength of statNsigma-closed} shows that $\stat\NP_{\PP,\omega_1}$ is strictly stronger than the remaining principles for some $\sigma$-closed forcing $\PP$. 
} 
\label{diagram of implications for distributive forcings} 
\end{figure}

\begin{lemma} 
\label{ubFA implies FA for sigma-distributive forcings} 
For any ${<}\kappa$-distributive forcing $\PP$, $\ub\FA_{\PP,\kappa} \implies\FA_{\PP,\kappa}$. 
\end{lemma} 
\begin{proof} 
Given a sequence $\vec{D}=\langle D_i\mid i<\kappa\rangle$ of open dense subsets of $\PP$, let $E_j=\bigcap_{i\leq j}D_i$ for $j<\kappa$. 
If for a filter $g$, $g\cap E_j\neq\emptyset$ for unboundedly many $j<\kappa$, then $g\cap D_i\neq \emptyset$ for all $i<\kappa$. 
\end{proof} 

\begin{lemma}
\label{distributive stat} 
	Let $\PP$ be ${<}\kappa$-distributive. $\stat\NP_{\PP,\kappa}\implies\FA^+_{\PP,\kappa}$
\end{lemma}
\begin{proof} 
Suppose that $\vec{D}=\langle D_i\mid i<\kappa\rangle$ is a sequence of open dense subsets of $\PP$ and $\sigma=\{(\check{\alpha},p)\mid p\in S_\alpha\}$ is a name with $1\forces_\PP ``\sigma$ is stationary''. 
For each $j<\kappa$, let $E_j=\bigcap_{i\leq j}D_i$. 
For $j<\kappa$ and $p\in\PP$, let $E_{j,p}$ denote a subset of $\{q\in E_j\mid q\leq p\}$ that is dense below $p$. 
Let 
$$\tau=\{(\check{\alpha},q)\mid \alpha<\kappa,\ \exists p\in S_\alpha\ q\in E_{j,p}\}.$$ 
$1\Vdash_\PP ``\tau$ is stationary'', since $1\Vdash_\PP \sigma=\tau$. 
By $\stat\NP_{\PP,\kappa}$, there is a filter $g$ such that $\tau^g$ is stationary. 
By the definition of $\tau$, $\tau^g\subseteq \sigma^g$. 
Thus $\sigma^g$ is stationary. 
We further have $g\cap E_j$ for unboundedly many $j<\kappa$ and hence $g\cap D_i\neq\emptyset$ for all $i<\kappa$. 
\end{proof}

An equivalent argument can be made with names for unbounded sets, or for sets containing a club.

\subsubsection{$\sigma$-closed forcings} 
Note that $\FA_{\PP,\omega_1}$ fails for some $\sigma$-distributive forcings, for instance for Suslin trees. 
But $\FA_{\sigma-\mathrm{closed},\omega_1}$ is provable: if $\langle D_\alpha \mid \alpha<\omega_1 \rangle$ is a sequence of dense subsets of a $\sigma$-closed $\PP$, let $\langle p_\alpha\mid \alpha<\omega_1\rangle$ be a decreasing sequence of conditions in $\PP$ with $p_\alpha\in D_\alpha$ and let $g=\{q \in \PP \mid \exists \alpha<\omega_1\ p_\alpha\leq q\}$. 
Therefore, the other principles in Figure \ref{diagram of implications for distributive forcings} are provable, with the exception of $\stat\NP_{\PP,\omega_1} $ by the next lemma. The lemma follows from known results. 

\begin{lemma} 
\label{Remark_strength of statNsigma-closed} 
It is consistent that there is a $\sigma$-closed forcing $\PP$ such that $\stat\NP_{\PP}$ fails. 
\end{lemma} 
\begin{proof} 
It suffices to argue that $\stat\NP_{\PP}$ has large cardinal strength for some $\sigma$-closed forcing $\PP$. 
Note that $\stat\NP_\PP$ implies $\FA^+_\PP$ for any $\sigma$-closed forcing $\PP$ by Lemma \ref{distributive stat}. 
There is a cardinal $\mu\geq\omega_2$ such that $\FA^+_{\mathrm{Col}(\omega_1,\mu)}$ implies the failure of $\square(\kappa)$ for all regular $\kappa\geq\omega_2$ by \cite[Page 20 \& Proposition 14]{foreman1988martin} and \cite[Theorem 2.1]{sakai2015stationary}.\footnote{A more direct argument using \cite[Page 20]{foreman1988martin} and \cite[Theorem 3.8]{velickovic1992forcing} should be possible, but the required results are not explicitly mentioned there.} 
The proofs show that a single collapse suffices for the conclusion. 
The failure of $\square(\kappa^+)$ and thus Jensen's $\square_\kappa$ at a singular strong limit cardinal $\kappa$ implies the existence of an inner model with a proper class of Woodin cardinals by \cite[Theorem 0.2]{trang2016pfa} and \cite[Corollary 0.7]{sargsyan2012strength}. 
\end{proof} 

Presaturation of the nonstationary ideal on $\omega_1$ is another interesting consequence of $\stat\NP_{\sigma\text{-}\mathrm{closed},\omega_1}$ (equivalently, of $\FA^+_{\sigma\text{-closed},\omega_1}$) \cite[Theorem 25]{foreman1988martin}. 
Even for very well-behaved $\sigma$-closed forcings $\PP$, $\stat\NP_{\PP,\omega_1}$ is an interesting axiom. 
For instance, Sakai shoed in \cite[Section 3]{sakaiMA} that $\FA^+_{\Add(\omega_1),\omega_1}$ and thus $\stat\NP_{\Add(\omega_1),\omega_1}$ is not provable in $\ZFC$. 
We have not studied the weakest stationary name principle for $\sigma$-closed forcing: 

\begin{question} 
Is $\stat\BN^1_{\sigma\text{-}\mathrm{closed}}$ provable in $\ZFC$? 
\end{question}

\subsubsection{c.c.c. forcings} 
\label{section ccc forcings} 
The class of c.c.c. forcings is rather more interesting. It has also historically been a class where forcing axioms have been frequently used; for example $\FA_{\text{c.c.c.},\omega_1}$ is the well-known Martin's Axiom $\MA_{\omega_1}$. 
Note that $\FA_{\PP,\kappa}$ is equivalent to $\BFA^\omega_{\PP,\kappa}$. 

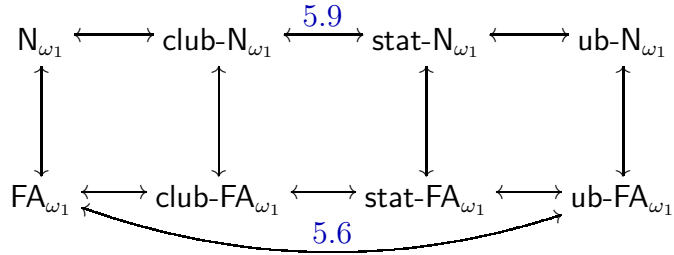
\begin{figure}[H]
	\[ \xymatrix@R=3.5em{ 
		{\txt{$\NP_{\omega_1}$}} \ar@{<->}[r] \ar@{<->}[d] & {\txt{$\club\NP_{\omega_1}$}} \ar@{<->}[d] \ar@{<->}[r]^{\ref{Baumgartner's lemma}}  & \txt{$\stat\NP_{\omega_1}$} \ar@{<->}[r] \ar@{<->}[d] & \txt{$\ub\NP_{\omega_1}$} \ar@{<->}[d]& \\ 
		{\txt{$\FA_{\omega_1}$}} \ar@{<->}[r]& \txt{$\club\FA_{\omega_1}$} \ar@{<->}[r] & \txt{$\stat\FA_{\omega_1}$} \ar@{<->}[r] & \txt{$\ub\FA_{\omega_1}$}\ar@/^0.8cm/@{<->}[lll]_{\ref{characterisation of precaliber}}  \\ 
	}\]
	\caption{Forcing axioms and name principles at $\omega_1$ for the class of all c.c.c. forcings. } 
	\label{diagram of implications for ccc forcings} 
\end{figure} 

All principles in Figure \ref{diagram of implications} for $\kappa=\omega_1$ turn out to be equivalent to $\FA_{\omega_1}$. 
The implications are valid for the class of all c.c.c. forcings, but not for all single c.c.c. forcings. 
For instance, for the class of $\sigma$-centred forcings, the right side of Figure \ref{diagram of implications} is provable in $\ZFC$ by Lemma \ref{Lemma stat-NP for sigma-centred}, but the left side is not.

We first derive the implication $\ub\FA_{c.c.c.,\omega_1} \Longrightarrow \FA_{c.c.c.,\omega_1}$ from well-known results. 
Note that this implication does not hold for individual c.c.c. forcings, for instance it fails for Cohen forcing by Lemma \ref{Lemma stat-NP for sigma-centred} and Remark \ref{Remark_FACohen_meagre}. 
We need the following definition: 

\begin{definition} 
\label{definition centred} 
Suppose that $\PP$ is a forcing. 
\begin{enumerate-(1)} 
\item 
A subset $A$ of $\PP$ is \emph{centred} if every finite subset of $A$ has a lower bound in $\PP$. 
$A$ is \emph{$\sigma$-centred} if it is a union of countably many centred sets. 
\item 
$\PP$ is \emph{precaliber $\kappa$} if, whenever $A\in [\PP]^\kappa$, there is some $B\in [A]^\kappa$ that is centred. 
\end{enumerate-(1)} 
\end{definition} 

The hard implications in the next lemma are due to Todor\v{c}evi\'c and Veli\v{c}kovi{\'c} \cite{todorcevic1987martin}. 

\begin{lemma} 
\label{characterisation of precaliber} 
The following conditions are equivalent: 
\begin{enumerate-(1)} 
\item 
\label{characterisation of precaliber 1} 
$\ub\FA_{\mathrm{c.c.c.},\omega_1}$ holds.  
\item 
\label{characterisation of precaliber 2} 
Every c.c.c. forcing is precaliber $\omega_1$. 
\item 
\label{characterisation of precaliber 3} 
Every c.c.c. forcing of size $\omega_1$ is $\sigma$-centred. 
\item 
\label{characterisation of precaliber 4} 
$\FA_{\mathrm{c.c.c.},\omega_1}$ holds. 
\end{enumerate-(1)} 
\end{lemma} 
\begin{proof} 
\ref{characterisation of precaliber 1}$\Rightarrow$\ref{characterisation of precaliber 2}: This follows immediately from the proof of \cite[Theorem 16.21]{jech2013set}. 
The proof only requires meeting unboundedly many dense sets. 

\ref{characterisation of precaliber 2}$\Rightarrow$\ref{characterisation of precaliber 3}: See \cite[Corollary 2.7]{todorcevic1987martin}. 

\ref{characterisation of precaliber 3}$\Rightarrow$\ref{characterisation of precaliber 4}: See \cite[Theorem 3.3]{todorcevic1987martin}.   

\ref{characterisation of precaliber 4}$\Rightarrow$\ref{characterisation of precaliber 1}: This is immediate. 
\end{proof}

Given Lemma \ref{characterisation of precaliber}, one wonders whether the characterisation also holds for $\sigma$-centered forcings instead of c.c.c. forcings. 
The next lemma together with the fact that $\FA_{\sigma\text{-centred}}$ is equivalent to $\mathfrak{p}>\omega_1$ (see  \cite[Theorem 3.1]{todorcevic1987martin}) shows that this is not the case. 

\begin{lemma} 
\label{Lemma stat-NP for sigma-centred} 
For any cardinal $\kappa$ with $\cof(\kappa)>\omega$, 
$\stat\NP_{\sigma\text{-}\mathrm{centred},\kappa}$ holds. 
\end{lemma} 
\begin{proof} 
Suppose that $\sigma$ is name for a stationary subset of $\omega_1$. 
Let $f\colon \PP\rightarrow \omega$ witness that $\PP$ is $\sigma$-centered. 
Let $S$ be the stationary set of $\alpha$ such that $p\forces \alpha\in \sigma$ for some $p\in\PP$. 
For each $\alpha\in S$, let $p_\alpha$ be such that $(\alpha,p_\alpha)\in \sigma$. 
There is a stationary subset $R$ of $S$ and $n\in\omega$ with $f(p_\alpha)=n$ for all $\alpha\in R$. 
Let $g$ be a filter containing $p_\alpha$ for all $\alpha\in S$. 
Then $R\subseteq \sigma^g$, as required. 
\end{proof} 

This suggests to ask whether $\FA_{\sigma\text{-}\mathrm{centred}}$ implies $\FA^+_{\sigma\text{-}\mathrm{centred}}$ as well. 

A long-standing open question asks whether one can replace precaliber $\omega_1$ by Knaster in the implication \ref{characterisation of precaliber 2}$\Rightarrow$\ref{characterisation of precaliber 4} of Lemma \ref{characterisation of precaliber}. 
Recall that a subset of $\PP$ is \emph{linked} if it consists of pairwise compatible conditions. 
$\PP$ is called \emph{Knaster} if, whenever $A\in [\PP]^{\omega_1}$, there is some $B\in [A]^{\omega_1}$ that is linked. 

\begin{question} \cite[Problem 11.1]{todorvcevic2011forcing} 
Does the statement that every c.c.c. forcing is  Knaster imply $\FA_{\mathrm{c.c.c.},\omega_1}$? 
\end{question} 

We now turn to the implication $\FA_{\mathrm{c.c.c.},\omega_1} \Longrightarrow \stat\NP_{\mathrm{c.c.c.},\omega_1}$. 
To this end, we reconstruct 
Baumgartner's unpublished result $\FA_{\mathrm{c.c.c.},\kappa} \Longrightarrow\FA_{\mathrm{c.c.c.},\kappa}^{+n}$ that is mentioned without proof in \cite[Section 8]{baumgartner1984applications} and \cite[Page 14]{beaudoin1991proper}. 
Here $\FA_\kappa^{+n}$ denotes the version of $\FA^+$ with $n$ many names for stationary subsets of $\kappa$. 

\begin{lemma}[Baumgartner] 
\label{Baumgartner's lemma} 
For any uncountable cardinal $\kappa$ and for any $n\in\omega$, 
$\FA_{\mathrm{c.c.c.},\kappa}$ implies $\FA_{\mathrm{c.c.c.},\kappa}^{+n}$. 
\end{lemma} 
\begin{proof} 
Suppose that for each $i<n$, $\sigma_i$ is a rank $1$ $\PP$-name for a stationary subset of $\omega_1$. 
For each $\vec{\alpha}=\langle \alpha_i\mid i< n\rangle \in \kappa^n$, let $A_{\vec{\alpha}}$ be a maximal antichain of conditions which strongly decide $\alpha\in\sigma_i$ for each $i<k$. 
Let $A=\bigcup_{\vec{\alpha}\in \kappa^n}A_{\vec{\alpha}}$. 
Since $\PP$ satisfies the c.c.c. and $|A|\leq\omega_1$, there exists a subforcing $\QQ \subseteq \PP$ with $A\subseteq \QQ$ and $|\QQ|\leq\omega_1$ such that compatibility is absolute between $\PP$ and $\QQ$. 
In particular, $\QQ$ is c.c.c. 

Since every c.c.c. forcing of size $\omega_1$ is $\sigma$-centred by $\MA_{\omega_1}$ (see \cite[Theorem 4.5]{weiss1984versions}), there is a sequence $\vec{g}=\langle g_k\mid k\in\omega\rangle$ of filters $g_k$ on $\PP$ with $\QQ\subseteq \bigcup_{k\in\omega} g_k$. 
Morover, it follows from the proof of \cite[Theorem 4.5]{weiss1984versions} (by a density argument) that we can choose the filters $g_k$ such that $g_k\cap B_\alpha\neq\emptyset$ for all $(k,\alpha)\in \omega\times \kappa$, where $\vec{B}=\langle B_\alpha\mid \alpha<\kappa\rangle$ is any sequence of dense subsets of $\PP$. (The conditions in the c.c.c. forcing consists of finite approximations to finitely many filters.) 

It remains to  find some $k\in\omega$ such that for all $i<n$, the set $\sigma_i^{g_k}$ is stationary. 
Let $G$ be $\PP$-generic over $V$. 
We claim that   
$$\prod_{i<n} \sigma_i^G \subseteq \bigcup_{k\in\omega}\ \prod_{i<n}\sigma_i^{g_k}.$$ 
To see this, suppose that $\vec{\alpha}=\langle \alpha_i\mid i<n\rangle \in \prod_{i<n} \sigma_i^G $ and let $p \in A_{\vec{\alpha}}\cap G$. 
Then $p\Vdash^+ \bigwedge_{i<n} \alpha_i\in \sigma_i$. 
Since $p\in\QQ$, we have $p\in g_k$ for some $k\in\omega$. 
Hence $\vec{\alpha} \in \prod_{i<n} \sigma_i^{g_k}$. 
Since $\sigma_i^G$ is stationary for all $i<n$, the above inclusion easily yields that there is some $k\in\omega$ such that $\prod_{i<n} \sigma_i^{g_k}$ is stationary. 
\end{proof} 

Our proof of the previous lemma does not work for $\MA^{+\omega}$. 
In fact, Baumgartner asked in \cite[Section 8]{baumgartner1984applications}: 

\begin{question}[Baumgartner 1984] Does $\MA_{\omega_1}$ imply 
$\MA^{+\omega_1}_{\omega_1}$? 
\end{question} 

We finally turn to bounded name principles for c.c.c. forcings. 

\begin{lemma} \ 
\label{clubBNccc} 
\begin{enumerate-(1)} 
\item 
\label{clubBNccc 1}
$\club\BN_{\mathrm{c.c.c.}}^1$ holds.
\item
\label{clubBNccc 2}
For any c.c.c. forcing $\PP$, 
$\ub\BN_{\PP}^1$ implies $\ub\FA_\PP$. 
\end{enumerate-(1)} 
\end{lemma} 
\begin{proof} 
\ref{clubBNccc 1} 
If $\sigma$ is a $\PP$-name for a set that contains a club, then by the c.c.c. there is a club $C$ with $1\forces C\subseteq \sigma$. 
Since $\sigma$ is $1$-bounded, $(\alpha,1)\in \sigma$ for all $\alpha\in C$. 
Thus for every filter $g$, we have $C\subseteq \sigma^g$. 

\ref{clubBNccc 2} 
Suppose that $\PP$ satisfies the c.c.c. 
Suppose that $\vec{D}=\langle D_\alpha\mid \alpha<\omega_1\rangle$ is a sequence of dense subsets of $\PP$. 
Let $A_\alpha$ be a maximal antichain in $D_\alpha$ and let $\vec{a}_\alpha=\langle a^n_\alpha\mid n\in\omega\rangle$ enumerate $A_\alpha$. 
(For ease of notation, we assume for that each $A_\alpha$ is infinite.) 
Let $\sigma=\{(\omega\cdot\alpha + n, a^n_\alpha)\mid \alpha<\omega_1,\ n\in\omega\}$. 
By $\ub\BN^1_\PP$, there is a filter $g$ such that $\sigma^g$ is unbounded. 
Hence $D_\alpha\cap g\neq\emptyset$ for unboundedly many $\alpha<\omega_1$. 
\end{proof} 

For any c.c.c. forcing $\PP$, the principles $\ub\BN_{\PP}^1$, $\ub\NP_{\PP}$ and $\ub\FA_{\PP}$ are equivalent by Lemma \ref{clubBNccc} \ref{clubBNccc 2} and the implications in Figure \ref{diagram of implications for ccc forcings}. 
We do not know what is their relationship with $\stat\BN^1_{\mathrm{c.c.c.}}$. 
However, we will show in Lemma \ref{CH implies failure of statN(random)} below that $\stat\BN^1_{\mathrm{random},\omega_1}$ is not provable in $\ZFC$. 

Regarding Lemma \ref{clubBNccc} \ref{clubBNccc 1}, it is also easy to see that $\club\BN_{\sigma\text{-}\mathrm{closed}}^1$ is provable. 
This suggests to ask: 

\begin{question} 
Is $\club\BN_{\PP}^1$ is provable for any proper forcing $\PP$? 
\end{question}

\subsection{Specific forcings}

\subsubsection{Cohen forcing} 
\label{section - Cohen forcing} 

We will now drop down from classes of forcings, to forcing axioms on specific forcings. This is also where we prove most of the negative results in the diagram from earlier. We start with the simplest, Cohen forcing and let $\kappa=\omega_1$. 
For Cohen forcing, all principles in the right part of the next diagram are provable in $\ZFC$ by Lemma \ref{Lemma stat-NP for sigma-centred} (on $\sigma$-centred forcing) and the basic implications in Figure \ref{diagram of implications}. 
The left part is not provable by Remark \ref{Remark_FACohen_meagre} below. 

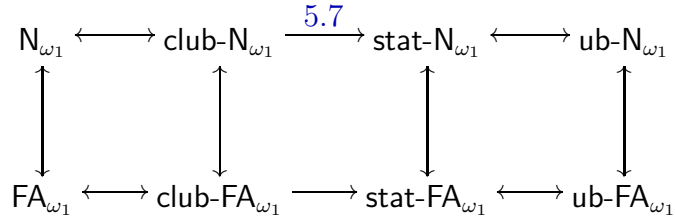
\begin{figure}[H]
\[ \xymatrix@R=3.5em{ 
{\txt{$\NP_{\omega_1}$}} \ar@{<->}[r] \ar@{<->}[d] & {\txt{$\club\NP_{\omega_1}$}} \ar@{<->}[d] \ar@{->}[r]^{\ref{Lemma stat-NP for sigma-centred}}  & \txt{$\stat\NP_{\omega_1}$} \ar@{<->}[r] \ar@{<->}[d] & \txt{$\ub\NP_{\omega_1}$} \ar@{<->}[d]& \\ 
{\txt{$\FA_{\omega_1}$}} \ar@{<->}[r]& \txt{$\club\FA_{\omega_1}$} \ar@{->}[r] & \txt{$\stat\FA_{\omega_1}$} \ar@{<->}[r] & \txt{$\ub\FA_{\omega_1}$}  \\ 
}\]
\caption{Forcing axioms and name principles at $\omega_1$ for Cohen forcing. } 
\label{diagram of implications for Cohen forcing} 
\end{figure} 

Our first result is an improvement to Lemma \ref{Lemma stat-NP for sigma-centred}. 
It shows that a simultaneous version of the stationary forcing axiom for countably many sequences of dense sets holds. 

\begin{lemma}  
Let $\PP$ be Cohen forcing and $\kappa$ a cardinal with $\cof(\kappa)>\omega$. 
For each $n\in\omega$, let $\vec{D}_n=\langle D^n_\alpha \mid \alpha<\omega_1 \rangle$ be a sequence of dense sets. 
Then there exists a filter $g\in V$ such that for all $n$, the trace $\Tr_{g,\vec{D}_n}$ is stationary in $\kappa$.\footnote{See Definition \ref{notation_trace}.} 
\end{lemma} 
\begin{proof} 
Suppose that there is no filter $g$ as described.
For $x\in 2^\omega$, let us write $g_x$ to denote the filter $\{x{\upharpoonright}n : n\in\omega\}$.
Then for each $x\in 2^\omega$, the filter $g_x$ does not have the required property. 
So there is a natural number $n_x$ and a club $C_x\subseteq \kappa$ with $g_x\cap D^{n_x}_\alpha=\emptyset$ for all $\alpha\in C_x$.
Then the sets $A_n:=\{x\in 2^\omega\mid  n_x=n\}$ partition $2^\omega$. 
By the Baire Category Theorem, not all $A_n$ are nowhere dense. 
So there is some $n\in\omega$ and basic some open subset $N_t=\{x \in 2^\omega \mid t \subseteq x\}$ for some $t\in 2^{<\omega}$ such that $A_n\cap N_t$ is dense in $N_t$. 
Fix a countable set $D\subseteq A_n\cap U$ which is dense in $U$. 
Let $\alpha$ be an element of the club $\bigcap_{x\in D}C_x$. 
Let further $u\in D^n_\alpha$ with $u\leq t$. 
Since $D$ is dense in $N_t$, there is some $x\in D\cap N_u$. 
Then $u\in g_x\cap D^n_\alpha$ and hence $g_x \cap D^n_\alpha\neq\emptyset$. 
On the other hand, we have  $x\in A_n$ and hence $n_x=n$. Since also $\alpha\in C_x$, we have  $g_x \cap D^n_\alpha=\emptyset$.
\end{proof} 

Using a variant of the previous proof, we can also improve $\stat\NP_\PP$ to work for finitely many names. 

\begin{lemma} 
\label{lemma_omega-FA-Cohen} 
Let $\PP$ be Cohen forcing and $\kappa$ a cardinal with $\cof(\kappa)>\omega$. 
Suppose that $\vec{\sigma}=\langle \sigma_i \mid i\leq n\rangle$ is a sequence of rank $1$ $\PP$-names such that for each $i\leq n$, $
\PP\Vdash \sigma_i$ is stationary in $\kappa$. 
Then there is a filter $g$ on $\PP$ such that for all $i\leq n$, $\sigma_i^g$ is stationary in $\kappa$. 
In particular, $\stat\NP_{\PP,\kappa}$ holds. 
\end{lemma} 
\begin{proof}
As in the previous proof, let $g_x=\{x{\upharpoonright}n : n\in\omega\}$ for $x\in 2^\omega$. 
The result will follow from the next claim. 

\begin{claim} 
\label{claim_omega-FA-Cohen} 
If $D$ is any dense subset of $2^\omega$, then there is some $x\in D$ such that $\sigma_i^{g_x}$ is stationary in $\kappa$ for all $i\leq n$. 
\end{claim} 
\begin{proof} 
We can assume that $D$ is countable. 
If the claim fails, then for each $x\in D$, there is  some $i\leq n$ and a club $C_x$ such that $\sigma_i^{g_x} \cap C_x=\emptyset$. 
Then $C:=\bigcap_{x\in D} C_x$ is a club. 
Moreover, for each $x\in D$, there is some $i\leq n$ such that $\sigma_i^{g_x} \cap C=\emptyset$. 
There is some $p\in \PP$ such that for each $i\leq n$, there is some $\alpha_i\in C$ such that $p\forces \check{\alpha}_i\in\sigma_i$. 
By Lemma \ref{Prop_sforcingAndForcing}, we can assume that $p\sforces \check{\alpha}_i\in\sigma_i$ for all $i\leq n$. 
Now, since $D$ is dense, we can find some $x\in D$ with $p\subseteq x$. 
Then $p\in g_x$, so by \ref{Prop_sforcingAndInterpretation} we conclude $\alpha_i\in \sigma_i^{g_x}$ for all $i\leq n$. 
This contradicts the above property of $C$. 
\end{proof} 
This completes the proof of Lemma \ref{lemma_omega-FA-Cohen}. 
\end{proof} 

Given the previous result about $\stat\FA$, we might expect to be able to correctly interpret $\omega$ many names. But 
the above proof does not work: 
it breaks down where we introduce $p$. For each $i$, we can find $p_i$ strongly forcing $\alpha_i\in \sigma_i$; but then we would want to take some $p$ that was below every $p_i$ and that is only possible in $\sigma$-closed forcings.

We can, however, apply the same technique in the presence of $\FA$ to prove $\FA^+$.

\begin{lemma}
\label{lemma_FA_implies_FA+} 
Let $\PP$ be Cohen forcing and $\kappa$ a cardinal with $\cof(\kappa)>\omega$. 
Then $\FA_\PP$ implies $\FA^+_\PP$. 
\end{lemma} 
\begin{proof} 
We will in fact prove a stronger version for finitely many names. 
Suppose that $\vec{\sigma}=\langle \sigma_i \mid i\leq n\rangle$ is a sequence of rank $1$ $\PP$-names such that for each $i\leq n$, $
\PP\Vdash \sigma_i$ is stationary in $\kappa$. 
Suppose that $\vec{D}=\langle D_\alpha\mid  \alpha<\kappa\rangle$ is a sequence of dense open sets. 
Then 
$$D:=\{x\in 2^\omega\mid \forall \alpha<\kappa\ \exists p\in D_\alpha\ p\subseteq x\}$$ 
consists of all reals $x$ such that $g_x \cap D_\alpha\neq\emptyset$ for all $\alpha<\omega_1$. 

The next claim suffices. 
By Claim \ref{claim_omega-FA-Cohen}, it implies that for some $x\in D$, $\sigma_i^{g_x}$ is stationary for all $i\leq n$. 

\begin{claim} 
$D$ is dense in $2^\omega$. 
\end{claim} 
\begin{proof} 
Fix $q\in \PP$; we will find some $x\in D$ with $q\subseteq x$. 
Since the forcing $\PP_q:=\{p\in \PP\mid p\leq q\}$ is isomorphic to Cohen forcing via the map $r\mapsto q^\smallfrown r$, $\FA_{\PP_q}$ holds. Hence, we can find a filter $g$ on $\PP_q$ which meets $D_\alpha \cap \PP_q$ for every $\alpha<\omega_1$. 
$\cup g$ is an element of $2^{\leq \omega}$ with $q\subseteq \cup g$ by compatibility of elements of a filter. 
Then any real $x$ with $\cup g \subseteq x$ satisfies $x\in D$ and $q\subseteq x$. 
\end{proof} 
Lemma \ref{lemma_FA_implies_FA+} follows. 
\end{proof}

\begin{remark} 
\label{Remark_FACohen_meagre} 
Note that $\FA_{\text{Cohen},\omega_1}$ also has a well known characterisation via sets of reals: 
it is equivalent to the statement that the union of $\omega_1$ many meagre sets does not cover $2^\omega$. 
In particular, $\FA_{\text{Cohen},\omega_1}$ is not provable in $\ZFC$. 
\end{remark}

\subsubsection{Random forcing} 

The product topology on $2^\omega$ is induced by the basic open sets $N_t= \{ x\in 2^\omega \mid t\subseteq x \}$ for $t\in 2^{<\omega}$. 
\emph{Lebesgue measure} is by definition the unique measure $\mu$ on the Borel subsets of $2^\omega$ with $\mu(N_t)=\frac{2}{2^n}$. 

\begin{definition} 
\emph{Random forcing} $\PP$ is the set of Borel subsets of $2^\omega$ with positive Lebesgue measure. 
$\PP$ is quasi-ordered by inclusion, i.e. $p\leq q :\Leftrightarrow p\subseteq q$ for $p,q \in \PP$. 
\end{definition} 

Strictly speaking, random forcing is the partial order obtained by taking the quotient of the preorder, where two conditions are equivalent if their symmetric difference has measure $0$. 
To simplify notation, we will talk about Borel sets of positive measure as if they were conditions in random forcing.

\begin{figure}[H]
\[ \xymatrix@R=3.5em{ 
{\txt{$\NP_{\omega_1}$}} \ar@{<->}[r] \ar@{<->}[d] & {\txt{$\club\NP_{\omega_1}$}} \ar@{<->}[d] \ar@{<..}[r]^{\ref{Lemma random statN implies FA+}}  & \txt{$\stat\NP_{\omega_1}$} \ar@{..>}[r] \ar@{..>}[d] & \txt{$\ub\NP_{\omega_1}$} \ar@{<->}[d]& \\ 
{\txt{$\FA_{\omega_1}$}} \ar@{<->}[r]& \txt{$\club\FA_{\omega_1}$} \ar@{<->}[r] & \txt{$\stat\FA_{\omega_1}$} \ar@{<->}[r] & \txt{$\ub\FA_{\omega_1}$}\ar@/^0.8cm/@{<->}[lll]_{\ref{Lemma_Random ubFA}}  \\ 
}\]
\caption{Forcing axioms and name principles at $\omega_1$ for random forcing. } 
\label{diagram of implications for random forcing} 
\end{figure}
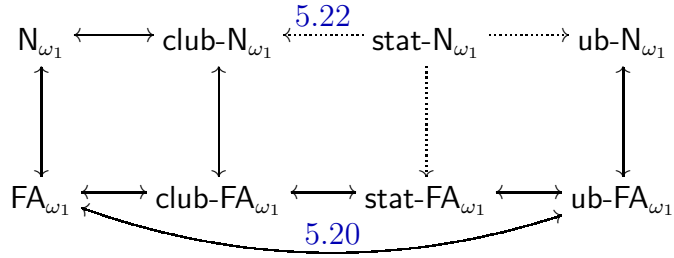

We have seen in Lemma \ref{Lemma stat-NP for sigma-centred} and the following remark that $\ub\FA_\PP$ implies $\FA_\PP$ for $\sigma$-centred forcings. 
However, random forcing is not $\sigma$-centred by \cite[Lemma 3.7]{brendle2009forcing}. 
The implication still holds: 

\begin{lemma}
	\label{Lemma_Random ubFA}
	Let $\PP$ denote random forcing. The following are equivalent:
	\begin{enumerate-(1)}
		\item 
		\label{Lemma_Random ubFA 1}
		$\FA_{\PP,\omega_1}$
		\item 
		$\ub\FA_{\PP,\omega_1}$
	    \label{Lemma_Random ubFA 2}
		\item 
		$2^\omega$ is not the union of $\omega_1$ many null sets
        \label{Lemma_Random ubFA 3}
	\end{enumerate-(1)}
\end{lemma}

The equivalence of \ref{Lemma_Random ubFA 1} and \ref{Lemma_Random ubFA 3} is a well-known fact, but we really interested in the equivalence of \ref{Lemma_Random ubFA 1} and \ref{Lemma_Random ubFA 2}.
The proof of \ref{Lemma_Random ubFA 2}$\Rightarrow$\ref{Lemma_Random ubFA 3} also works for certain  forcings of the form $\PP_I$. 
$\PP_I$ consists of all  Borel subsets $B\notin I$ of $2^\omega$, where $I$ is a $\sigma$-ideal on the Borel subsets of the Cantor space, ordered by inclusion up to sets in $I$. 
For  \ref{Lemma_Random ubFA 2}$\Rightarrow$\ref{Lemma_Random ubFA 3}, it suffices that the set of closed $p\in \PP$ is dense in $\PP$ and $N_t\notin I$ for all $t\in 2^{<\omega}$. 
If additionally \ref{Lemma_Random ubFA 3}$\Rightarrow$\ref{Lemma_Random ubFA 1} holds, then $\ub\FA_{\PP_I,\omega_1}$ implies $\FA_{\PP_I,\omega_1}$.

\begin{proof}
	\ref{Lemma_Random ubFA 1}$\Rightarrow$\ref{Lemma_Random ubFA 2}: Immediate.
	
	\ref{Lemma_Random ubFA 2}$\Rightarrow$\ref{Lemma_Random ubFA 3}: We prove the contrapositive. Suppose $2^\omega = \bigcup_{\alpha<\omega_1} S_\alpha$, where $S_\alpha \subseteq 2^\omega$ has measure $0$. Without loss of generality, we may assume that $\langle S_\alpha \mid \alpha<\omega_1 \rangle$ is an increasing sequence; i.e. $\alpha<\beta<\omega_1\Rightarrow S_\alpha \subseteq S_\beta$. 
Then 
$$D_\alpha=\{B\in \PP\mid  B\subseteq 2^\omega\setminus S_\alpha \text{ and } B \text{ is closed}\}$$ 
is dense. 

Let $g\in V$ be a filter. Without loss of generality, assume $g$ is an ultrafilter. 
Then for any $n\in\omega$, there is some $t\in 2^n$ with $N_t\in g$. 
It follows that there is a unique $x\in 2^\omega$ such that $N_t\in g$ for all $t\subseteq x$. 
It is easy to check that $x$ is in the closure of any element of $g$. 

Towards a contradiction, suppose that for unboundedly many $\alpha$ we can find $B_\alpha\in D_\alpha\cap g$. Then $B_\alpha$ is closed, so $x\in B_\alpha\subseteq 2^\omega \setminus S_\alpha$ so $x\not \in S_\alpha$. 
This contradicts the assumptions that  
$2^\omega = \bigcup S_\alpha$ and the $S_\alpha$ are increasing.
	
	\ref{Lemma_Random ubFA 3}$Rightarrow$\ref{Lemma_Random ubFA 1}: Again we prove the contrapositive. Let $\langle D_\alpha\mid  \alpha<\omega_1\rangle$ be a sequence of predense sets such that there is no filter in $V$ meeting all of them. $\PP$ has the c.c.c., so without loss of generality we may assume every $D_\alpha$ is countable.
	
	Fix the following notation. 
	Recall that $x\in 2^\omega$ is a density point of $B$ if $\frac{\mu(B\cap N_{(x\vert k)})}{\mu(N_{(x\vert k)})}$ tends to $1$ as $k$ tends to infinity. 
	For $B\in\PP$, let $D(B)$ be the set of density points of $B$. 	
	For $\alpha<\omega_1$, let 
	$$T_\alpha=\bigcup_{B\in D_\alpha}D(B) \text{ \ \ \ and \ \ \ } S_\alpha =2^\omega\setminus T_\alpha.$$ 
	
	We first show that $S_\alpha$ is a null set. 
	To see this, suppose that $S_\alpha$ has positive measure. Then we can find a closed subset $C\subseteq S_\alpha$ with positive measure. 
	Since $D_\alpha$ is predense, we can find some $B\in D_\alpha$ with $\mu(B\cap C)>0$. 
	 Since $B\triangle D(B)$ is null by Lebesgue's Density Theorem, we have $\mu(D(B)\cap C)>0$. 
	 This contradicts $D(B)\cap C\subseteq  T_\alpha \cap C= \emptyset$. 
	 
	
	We now show $\bigcup_{\alpha<\omega_1}S_\alpha=2^\omega$. 
	To see this, take any $x\in 2^\omega$ and let 
	$$g_x=\{B\in \PP: x\in D(B)\}$$ 
	denote the filter generated by $x$. 
    Take $\alpha<\omega_1$ such that $g_x\cap D_\alpha=\emptyset$. 
    We show that $x\in S_\alpha$, as required. 
	Otherwise $x\in T_\alpha$, 
	so we can find $B\in D_\alpha$ with $x\in D(B)$. But then $B\in g_x\cap D_\alpha$. 
	This contradicts $g_x\cap D_\alpha=\emptyset$. 
\end{proof}

Combining the proofs of 	\ref{Lemma_Random ubFA 2}$\Rightarrow$\ref{Lemma_Random ubFA 3} and 	\ref{Lemma_Random ubFA 3}$\Rightarrow$\ref{Lemma_Random ubFA 1}, we can obtain the following refinement:

\begin{lemma}
\label{corollary_dense subsets for ubFA and FA} 
Let $\PP$ be random forcing. 
	Let $\langle D_\alpha\mid \alpha<\omega_1\rangle$ be a collection of predense sets. 
	There exists another collection $\langle D'_\alpha\mid \alpha<\omega_1\rangle$ of dense sets, such that if a filter $g$ meets unboundedly many $D'_\alpha$, then it can be extended to a filter $g'$ which meets every $D_\alpha$.
\end{lemma}
\begin{proof}
	Define $S_\alpha$ as in the proof of 	\ref{Lemma_Random ubFA 3}$\Rightarrow$\ref{Lemma_Random ubFA 1}. 
	Then for any $x\in 2^\omega$, we have $g_x\cap D_\alpha\neq\emptyset$ or $x \in S_\alpha$. 	
	Consider the null sets $S'_\alpha=\bigcup_{\beta<\alpha}S_\beta$. 
	Then define $D'_\alpha$ from $S'_\alpha$ in the same way we defined $D_\alpha$ from $S_\alpha$ in the proof of \ref{Lemma_Random ubFA 2}$\Rightarrow$\ref{Lemma_Random ubFA 3}. 
As in the proof of \ref{Lemma_Random ubFA 2}$\Rightarrow$\ref{Lemma_Random ubFA 3}, we obtain the following for any  $x \in 2^\omega$ and $\alpha<\omega_1$: if $g_x\cap D'_\alpha \neq\emptyset$, then $x\notin S'_\alpha$. 
Let $g$ be a filter which meets unboundedly many $D'_\alpha$. 
Then $g\subseteq g_x$ for some $x\in 2^\omega$. 
We have seen that $x\notin S'_\alpha$ for unboundedly many $\alpha$. 
Therefore $x$ misses all $S'_\alpha$ and all $S_\alpha$. 
By the choice of the $S_\alpha$, we have $g_x\cap D_\alpha\neq\emptyset$ for all $\alpha<\omega_1$. 
\end{proof}

This then allows us to prove that $\stat\NP$ alone gives us the full $\FA^+$.

\begin{lemma} 
\label{Lemma random statN implies FA+} 
Let $\PP$ be random forcing. 
Then $\stat\NP_\PP\implies\FA^+_\PP$. 
\end{lemma} 
\begin{proof} 
Suppose that $\langle D_\alpha \mid \alpha <\omega_1\rangle$ is a sequence of dense subsets of $\PP$. 
Suppose further that $\sigma$ 
is a rank $1$ name which is forced to be stationary. 
Let $\langle D'_\alpha \mid \alpha<\omega_1\rangle$ be a sequence as in Lemma \ref{corollary_dense subsets for ubFA and FA} and 
$$ \tau = \{ (\check{\alpha}, p) \mid p\in D'_\alpha \wedge p \sforces \check{\alpha}\in \sigma  \}.  $$
Note that $\PP\forces \sigma=\tau$. 
By $\stat\NP_\PP$, we obtain a filter $g$ such that $\tau^g$ is stationary. 
Since $\tau^h \subseteq \sigma^h$ for all filters $h$, $\sigma^g$ is stationary as well. 
Moreover, 
$g \cap D'_\alpha\neq \emptyset$, for stationarily many $\alpha$. 
By the choice of $\langle D'_\alpha \mid \alpha<\omega_1\rangle$, we can extend $g$ to a filter $g'$ such that $g'\cap D_\alpha\neq \emptyset$ for all $\alpha<\omega_1$. 
Moreover, $\sigma^g \subseteq \sigma^{g'}$, so $\sigma^{g'}$ is stationary. 
\end{proof} 

The missing link in Figure \ref{diagram of implications for random forcing} is: 

\begin{question} 
If $\PP$ denotes random forcing, does $\FA_{\PP,\omega_1}$ imply $\stat\NP_{\PP,\omega_1}$? 
\end{question}

We finally show that the $1$-bounded stationary name principle for random forcing is non-trivial, as we discussed at the end of Section \ref{section ccc forcings}. 
 
\begin{lemma} 
\label{CH implies failure of statN(random)}
Let $\kappa=2^{\aleph_0}$ and assume that every set of size ${<}\kappa$ is null.\footnote{This assumption is equivalent to $\mathrm{non}(\mathrm{null})=2^{\aleph_0}$. It follows from $\MA$, but not from $\FA_{\mathrm{random}}$ by known facts about Cichon's diagram.} 
Then $\stat\BN_{\PP,\kappa}^1$ fails for random forcing $\PP$. 
In particular, $\CH$ implies that $\stat\BN_{\PP,\omega_1}^1$ fails. 
\end{lemma} 
\begin{proof} 
It suffices to show that $\stat\BN_{\PP,\kappa}^\omega$ fails. 
To see this, apply Corollary \ref{corollary_equivalence of lambda-bounded and 1-bounded name principles} and use the fact that random forcing is well-met and for any $q\in \PP$, the forcing $\PP_q$ is isomorphic to $\PP$ by \cite[Theorem 17.41]{kechris2012classical}. 
Let $\vec{x}=\langle x_\alpha \mid \alpha<\kappa\rangle$ enumerate all reals. 
Then $C_\beta:= \{ x_\alpha \mid \alpha<\beta  \}$ is null for all $\beta<\kappa$. 
For each $\alpha<\kappa$, let $A_\alpha$ be a countable set of approximations to the complement of $C_\alpha$ in the following sense: 
\begin{enumerate-(a)} 
\item 
Each element of $A_\alpha$ is a closed set disjoint from $C_\alpha$, and 
\item 
\label{condition for random forcing 2} 
For all $\epsilon>0$, $A_\alpha$ contains a set $C$ with $\mu(C)\geq 1-\epsilon$. 
\end{enumerate-(a)} 
Let $\sigma= \{ (\check{\alpha},p) \mid p \in A_\alpha \}$. 
Then $\Vdash_{\PP} \sigma$ is stationary, since each $A_\alpha$ is predense by \ref{condition for random forcing 2}. 
We claim that there is no filter $g$ in $V$ such that $\sigma^g$ is unbounded. 
If $g$ were such a filter, then we could assume that for every $n\in\omega$, $g$ contains $N_{t_n}$ for some (unique) $t_n\in 2^n$ by extending $g$. (Clearly $\sigma^g$ will remain unbounded.) 
Let $x=\bigcup_{n\in\omega} t_n$ and 
suppose that $x=x_\alpha$. 
Since $\sigma^g$ is unbounded, there is some $\gamma>\alpha$ in $\sigma^g$. 
Find some $p\in A_\gamma$ with $p\in g$. 
By the definition of $A_\gamma$, $p$ is a closed set with $x_\alpha\notin p$. 
Hence $p \cap N_{t_n}=\emptyset$ for some $n\in\omega$. 
But this contradicts the fact that both $p$ and $N_{t_n}$ are in $g$. 
\end{proof}

\subsubsection{Hechler forcing} 
For $\sigma$-centred forcings $\PP$, the principles on the right side of Figure \ref{diagram of implications} are provable in $\ZFC$ (see Lemma \ref{Lemma stat-NP for sigma-centred}). 
A subtle difference appears when we add the requirement that the filter has to meet countably many fixed dense sets. 
We write $\omega\text{-}\ub\FA$ for this axiom (see Definition \ref{Defn_SpecialFA}). 
For some forcings, this axiom is stronger that $\ub\FA$. 
To see this, we will make use of the fact that for Hechler forcing, a filter that meets certain countably many dense sets corresponds to a real. 
Recall that a subset $A\subseteq \omega^\omega$ is \emph{unbounded} if no $y\in \omega^\omega$ eventually strictly dominates all $x\in A$, i.e. $\exists m\ \forall n\geq m\ x(n)< y(n)$. 
The next result shows that $\omega\text{-}\ub\FA_{\omega_1}$ for Hechler forcing implies the negation of the continuum hypothesis. 

\begin{lemma}
	Let $\PP$ denote Hechler forcing. If $\omega\text{-}\ub\FA_\PP$ holds, then the size of any unbounded family is at least $\omega_2$. 
\end{lemma}
\begin{proof}
Towards a contradiction, suppose $\omega\text{-}\ub\FA_\PP$ holds and $A$ is an unbounded family of size $\omega_1$. 
Let us enumerate its elements as $\vec{x}=\langle x_\alpha \mid \alpha<\omega_1 \rangle$. 
	We define the following dense sets: 
	For $\alpha<\omega_1$, we define a real $y_\alpha$ by taking a sort of ``diagonal maximum'' of $\vec{x}$. 
	Let $\pi:\alpha \rightarrow \omega$ be a bijection and let 
		$$y_\alpha(n)=\max \{x_\gamma(n): \pi(\gamma)\leq n\}.$$  
	It is easy to check that $y_\alpha$ is well defined, and that it eventually dominates $x_\gamma$ for all $\gamma<\alpha$.
	We now define
	\begin{equation*}
		D_\alpha = \{(s,x)\in \PP: x \text{ eventually strictly dominates }y_\gamma\}
	\end{equation*}
			For $n<\omega$, let
	\begin{equation*}
		E_n=\{(s,x)\in \PP: \text{length}(s)\geq n\}
	\end{equation*}

	Now let $g\in V$ be a filter meeting unboundedly many $D_\alpha$ and all $E_n$. Since $g$ meets all $E_n$, the first components of its conditions are arbitrarily long. Since all its elements are compatible, this means that the union $\cup \{s: (s,x)\in g\}$ is a real $y$. And $y$ must eventually strictly dominate $x$ for every $(s,x)\in g$. 
	But there are unboundedly many $\alpha$ such that $g$ meets $D_\alpha$. For any such $D_\alpha$, then, we have $(s,x)\in g$ where $x$ eventually strictly dominates $y_\alpha$. 
	Hence, $y$ must eventually strictly dominate unboundedly many $y_\alpha$ and hence every $x \in A$. 
	But $A$ was assumed to be unbounded. 
\end{proof}

%


\subsubsection{Suslin trees} 

A Suslin tree is a tree of height $\omega_1$, with no uncountable branches or antichains. The existence of Suslin trees is not provable from $\ZFC$, but follows from $\diamondsuit_{\omega_1}$. We can of course think of a Suslin  tree $T$ as a forcing; it will add a cofinal branch through the tree. 
We use Suslin trees as test cases for the weakest principles defined above. 
As expected, we can show that $\stat\BN^1_{T,\omega_1}$ fails in most cases. 

\begin{lemma}
	Suppose $T$ is a Suslin tree. Then $\stat\BN^\omega_{T,\omega_1}$ fails.
\end{lemma}
\begin{proof}
	Let $\sigma=\{\langle \alpha,p\rangle: \alpha<\omega_1, p\in T, \text{height}(p)=\alpha\}$. It is easy to see that $\sigma$ is $\omega$ bounded, and is forced to be equal to $\omega_1$. But any filter $g\in V$ is a subset of a branch in $V$, and therefore countable. So $\sigma^g$ is not stationary, or even unbounded.
\end{proof}

\begin{corollary}
\label{Suslin trees} 
	Suppose that a Suslin tree exists. 
	Then there exists a Suslin tree $T$ such that  $\stat\BN^1_{T,\omega_1}$ fails.
\end{corollary}
\begin{proof}
	Let $T$ be any Suslin tree. By the previous lemma we know that $\stat\BN^\omega_{T,\omega_1}$ fails. But then by Corollary \ref{lemma_failure of lambda-bounded and 1-bounded name principle for trees}, $T$ contains a subtree $S$ such that $\stat\BN^1_{S,\omega_1}$ fails.
\end{proof}

This also tells us that $\stat\BN^1_{\PP,\omega_1}$ is not equivalent to $\stat\BFA^1_{\PP,\omega_1}$, since the latter is trivially provable for any forcing in $\ZFC$. 

In fact, if we assume $\diamondsuit_{\omega_1}$ (which is somewhat stronger than the existence of a Suslin tree, see \cite[Section 3]{rinot2011jensen}) then we can do better than this: we can show that $\stat\BN^1_{\omega_1}$ fails for every Suslin tree. 

\begin{lemma} 
\label{Lemma diamond, Suslin tree and failure of statBN1} 
Suppose $\diamondsuit_{\omega_1}$ holds. If $T$ is a Suslin tree, then $\stat\BN^1_{T,\omega_1}$ fails.
\end{lemma} 
\begin{proof} 

	Let $(A_\gamma)$ be the sequence given by $\diamondsuit_{\omega_1}$. That is, let it be such that $A_\gamma \subseteq \gamma$ and for any $S\subseteq \omega_1$, the set $\{\gamma<\omega_1:S\cap \gamma=A_\gamma\}$ is stationary. We build up a rank 1 name $\sigma=\{\langle \check{\alpha},p\rangle:\alpha<\gamma, p\in B_\alpha\}$ recursively as follows.
	
	Suppose we have defined $B_\gamma$ for all $\gamma<\alpha$. Consider $\bigcup_{\gamma\in A_\alpha}B_\gamma$. If this union is predense, then we let $B_\alpha=\emptyset$. Otherwise, choose a condition $p\in T$, sitting beyond level $\alpha$ of the Suslin tree, such that $p$ is incompatible with every element of that union. Let $B_\alpha=\{p\}$.
	
	If $G$ is a generic filter, then every club $C'\subseteq \omega_1$ in $V[G]$ contains a club $C\in V$. Hence, to show that $T\Vdash``\sigma \text{ is stationary}"$ we only need to show that for every club $C\in V$, the set $\bigcup_{\alpha\in C}B_\alpha$ is predense. Suppose for some club $C$ that is not the case. For stationarily many $\alpha$, we have that $C\cap \alpha=S_\alpha$ and hence the union we are looking at in defining $B_\alpha$ is $\bigcup_{\gamma\in A_\alpha}B_\gamma=\bigcup_{\gamma\in C\cap \alpha} B_\gamma$. Hence, the union is not predense, and $B_\alpha$ contains an element that is incompatible with every element of $\bigcup_{\gamma\in C\cap \alpha} B_\gamma$. But this is true for unboundedly many such $\alpha$, so this gives us an $\omega_1$ long sequence of pairwise incompatible conditions, i.e. an uncountable antichain. Since a Suslin tree is by definition c.c.c., this is a contradiction. Hence $T\forces ``\sigma \text{ is stationary}"$.
	
	But now let $g\in V$ be a filter. By extending it if necessary, without loss of generality we can assume $g$ is a maximal branch of the tree. Since $g\in V$, we know that $g$ is countable, so let the supremum of the heights of its elements be $\gamma$. Let $\alpha>\gamma$, and let $q\in g$. Since $B_\alpha$ is at most a singleton $\{p\}$ with $\text{ht}(p)\geq \alpha >\gamma > \text{ht}(q)$, and since $T$ is atomless, we know there is some $r\leq q$ with $r\forces \alpha \not \in \sigma$. Hence $q \not\forces \alpha \in \sigma$. Since this is true for all $q\in g$, it follows that $\alpha \not \in \sigma^{(g)}$. Hence far from being stationary, $\sigma^{(g)}$ is not even unbounded!
\end{proof} 

So (assuming the existence of Suslin trees) there are certainly some Suslin trees in which $\stat\BN^1$ fails. And with strong enough assumptions, we can show that $\stat\BN^1$ fails for every tree. So it's natural to ask:

\begin{question} 
Can we show in $\ZFC$ that $\stat\BN^1_{T,\omega_1}$ fails for every Suslin tree $T$? 
\end{question} 

Note that we can show the failure of $\ub\BN^1_{T,\omega_1}$ for any Suslin tree. Enumerate its level $\alpha$ elements as $\{p_{\alpha,n} : n\in \omega\}$. Now let 
$$ \sigma = \{(\check{\beta}, p_{\alpha,n}) : \alpha<\omega_1,n\in \omega, \beta=\omega.\alpha+n\}$$
Then $\sigma$ is forced to be unbounded but if $g\in V$ is such that $\sigma^g$ is unbounded, then $g$ defines an uncountable branch through $T$.

\subsubsection{Club shooting} 

The next lemma is a counterexample to the implication $\club\BFA_\kappa^\lambda$ $\Rightarrow$ $\club\BN_\kappa^\lambda$ in Figure \ref{diagram of implications - bounded with lambda<kappa}. 
It is open whether there is such a counterexample for complete Boolean algebras.

Suppose that $S$ is a stationary and co-stationary subset of $\omega_1$. 
Let $\PP_S$ denote the forcing that shoots a club through $S$. 
Its conditions are closed bounded subsets of $S$, ordered by end extension. 

\begin{lemma} \ 
\label{Lemma separating BFA from clubBN} 
\begin{enumerate-(1)} 
\item 
\label{Lemma separating BFA from clubBN 1} 
$\BFA^\omega_{\PP_S,\omega_1}$ holds. 
\item 
\label{Lemma separating BFA from clubBN 2} 
$\club\BN^1_{\PP_S,\omega_1}$ fails. 
\end{enumerate-(1)} 
In particular, for no $1\leq \lambda\leq \omega$ does $\BFA^\lambda_{\PP_S,\omega_1}$ imply $\club\BN^\lambda_{\PP_S,\omega_1}$. 
\end{lemma} 
\begin{proof} 
\ref{Lemma separating BFA from clubBN 1}: 
We claim that every maximal antichain $A\neq \{1_{\PP_S} \}$ is uncountable. 
(This shows that $\BFA^\omega_{\PP_S,\omega_1}$ holds vacuously.) 
To see this, suppose that $A$ is countable. 
Let $\alpha=\sup\{ \min(p)\mid p\in A\}$ and find some $\beta>\alpha$ in $S$. 
Then $q=\{\beta\}$ is incompatible with all $p\in A$, so $A$ cannot be maximal. 

\ref{Lemma separating BFA from clubBN 2}:  
$\sigma=\check{S}$ is $1$-bounded and $\PP_S \forces ``\sigma$ contains a club''.  
But for every filter $g$, $\sigma^g=S$ does not contain a club, since $S$ is co-stationary. 
\end{proof}

\section{Conclusion} 

The above results show that often, name principles are equivalent to forcing axioms. 
This provides an understanding of basic name principles $\NP_{\PP,\kappa}$ and of simultaneous name principles for $\Sigma_0$-formulas. 
For bounded names, the results provide new characterisations of the bounded forcing axioms $\BFA^\lambda$ for $\lambda\geq\kappa$. 
Name principles are closely related with generic absoluteness and can be used to reprove Bagaria's equivalence between bounded forcing axioms of the form $\BFA^\kappa$and generic absoluteness principles. 
Bagaria's result has been recently extended by Fuchs \cite{fuchs2021aronszajn}. 
He introduced a notion of $\Sigma^1_1(\kappa,\lambda)$-absoluteness for cardinals $\lambda\geq\kappa$ and proved that it is equivalent to $\BFA^\lambda_\kappa$. 
It remains to see if this can be derived from our results. 

Several problems about the unbounded name principle $\ub\FA_\kappa$ remain unclear. 
The results in Lemmas \ref{ubFA implies BFA} and \ref{ubFA implies FA for sigma-distributive forcings} about obtaining (bounded) forcing axioms from $\ub\FA_\kappa$ for forcings that do not add reals or ${<}\kappa$-sequences, respectively, hint at possible generalisations (see Question \ref{question ubFA BFA}). 
For forcings which add reals, we have that $\ub\FA_{\omega_1}$ is trivial for all $\sigma$-linked forcings and implies $\FA_{\omega_1}$ for random forcing. 
In all these cases, $\ub\FA_{\omega_1}$ and $\stat\FA_{\omega_1}$ are either both trivial or both equivalent to $\FA_{\omega_1}$. 
Can we separate $\ub\FA_{\omega_1}$ from $\stat\FA_{\omega_1}$ (See Question \ref{Question ubFA versus statFA})? Can $\ub\FA_{\omega_1}$ be nontrivial but not imply $\FA_{\omega_1}$? 
It remains to study other forcings adding reals and Baumgartner's forcing \cite[Section 3]{baumgartner1984applications} (see Question \ref{Question Baumgartner's forcing}). 

The stationary name principle $\stat\NP_{\omega_1}$ follows from the forcing axiom $\FA_{\omega_1}$ for some classes of forcings. 
For example, for the class of c.c.c. forcings both $\stat\NP_{\omega_1}$ and $\FA^+_{\omega_1}$ are equivalent to $\FA_{\omega_1}$ by results of Baumgartner (see Lemma \ref{Baumgartner's lemma}), Todor\v{c}evi\'c and Veli\v{c}kovi{\'c} \cite{todorcevic1987martin} (see Lemma \ref{characterisation of precaliber}). 
In general, $\FA^+$ goes beyond the forcing axiom, since being stationary is not first-order over $(\kappa,\in)$. 
For example, for the class of proper forcings,  $\PFA^+$ is strictly stronger that $\PFA$ by results of Beaudoin \cite[Corollary 3.2]{beaudoin1991proper} and Magidor (see \cite{shelah1987semiproper}). 
So $\FA^+$ and $\BFA^+$ do not fall in the scope of generic absoluteness principles, unless one artificially adds a predicate for the nonstationary ideal. 
Can one formulate $\PFA^+$ as a generic absoluteness or name principle for a logic beyond first order? 
Some questions remain about the weak variant $\stat\BN_{\PP,\omega_1}^1$ of $\stat\NP_{\omega_1}$. 
It is nontrivial for random forcing (see Lemma \ref{CH implies failure of statN(random)}) and for Suslin trees (see Corollary \ref{Suslin trees}). 
What is its relation with other principles? 
Does  $\stat\BN_{c.c.c.,\omega_1}^1$ imply $\MA_{\omega_1}$?


\bibliographystyle{plain}
\bibliography{references}

\begin{thebibliography}{10}

\bibitem{abraham1983forcing}
Uri Abraham and Saharon Shelah.
\newblock Forcing closed unbounded sets.
\newblock {\em The Journal of Symbolic Logic}, 48(3):643--657, 1983.

\bibitem{adolf2018singularizing}
Dominik Adolf, Arthur Apter, and Peter Koepke.
\newblock Singularizing successor cardinals by forcing.
\newblock {\em Proceedings of the American Mathematical Society},
  146(2):773--783, 2018.

\bibitem{bagaria1997characterization}
Joan Bagaria.
\newblock A characterization of {M}artin's axiom in terms of absoluteness.
\newblock {\em The Journal of Symbolic Logic}, 62(2):366--372, 1997.

\bibitem{bagaria2000bounded}
Joan Bagaria.
\newblock Bounded forcing axioms as principles of generic absoluteness.
\newblock {\em Archive for Mathematical Logic}, 39(6):393--401, 2000.

\bibitem{baumgartner1984applications}
James~E. Baumgartner.
\newblock Applications of the proper forcing axiom.
\newblock In {\em Handbook of set-theoretic topology}, pages 913--959.
  Elsevier, 1984.

\bibitem{beaudoin1991proper}
Robert Beaudoin.
\newblock The proper forcing axiom and stationary set reflection.
\newblock {\em Pacific journal of mathematics}, 149(1):13--24, 1991.

\bibitem{brendle2009forcing}
J{\"o}rg Brendle.
\newblock Forcing and the structure of the real line: {T}he {B}ogot{\'a}
  lectures.
\newblock {\em Available at
  \url{http://www.logic.univie.ac.at/~ykhomski/ST2013/bogotalecture.pdf}},
  2009.

\bibitem{dimontemottoros}
Vincenzo Dimonte and Luca~Motto Ros.
\newblock Generalized descriptive set theory at singular cardinals of countable
  cofinality.
\newblock In preparation.

\bibitem{foreman1988martin}
Matthew Foreman, Menachem Magidor, and Saharon Shelah.
\newblock Martin's maximum, saturated ideals, and non-regular ultrafilters.
  {P}art {I}.
\newblock {\em Annals of Mathematics}, pages 1--47, 1988.

\bibitem{fuchs2021aronszajn}
Gunter Fuchs.
\newblock Aronszajn tree preservation and bounded forcing axioms.
\newblock {\em The Journal of Symbolic Logic}, pages 1--16, 2021.

\bibitem{fuchs2018subcomplete}
Gunter Fuchs and Kaethe Minden.
\newblock Subcomplete forcing, trees, and generic absoluteness.
\newblock {\em The Journal of Symbolic Logic}, 83(3):1282--1305, 2018.

\bibitem{hamkins2012well}
Joel~David Hamkins and Daniel~Evan Seabold.
\newblock Well-founded {B}oolean ultrapowers as large cardinal embeddings.
\newblock {\em arXiv preprint arXiv:1206.6075}, 2012.

\bibitem{holy2019sufficient}
Peter Holy, Regula Krapf, and Philipp Schlicht.
\newblock Sufficient conditions for the forcing theorem, and turning proper
  classes into sets.
\newblock {\em Fundamenta Mathematicae}, 246:27--44, 2019.

\bibitem{jech2013set}
Thomas Jech.
\newblock {\em Set theory}.
\newblock Springer Science \& Business Media, 2013.

\bibitem{jensen2014subcomplete}
Ronald Jensen.
\newblock Subcomplete forcing and {$\mathcal{L}$}-forcing.
\newblock In {\em E-recursion, forcing and C*-algebras}, pages 83--182. World
  Scientific, 2014.

\bibitem{jensen2013k}
Ronald Jensen and John Steel.
\newblock K without the measurable.
\newblock {\em The Journal of Symbolic Logic}, 78(3):708--734, 2013.

\bibitem{kechris2012classical}
Alexander Kechris.
\newblock {\em Classical descriptive set theory}, volume 156.
\newblock Springer Science \& Business Media, 2012.

\bibitem{rinot2011jensen}
Assaf Rinot.
\newblock Jensen’s diamond principle and its relatives.
\newblock {\em Set theory and its applications}, 533:125--156, 2011.

\bibitem{sakaiMA}
Hiroshi Sakai.
\newblock Separation of ${MA}^+(\sigma\text{-closed})$ from stationary
  reflection principles.
\newblock Preprint, 2014.

\bibitem{sakai2015stationary}
Hiroshi Sakai and Boban Veli{\v{c}}kovi{\'c}.
\newblock Stationary reflection principles and two cardinal tree properties.
\newblock {\em Journal of the Institute of Mathematics of Jussieu},
  14(1):69--85, 2015.

\bibitem{sargsyan2012strength}
Grigor Sargsyan.
\newblock On the strength of {PFA} {I}.
\newblock {\em Available at \url{http://www.math.rutgers.edu/~gs481}}, 2012.

\bibitem{shelah1987semiproper}
Saharon Shelah.
\newblock Semiproper forcing axiom implies {M}artin's maximum but not {PFA}+.
\newblock {\em The Journal of Symbolic Logic}, 52(2):360--367, 1987.

\bibitem{todorvcevic2011forcing}
Stevo Todor{\v{c}}evi{\'c}.
\newblock Forcing with a coherent {S}uslin tree.
\newblock {\em Preprint}, 2011.

\bibitem{todorcevic1987martin}
Stevo Todor\v{c}evi\'c and Boban Veli\v{c}kovi{\'c}.
\newblock Martin's axiom and partitions.
\newblock {\em Compositio Mathematica}, 63(3):391--408, 1987.

\bibitem{trang2016pfa}
Nam Trang.
\newblock {PFA} and guessing models.
\newblock {\em Israel Journal of Mathematics}, 215(2):607--667, 2016.

\bibitem{velickovic1992forcing}
Boban Veli\v{c}kovi{\'c}.
\newblock Forcing axioms and stationary sets.
\newblock {\em Advances in Mathematics}, 94(2):256--284, 1992.

\bibitem{weiss1984versions}
William Weiss.
\newblock Versions of {M}artin's axiom.
\newblock In {\em Handbook of set-theoretic topology}, pages 827--886.
  Elsevier, 1984.

\end{thebibliography}

\end{document}